\documentclass{amsart}

\usepackage[margin=1in]{geometry}
\usepackage{amsmath, amssymb}
\usepackage[all,cmtip]{xy}
\usepackage{color,graphicx}
\usepackage{amsthm}
\usepackage{amsfonts}
\usepackage{amscd}
\usepackage{comment}
\usepackage{mathrsfs}
\usepackage[toc,page]{appendix}
\usepackage[utf8]{inputenc}
\usepackage[mathscr]{euscript}
\usepackage{enumitem}
\usepackage{xcolor}
\usepackage{todonotes}

\usepackage[bookmarks=true, bookmarksopen=true,%
bookmarksdepth=3,bookmarksopenlevel=2,%
colorlinks=true,%
linkcolor=blue,%
citecolor=blue,%
filecolor=blue,%
menucolor=blue,%
urlcolor=blue]{hyperref}
\usepackage{cleveref}

\def\CC{\mathbb{C}}

\def\FF{\mathbb{F}}
\def\GG{\mathbb{G}}

\def\LL{\mathbb{L}}

\def\QQ{\mathbb{Q}}
\def\RR{\mathbb{R}}
\def\R{\mathbb{R}}

\def\TT{\mathbb{T}}

\def\ZZ{\mathbb{Z}}

\def\calH{\mathcal{H}}

\def\calO{\mathcal{O}}
\def\cO{\mathcal{O}}
\def\calP{\mathcal{P}}

\def\bF{\mathbf{F}}

\def\bU{\mathbf{U}}
\def\bV{\mathbf{V}}

\def\bs{\mathbf{s}}

\newcommand\frH{\mathfrak{H}}

\newcommand\frc{\mathfrak{c}}

\newcommand\fre{\mathfrak{e}}

\newcommand\frg{\mathfrak{g}}

\newcommand\frk{\mathfrak{k}}
\newcommand\frl{\mathfrak{l}}

\newcommand\frn{\mathfrak{n}}
\newcommand\frp{\mathfrak{p}}
\newcommand\frq{\mathfrak{q}}
\newcommand\frr{\mathfrak{r}}

\newcommand\frt{\mathfrak{t}}

\newcommand\frz{\mathfrak{z}}

\newcommand\tilV{\widetilde{V}}

\newcommand\tilq{\widetilde{q}}

\newcommand{\usG}{\underline{\smash{G}}}

\newcommand\aff{\textup{aff}}
\newcommand\ellp{\textup{ell}}

\newcommand\ad{\textup{ad}}
\newcommand\Ad{\textup{Ad}}

\newcommand{\Bun}{\textup{Bun}}

\newcommand{\Corr}{\textup{Corr}}

\newcommand\Jac{\textup{Jac}}

\newcommand\Loc{\textup{Locsys}}

\newcommand\Perv{\textup{Perv}}
\newcommand{\Pic}{\textup{Pic}}

\newcommand{\QCoh}{\textup{QCoh}}
\newcommand{\IndCoh}{\textup{IndCoh}}

\newcommand{\Res}{\textup{Res}}

\newcommand\Spec{\textup{Spec}}

\newcommand{\Tr}{\textup{Tr}}

\newcommand{\Vect}{\textup{Vect}}

\newcommand{\Ker}{\textup{Ker}}
\newcommand{\Image}{\textup{Im}}

\newcommand{\Exp}{\textup{Exp}}

\newcommand\Aut{\textup{Aut}}
\newcommand\Hom{\textup{Hom}}

\newcommand\Mor{\textup{Mor}}

\newcommand\gl{\mathfrak{gl}}

\newcommand{\der}{\textup{der}}

\newcommand\nc{\newcommand}
\nc\on{\operatorname}
\nc\ol{\overline}
\nc\ul{\underline}
\nc\us[1]{\underline{\smash{#1}}}

\nc\oo{\infty}

\nc\Cone{\mathit{Cone}}
\nc\ssupp{\mathit{ss}}
\nc\risom{\stackrel{\sim}{\to}}
\nc\Sh{\mathit{Sh}}
\nc\un{\diamondsuit}
\nc\orient{\mathit{or}}
\nc\sing{\mathit{sing}}
\nc\MF{\on{MF}}
\nc\inthom{\mathit{Hom}}

\renewcommand{\t}{\mathfrak{t}}
\renewcommand{\S}{\mathbb{S}}

\newcommand{\x}{\times}

\newcommand{\C}{\mathbb{C}}
\newcommand{\Z}{\mathbb{Z}}
\newcommand{\N}{\mathcal{N}}
\newcommand{\g}{\mathfrak{g}}

\newcommand{\mfp}{\mathfrak{p}}
\newcommand{\mfl}{\mathfrak{l}}

\newcommand{\mfc}{\mathfrak{c}}

\newcommand{\modu}{\text{-mod}}

\newcommand{\B}{\textbf{B}}

\newcommand{\M}{\mathcal{M}}

\newcommand{\bfa}{\mathbf{a}}
\newcommand{\bfb}{\mathbf{b}}

\newcommand{\colim}{\textup{colim}}

\newcommand{\resp}{\text{resp. }}

\newcommand{\Stk}{\mathscr{S}\textup{tk}}
\newcommand{\Cat}{\mathscr{C}\textup{at}}
\newcommand{\Set}{\mathscr{S}\textup{et}}

\newcommand{\Cplx}{\mathscr{C}\textup{plx}}

\newcommand{\Ularrow}{\mathbin{\rotatebox[origin=c]{45}{$\Uparrow$}}}
\newcommand{\Urarrow}{\mathbin{\rotatebox[origin=c]{-45}{$\Uparrow$}}}

\newtheorem{thm}[equation]{Theorem}
\newtheorem{prop}[equation]{Proposition}
\newtheorem{lem}[equation]{Lemma}
\newtheorem{cor}[equation]{Corollary}

\theoremstyle{definition}
\newtheorem{defn}[equation]{Definition}

\newtheorem{eg}[equation]{Example}
\newtheorem{no}[equation]{Notation}
\newtheorem{rmk}[equation]{Remark}

\newtheorem{cons}[equation]{Construction}
\newtheorem{conj}[equation]{Conjecture}
\newtheorem{claim}[equation]{Claim}

\numberwithin{equation}{section}

\begin{document}

\nc{\Conn}{\mathrm{Conn}}

\nc{\fg}{\mathfrak g}
\nc{\fh}{\mathfrak h}

\nc{\cN}{\mathcal N}


\title[Uniformization of semistable bundles on elliptic curves]{Uniformization of semistable bundles\\ on elliptic curves}

\author{Penghui Li}

\author{David Nadler}
\address{YMSC\\Tsinghua University\\Beijing\\100084\\China}
\address{Department of Mathematics\\University of California, Berkeley\\Berkeley, CA  94720-3840}
\email{lipenghui@tsinghua.edu.cn}
\email{nadler@math.berkeley.edu}

\begin{abstract}
Let $G$ be a connected  reductive complex  algebraic group, and $E$ a complex elliptic curve. Let $G_E$ denote the  connected component of the trivial bundle in the stack of semistable $G$-bundles on $E$. We introduce a complex analytic uniformization of $G_E$ by adjoint quotients of reductive subgroups of the loop group of $G$. This can be viewed as a nonabelian version of the classical complex analytic uniformization $ E \simeq \mathbb{C}^*/q^{\mathbb{Z}}$. We similarly construct a complex analytic uniformization of $G$ itself via the exponential map, providing a nonabelian version of the standard isomorphism $\C^* \simeq \C/\Z$, and a complex analytic uniformization of $G_E$ generalizing the standard presentation $E \simeq \C/(\Z \oplus \Z \tau )$. Finally, we apply these results to the study of sheaves with nilpotent singular support. As an application to Betti geometric Langlands conjecture in genus 1, we define a functor from $Sh_\N(G_E)$ (the semistable part  of the automorphic category) to  $\IndCoh_{\check \N}(\Loc_{\check G} (E))$ (the spectral category).
\end{abstract}

\maketitle

\tableofcontents

\section{Introduction}

\subsection{Background}
Let $G$ be a connected complex reductive algebraic group, and $E$ a complex elliptic curve. 

The moduli of $G$-bundles on $E$ plays a distinguished role in representation theory, gauge theory and algebraic combinatorics (for example~\cite{BS,Sch} as the setting of the elliptic Hall algebra),
and its geometry has been the subject of a long and fruitful study. Atiyah~\cite{A} classified vector bundles on $E$ 
in terms of line bundles and their extensions. In particular, he showed rank $n$ vector bundles with trivial Jordan-Holder factors are in bijection with unipotent adjoint orbits in $GL(n)$, with the unique irreducible such vector bundle corresponding to the regular unipotent orbit. This initiated the organizing viewpoint  that  vector bundles on $E$ form an analogue of the adjoint quotient of $GL(n)$, where the ``eigenvalues" of a vector bundle are the line bundles appearing as its Jordan-Holder factors. In a beautiful series of papers, Friedman, Morgan and Witten~\cite{FMW1,FM1,FM2} extended this to any $G$, 
definitively describing the Jordan-Holder patterns 
and the geometry of the coarse moduli  of semistable bundles.  
Our focus here is the moduli stack of semistable bundles, and
specifically the construction of an analytic uniformization of   it by finite-dimensional subvarieties of the loop group of $G$.
We discuss  motivations and applications at the end of the introduction.

\subsubsection{Holomorphic loop group with twisted conjugation}
\label{introtwisted}
Thanks to complex function theory, the uniformization $E\simeq \CC^*/{q^{\ZZ}}$, with $|q|<1$, has been known since the 19th century.
Let $\Jac(E)$ be the Jacobian variety parameterizing degree zero line bundles on $E$. 
(Thanks to Serre's GAGA, one can equivalently consider algebraic or holomorphic  bundles.)
The Abel-Jacobi map $E\to \Jac(E)$, $x\mapsto \cO_E(x-x_0)$ is an isomorphism,
inducing a similar uniformization $\Jac(E)\simeq \CC^*/{q^{\ZZ}}$.

This  isomorphism also results from the following geometric observations. By the uniformization $E\simeq \CC^*/{q^{\ZZ}}$, holomorphic line bundles on $E$ are equivalent to equivariant holomorphic line bundles on $\CC^*$. Since every holomorphic line bundle on $\CC^*$ is trivializable, equivariant holomorphic line bundles are encoded by their equivariance up to gauge. Such data can be represented by elements of the holomorphic loop group $L_{\mathit{hol}}\CC^*$ up to $q$-twisted conjugacy. Within this identification,  one finds the uniformization of $\Jac(E)$ by the constant loops $\CC^* \subset L_{\mathit{hol}}\CC^*$ up to $q$-twisted conjugacy by the coweights
$\ZZ \simeq \Hom(\CC^*, \CC^*) \subset L_{\mathit{hol}}\CC^*$.

Now let $G_E:=\Bun_G^{ss,0}(E)$ denote the  connected component of the trivial bundle in the stack of semistable $G$-bundles on $E$.
By the uniformization $E\simeq \CC^*/{q^{\ZZ}}$,  isomorphism classes of $G$-bundles on $E$
are in bijection with $q$-twisted conjugacy classes in the holomorphic loop group $L_{\mathit{hol}}G$ (see for example~\cite{BG} where this is attributed
to Looijenga). We would like to enhance this to an analytic uniformization of $G_E$ by finite-dimensional subvarieties of  $L_{\mathit{hol}} G$. As a first attempt, we could take the constant loops $G\subset L_{\mathit{hol}} G$, but unfortunately, in general, the natural map $G/G\to G_E$ from the adjoint-quotient is neither surjective nor \'etale.
We will correct for both of these shortcomings by considering multiple charts together with their gluing; see the main results as described in Sect.~\ref{ss: intro main results}. 

\subsubsection{Connections on a circle with gauge transformation}
\label{introgauge}
Our arguments also apply to an easier situation to give a similar uniformization of $G/G$ in terms of (open subsets of) adjoint quotients of reductive Lie algebras. In this case, the role of $L_{hol}G$ with the action of twisted conjugation is replaced by the affine space of connections on a circle with the action of gauge transformation.

\subsection{Main results}\label{ss: intro main results}

Assume that $G$ is semisimple and simply-connected.  Let $T\subset G$ be a maximal torus, and denote by  $X_*(T) = \Hom(\CC^*, T)$ the coweight lattice.

The real affine space $\t_\R:= X_*(T) \otimes \R$ has a natural stratification by simplices coming from the hyperplanes 
$$H_{\alpha,n}:=\{ x \in \t_\R \;|\; \alpha(x)=n \}, \;\; \text{for } \alpha \text{ a root of } G, n \in \Z $$

Let $C$ be an alcove of $\t_\R$, i.e a top dimensional simplex.  There is a naturally defined category $\mathscr F_C$ of faces of $C$, whose objects are faces of $C$, i.e simplices in $\overline{C}$, and whose morphisms are given by the closure relation.

For any $J \in \mathscr{F}_C$, we have canonically associated finite-dimensional connected reductive subgroup $G_J\subset L_{hol} G$,
whose Lie algebra $\g_J \subset L_{hol} \g$ is spanned by $\t$ and those affine root spaces whose affine root vanishes on $J$. 

We introduce an analytic twisted adjoint-invariant open subset
 $\g^{se}_J \subset \g_J$ (\resp $G^{se}_J \subset G_J$)  of elements with ``small eigenvalues'' with respect to $J$. Roughly speaking, an eigenvalue in $\t$ (\resp $T$) is small with respect to $J$ if its real part (resp. $q$-part) lies in a simplex whose closure contains $J$ (for details, see definition before Proposition~\ref{seinet} for $G^{se}_J$, and Theorem~$\ref{covergroup} (6)$ for $\g^{se}_J$ ). Denote by $G_J/'G_J (\resp \g_J/'G_J)$ be the quotient stack w.r.t the twisted conjugation (Sect \ref{introtwisted}) (\resp the gauge action (Sect \ref{introgauge}, which we shall also refer to as a ``twisted" action)).

\begin{thm}[Theorem~$\ref{covergroup}(6)$, $\ref{colimitdiagram2}$]
\label{intro:colim}
There are isomorphisms of complex analytic stacks
$$(1)\xymatrix{
\colim_{J \in \mathscr{F}_C } 
\mathfrak{g}^{se}_J/'G_J  \ar[r]^-\sim &  G/G  
}
$$
$$(2)\xymatrix{
\colim_{J \in \mathscr{F}_C} 
G^{se}_J/'G_J  \ar[r]^-\sim & G_E  
}
$$
\end{thm}

One of our motivations for the above is to study $\oo$-categories of  complexes of sheaves with nilpotent singular support. To this end, we show by a propagation and an untwisting argument that:
\begin{prop}[Proposition~\ref{probagation}, \ref{untwisting}, \ref{group retraction}] For sheaves with nilpotent singular support, there are equivalences, compatible with the diagram $\mathscr{F}_C$:
$$
\xymatrix{
Sh_\N(\g_J/G_J) \ar[r]^-{\sim} & Sh_\N(\g_J/'G_J) \ar[r]^-\sim & Sh_\N(\g_J^{se}/'G_J)
}
$$
$$
\xymatrix{
Sh_\N(G_J/G_J) \ar[r]^-{\sim}  & Sh_\N(G_J/'G_J) \ar[r]^-\sim  & Sh_\N(G_J^{se}/'G_J)
}
$$
where $G_J/ G_J,\g_J/G_J$ are the quotient stacks by usual conjugations. 
\end{prop}
We deduce the following main result of the paper:
\begin{thm}[Theorem~\ref{groupcharactersheaf}]
\label{intro:lim}

There are equivalences of $\oo$-categories:
$$ (1)
\xymatrix{
Sh_\N (G/G)   \ar[r]^-\sim & \lim_{J \in\mathscr{F}_C^{op}} 
Sh_\N(\mathfrak{g}_J/ G_J)
}
$$
$$ (2)
\xymatrix{
Sh_\N (G_E)   \ar[r]^-\sim & \lim_{J \in\mathscr{F}_C^{op}} 
Sh_\N(G_J/G_J)
}
$$

\end{thm}

\begin{rmk}
	\label{remarkmainthm}
\begin{enumerate}
	\item In the limit, the arrow from $J$ to $J'$ is identified as parabolic restriction w.r.t the parabolic subalgebra $\frp^{J}_{J'}$ (\resp subgroup $P^J_{J'}$) defined by the relative position between $J$ and $J'$ (Definition~\ref{parabolicfacet}). Similarly, the (higher) commutativities are given by (higher) transitivity isomorphisms between parabolic restrictions.
	
	\item $Sh_{\mathcal{N}}(\mathfrak{g}/G)$ is by definition the category of character sheaves on $\g$. Fourier transform gives an equivalence $\TT: Sh_{\mathcal{N}}(\mathfrak{g}/K) \xrightarrow{\simeq} Sh(\N/G)$. The latter category is studied in the generalized Springer theory initiated in \cite{Lu7}. For characteristic 0 coefficients, $Sh(\N/G)$ is explicitly calculated in \cite{Rid,RR2}. For characteristic $p$ coefficients, the abelian category $\Perv(\N/K)$ is the subject of modular generalized Springer theory \cite{AHJR}.
	
	\item By \cite{MV}, $Sh_\N(G/G)$ agrees with the category of character sheaves introduced by Lusztig, which serves as a geometric avatar for the character theory of the finite group of Lie type $G(\FF_q)$.
	
	\item The group $G_J$ is explicitly known:  let $C$ be the standard alcove, and denote by $S_0$ the set of affine simple roots, then there is natural identification between $\mathscr{F}_C^{op}$ and $\mathscr{P}^{\circ}(S_0):=$ the category of proper subset of $S_0$, via $J \mapsto \{ \alpha \in S_0: \alpha(J)=0 \}$. Under this identification, $G_J$ is generated by the one parameter subgroup corresponding to the roots in $J$. Hence the Dynkin diagram of $G_J$ is $J$ viewed as a subdiagram of the affine Dynkin diagram of $G$. In particular, when $J$ is a vertex of the alcove, $G_J$ is isomorphic to either $G$ or a Pseudo-Levi subgroup of $G$ (a connected maximal rank reductive subgroup of $G$ that is not contained in any parabolic subgroup), and all Pseudo-Levi subgroups arise in this way, c.f Borel–de Siebenthal \cite{BDS}.

\item 

Recall that the conjugation actions in Theorem~\ref{intro:colim} are twisted (which depends on $q$). Nevertheless, in the last theorem, the conjugations are the usual (untwisted) ones. To achieve this, one needs to untwist all the conjugations compatibly with the diagram $\mathscr{F}_C^{op}$. (This essentially comes down to the fact that the simplices $J$'s are contractible, and the nilpotent cone in $\g_J/G_J$ (\resp  $G_J/G_J$) is constant along the direction of $J$.)
 Hence the right hand sides of Theorem~\ref{intro:lim} are completely Lie theoretic, and in particular, the right hand side in (2) is irrelevant to the elliptic curve $E$.

\item For a torus $T$, let $Loc(T/T)$ (\resp $Loc(T_E)$) be the $\oo$-category of local systems on the adjoint quotient
(\resp  on the degree zero component of $T$-bundles on $E$). 
Theorem~\ref{intro:lim} can be thought of as an analogy for a simple, simply-connected group, of the statement:
$$\xymatrix{ Loc(T/T) \ar[r]^-{\sim} & \lim_{BX_*(T)} Loc(\t/ T)} $$
$$\xymatrix{ Loc(T_E) \ar[r]^-{\sim} & \lim_{BX_*(T)} Loc(T/ T)} $$
where $BX_*(T)$ denotes the classifying space of the coweight lattice (viewed as an $\infty$-groupoid), the object in $BX_*(T)$ goes to $Loc(\t/T)$
(\resp $Loc(T/T)$), and all (higher) morphisms go to the identity.
Indeed, the first equivalence (and similarly the second one) can be seen as follows: a local systems on $T/T$ is a local system on $pt/T$ together with $X_*(T)$-action by monodromies, these data is equivalent to an object in $\lim_{BX_*(T)}(Loc(pt/T))$, which is equivalent to an object in $\lim_{BX_*(T)}(Loc(\frt/T))$, since $\frt$ is contractible.

Combining the statements for a simple, simply-connected group and a torus, one can obtain a general statement for any reductive group.
\end{enumerate}
\end{rmk}

\begin{eg}$(G=SL_2)$ Theorem~\ref{intro:lim} (1) gives:
  \[  Sh_\N(SL_2/SL_2) = 
\lim \{
\xymatrix{
 Sh_\mathcal{N}({\left[\begin{smallmatrix} 
\C & \C \\
 \C & \C 
 \end{smallmatrix}\right]}/\sim) \ar[r]_{\left[\begin{smallmatrix} 
* & * \\
 0 & * 
 \end{smallmatrix}\right]} &  Sh_{\mathcal{N}}({\left[\begin{smallmatrix} 
\C & 0 \\
0 & \C 
 \end{smallmatrix}\right]}/\sim)  & Sh_\mathcal{N}({\left[\begin{smallmatrix} 
\C & \C z \\
 \C z^{-1} & \C 
 \end{smallmatrix}\right]}/\sim) \ar[l]^{\left[\begin{smallmatrix} 
* & 0 \\
 * & * 
 \end{smallmatrix}\right]}
} \}
\]
where the matrices stand for the corresponding Lie subalgebras of $L\mathfrak{sl}_2$, and $/\sim$ is shorthand for taking the quotient by the corresponding adjoint action. The arrows ``$\rightarrow$" in the diagram are parabolic restrictions with respect to the indicated parabolic subalgebras. 
\end{eg}

\begin{eg} $(G=SL_3)$ 
Theorem~\ref{intro:lim} (1) gives

$Sh_\N(SL_3/SL_3)=\lim$
$$ 
\xymatrixrowsep{0.31in}
\xymatrixcolsep{-0.15in}
\xymatrix{
     &    &  {Sh_\N(\left[\begin{smallmatrix} 
\C & \C z & \C z\\
 \C z^{-1} & \C & \C \\
\C z^{-1} & \C & \C
 \end{smallmatrix}\right] /\sim)} \ar[dddl]_{\left[\begin{smallmatrix} 
* &  0 & 0 \\
* &  * & * \\
* & * & *
 \end{smallmatrix}\right]} \ar[dddr]^{\left[\begin{smallmatrix} 
* &  0 & * \\
 * &  * &  * \\
* & 0 & *
 \end{smallmatrix}\right]} \ar[dddd]|-{\left[\begin{smallmatrix} 
* &  0 & 0\\
 * & * & * \\
* & 0 & *
 \end{smallmatrix}\right]}&    &        \\
        &    &    &   &     \\
       &    &    &   &     \\
  & {Sh_\N(\left[\begin{smallmatrix} 
\C & 0  &0  \\
0  & \C & \C \\
0 & \C & \C
 \end{smallmatrix}\right] /\sim)}\ar[dr]|-{{\left[\begin{smallmatrix} 
* & 0 & 0\\
 0 & * & * \\
0 & 0 & *
 \end{smallmatrix}\right]}} \ar@{}[ru]|-{\Rightarrow} \ar@{}[dd]|-{\Downarrow}  &   & {Sh_\N(\left[\begin{smallmatrix} 
\C & 0 & \C z\\
0  & \C & 0 \\
\C z^{-1} & 0 & \C
 \end{smallmatrix}\right] /\sim)} \ar[dl]|-{\left[\begin{smallmatrix} 
* & 0 & 0 \\
0 & * & 0 \\
* & 0 & *
 \end{smallmatrix}\right]} \ar@{}[ul]|-{\Leftarrow} \ar@{}[dd]|-{\Downarrow} &     \\
     &    &  {Sh_\N(\left[\begin{smallmatrix} 
\C & 0 & 0\\
 0& \C & 0 \\
0 & 0 & \C
 \end{smallmatrix}\right] /\sim)} &    &        \\
     &    &    &    &        \\
{Sh_\N (\left[\begin{smallmatrix} 
\C& \C & \C\\
 \C & \C & \C \\
\C & \C & \C
 \end{smallmatrix}\right] /\sim)} \ar[uuur]^{\left[\begin{smallmatrix} 
* & * & * \\
 0 & * & * \\
0 & * & *
 \end{smallmatrix}\right]} \ar[rr]_{\left[\begin{smallmatrix} 
*& * & *\\
 * & * & * \\
0 & 0 & *
 \end{smallmatrix}\right]} \ar[uurr]|-{{\left[\begin{smallmatrix} 
*& * & *\\
 0 & * & * \\
0 & 0 & *
 \end{smallmatrix}\right]}}   &    &  {Sh_\N(\left[\begin{smallmatrix} 
\C & \C  & 0 \\
 \C  & \C & 0 \\
0 & 0& \C
 \end{smallmatrix}\right] /\sim)} \ar[uu]|-{\left[\begin{smallmatrix} 
*& *& 0\\
 0 &* & 0 \\
0 & 0 & *
 \end{smallmatrix}\right]}
\ar@{}[ul]|-{\Ularrow}  \ar@{}[ur]|-{\Urarrow}
  &      & {Sh_\N(\left[\begin{smallmatrix} 
\C & \C  & \C z\\
 \C  & \C & \C z\\
\C z^{-1} & \C z^{-1} & \C
 \end{smallmatrix}\right] /\sim)}  \ar[uuul]_{\left[\begin{smallmatrix} 
* & * & * \\
 0 & * & 0 \\
* & * & *
 \end{smallmatrix}\right]} \ar[ll]^{{\left[\begin{smallmatrix} 
* & * & 0\\
 * & * & 0 \\
* & * & *
 \end{smallmatrix}\right]}} \ar[uull]|-{\left[\begin{smallmatrix} 
* & * & 0 \\
 0 & * & 0 \\
* & * & *
 \end{smallmatrix}\right]}
}
$$
where the 2-arrows ``$\Rightarrow$" in the diagram are the transitivity natural isomorphisms between parabolic restrictions. 
\end{eg}

\begin{rmk}
The theorem is compatible with Springer theory in the sense that there is a commutative diagram:
$$
\xymatrix{ 
\C [W_{\aff}] \modu^\heartsuit \ar[d] \ar[r]^-\sim & \lim_{J \in \mathscr{F}_C} 
\C [W_J] \modu^\heartsuit  \ar[d]\\
\Perv_\N (G/G) \ar[d]  \ar[r]^-\sim & \lim_{J \in \mathscr{F}_C} 
\Perv_\N(\mathfrak{g}_J/G_J) \ar[d] \\
Sh_\N (G/G)   \ar[r]^-\sim & \lim_{J \in \mathscr{F}_C} 
Sh_\N(\mathfrak{g}_J/G_J)
}
$$
where $W_J$ is the Weyl group of $G_J$ (which equals the centralizer/stabilizer of $J$ in the affine Weyl group $W_{\aff}:= W \ltimes X_*(T)$); $A\modu^\heartsuit$ is the abelian category of $A$-modules. $\Perv(X)$ denotes the category of perverse sheaves on $X$; the first and second limit is taken inside $\Cat$ the category of  categories, and the last limit is taken inside $\Cat_\oo$ the category of $\oo$-categories. Note that the first isomorphism follows from the Coxeter presentation of $W_\aff$, hence Theorem~\ref{intro:lim} can be thought of as a Coxeter presentation of character sheaves. One can also upgrade the first line to an $\oo$-categorical statement, see \cite{Li2} for details.

We also define more general descent diagrams for general reductive groups. 
 See Section~
\ref{descentcategory} below for more details.
\end{rmk}

\subsection{Applications}

\subsubsection{Global Fourier/Radon transform}

There are natural integral transforms on the each term in the right hand side of Theorem~\ref{intro:lim}. Namely, the Fourier transform $\TT_J:Sh_\N(\g_J/G_J) \to Sh(\N_{\g_J}/G_J)$ in (1), the Radon transform $R_J: Sh_\N(G_J/G_J) \to Sh(B_J \backslash G_J/B_J)$ and the inverse Radon transform $\check{R}_J: Sh_\N(B_J \backslash G_J/B_J) \to Sh(G_J/_{ad}G_J) $ in (2), where $B_J := P^J_C$ is a Borel subgroup of $G_J$. The integral transforms are compatible with the diagram, hence pass to the (co)limit: 
$$ \TT: Sh_\N(G/G) \xrightarrow{\sim}  \lim_{J \in \mathscr F_C} Sh(\N_{\g_J}/G_J)  $$
$$ R: Sh_\N(G_E) \to \lim_{J \in \mathscr F_C} Sh(B_J\backslash G_J/B_J) $$
$$  \check{R}: \colim_{J \in \mathscr F_C}  Sh(B_J\backslash G_J/B_J) \to  \colim_{J \in \mathscr F_C} Sh_\N(G_J/G_J)  \simeq  Sh_\N(G_E)   $$
 We refer to these functors as the global Fourier/(inverse) Radon transform. Note that in the last row, we identify limit in $\oo$-categories with colimit in $\infty$-categories (with continuous functors), see \cite[Lemma 1.3.3]{Gai}. \\
 
 \paragraph{}\textit{Sheaf theoretic Kirillov orbit method.} The Kirillov orbit method is a heuristic method in representation theory that states every irreducible character of a Lie group should be given by the Fourier transform of an orbital distribution. The global Fourier transform  $\TT: Sh_\N(G/G) \xrightarrow{\sim}  \lim_{J \in \mathscr F_C} Sh(\N_{\g_J}/G_J)$ gives a sheaf theoretic realization of this heuristic: any character sheaf is given by a compatible system of (nilpotent) orbital sheaves. The appearance of the compatible  system indexed by $\mathscr{F}_C$ reflects the feature of sheaves: a sheaf is determined by its value on an open cover (indexed by $\mathscr{F}_C$ in this case) together with gluing, while an analytic function can be determined by its value on a single open subset. \\
 
 \paragraph{} \textit{Spectral description of character sheaves.}  The category $Sh(\N_{\g_J}/G_J)$ is described by the generalized Springer correspondence \cite{Lu7,RR2}. Based on that and the global Fourier transform,  the first author obtains the following spectral description of the category $Sh_\N(G/G)$ of character sheaves on $G$: denote $W_\aff^J:=N_{W_\aff} (W_J)/W_J$ , $\Lambda_J \subset W_\aff^J$ the subgroup of translations, $W^J:=W^J_\aff/\Lambda_J$, $\check S_J:=$ dual torus of $\Lambda_J \otimes \CC^*$ and $c_J$ the number of cuspidal sheaves on $\N_{\g_J}/G_J$. Denote $\mathcal{L}X:=X \times_{X \times X} X$ the derived loop space of $X$.
 Put $\widehat{G}:=\coprod_{J \in \mathscr{F}_C} ((\mathcal{L}\check S_J)/W^J)^{\coprod c_J}$, and $\QCoh(\widehat{G})$ the $\oo$-category of quasi-coherent sheaves on $\widehat{G}$.
\begin{thm}[\cite{Li2}]
	\label{charactersheavesspectral}
	  There is an equivalence of $\infty$-categories:
	$$Sh_\N(G/G) \simeq \QCoh(\widehat{G})$$
\end{thm}

 \paragraph{} \textit{Global Radon transform and Betti Geometric Langlands.} Let us first recall the Betti geometric Langlands conjecture proposed by Ben-Zvi--Nadler. Let $B$ be a Borel subgroup of $G$ and denote by $N$ its unipotent radical. Let $\check G$ be the Langlands dual group of $G$, 
 Let $\Sigma$ be a compact Riemann surface, and denote by $\Bun_G(\Sigma):=Map(\Sigma,BG)$ the derived moduli stack of holomorphic principal $G$-bundles on $\Sigma$, and by $\Loc_{\check G}(\Sigma):=Map(\Pi \Sigma,B \check G)$ the derived moduli stack of $\check G$ local system on $\Sigma$, where $\Pi \Sigma$ is the fundamental $\infty$-groupoid of $\Sigma$. Denote by $\IndCoh_{\check \N}(\Loc_{\check{G}}(\Sigma))$ the dg category of ind-coherent sheaves on $\Loc_{\check{G}}(\Sigma)$ with singular support in the global nilpotent cone $\check \N \subset T^{*,-1}\Loc_{\check{G}}(\Sigma)$ (c.f Arinkin-Gaitsgory \cite{AG}).

\begin{conj}[{\cite[Conjecture 1.5]{BZN16}}] 
	\label{bettigl}
	There is an equivalence of dg-categories:
	$$\LL_G: Sh_\N(\Bun_G(\Sigma)) \simeq  \IndCoh_{\check \N}(\Loc_{\check{G}}(\Sigma))$$
\end{conj}

\begin{rmk}
	We expect that Theorem~\ref{charactersheavesspectral} can be interpreted as semistable part of Conjecture~\ref{bettigl} for $\Sigma =$ nodal genus 1 curve. 
\end{rmk}	
We denote by $I$ the Iwahori subgroup corresponding to the alcove $C$, and by $N_J (\resp I_0)$ the unipotent radical of $B_J (\resp I)$. Denote by  $\check B$ the dual Borel subgroup of $\check G$, and by $\check N$ its unipotent radical. For a group $T$ acting on $X$,  denote by $Sh_{T \times T}(X) \subset Sh(X)$ the full subcategory of sheaves which are locally constant on $T$ orbits. Let $\Tr(A):=A \otimes_{A \otimes A^{op}} A$ be the trace of a monoidal category $A$. 

Now for $\Sigma=E$, we expect to define a functor:
$$\xymatrix{ \LL_{G}^{ss} :Sh_{\N}(G_E)  \ar[r]  & \IndCoh_{\check \N}(\Loc_{\check G} (E))  }$$ by the compositions of following functors:
\begin{enumerate}
	\item $Sh_{\N}(G_E) \simeq \lim Sh_\N(G_J/G_J)$. \\This is our main theorem. 
	
	\item $\lim Sh_{\N} (G_J/G_J) \simeq \colim \; Sh_{\N} (G_J/G_J) $. \\ Again, this is because we can identify limit in $\oo$-categories with colimit in $\infty$-categories (with continuous functors).
	
	\item (expected) $\colim \; Sh_{\N} (G_J/G_J)  \simeq  \colim \; \Tr(Sh_{T \times T}(N_J\backslash G_J/N_J))$. 
	\\ Denote by $Sh^u_{T \times T}(-) \subset Sh_{T \times T}(-)$ the full subcategory of sheaves with $T \times T$ acting by unipotent monodromies. Define the unipotent character sheaves $Sh_\N^u(G/G) \subset Sh_\N(G/G)$ to be the image of $Sh^u_{T \times T}(N\backslash G/N)$ under inverse Radon transform. This expected equivalence is known for unipotent monodromies by Ben-Zvi--Nadler \cite[Theorem 1.8]{BZN09}: $Sh_\N^u(G_J/G_J) \simeq \Tr (Sh^u_{T \times T}(N_J\backslash G_J/N_J))$ and it is expected for general monodromies \cite[Expectation 1.23]{BZN09}. 

	\item $\colim \; \Tr(Sh_{T \times T}(N_J\backslash G_J/N_J)) 
	\to \Tr(Sh_{T \times T}(I_0 \backslash LG/I_0))$. 
     \\This functor is induced by the monoidal functors $ Sh_{T \times T}(N_J\backslash G_J/N_J) \to \Tr(Sh_{T \times T}(I_0 \backslash LG/I_0))$.
	
	\item (expected) $\Tr(Sh_{T \times T}(I_0 \backslash LG/I_0)) 
	\simeq  \Tr(\IndCoh(\check B /\check B \times_{\check G/\check G} \check B /\check B ))$. \\
	This is known for unipotent monodromies by Bezrukavnikov \cite[Theorem 1 (2)]{Bez}:
	$Sh^u_{T \times T}(I_0 \backslash LG/I_0) \simeq \IndCoh^u(\check B /\check B \times_{\check G/\check G} \check B /\check B )$, where $\IndCoh^u(\check B /\check B \times_{\check G/\check G} \check B /\check B ) \subset \IndCoh(\check B /\check B \times_{\check G/\check G} \check B /\check B )$ is the full subcategory of objects supported on $\check N /\check N \times_{\check G/\check G} \check N /\check N$.  It is expected for general monodromies  \cite[Conjecture 58]{Bez}.
	
	\item  $\Tr(\IndCoh(\check B /\check B \times_{\check G/\check G} \check B /\check B )) \simeq \IndCoh_{\check \N}(\Loc_{\check G} (E)) $. 
	\\ This is proved by Ben-Zvi--Nadler--Preygel \cite[Theorem 4.4]{BZNP}. Denote by $\Loc^u_{\check G} (E) \subset \Loc_{\check G} (E)$ be the substack of local systems whose first monodromies are unipotent, and by \\
	${\IndCoh^u_{\check \N}(\Loc_{\check G} (E)) \subset \IndCoh_{\check \N}(\Loc_{\check G} (E))} $
	the full subcategory of sheaves supported on $\Loc^u_{\check G} (E) $. Inside the equivalence above, we have the unipotent version $\Tr(\IndCoh^u(\check B /\check B \times_{\check G/\check G} \check B /\check B )) \simeq \IndCoh^u_{\check \N}(\Loc_{\check G} (E))$.

\end{enumerate}

In particular, denote by $Sh^u_{\N}(G_E) :=\lim Sh^u_\N(G_J/ G_J) \subset Sh_\N(G_E)$ the full subcategory corresponding to unipotent character sheaves, then (3) and (5) above are known for unipotent monodromies, so we have defined a functor 
$\LL_G^{ss,u}: Sh^u_{\N}(G_E) \to  \IndCoh_{\check \N}(\Loc_{\check G} (E)^u)$.
Under Conjecture~\ref{bettigl}, we expect the following diagram to commute:
$$\xymatrix{  Sh^u_\N(G_E) \ar[r]^-{\LL^{ss,u}_G} \ar@{^{(}->}[d] &  \IndCoh_{\check \N}(\Loc_{\check G} (E)^u) \ar@{^{(}->}[d] \\
	Sh_\N(G_E) \ar[r]^-{\LL^{ss}_G} \ar@{^{(}->}[d]^{j_!} &  \IndCoh_{\check \N}(\Loc_{\check G} (E)) \ar@{=}[d] \\
	Sh_\N(\Bun_G(E))   \ar[r]^-{\LL_G}_-{\sim}  &    \IndCoh_{\check \N}(\Loc_{\check G} (E))
}$$
for $j: G_E \to \Bun_G(E)$ the open embedding. Hence it is natural to expect:

\begin{claim}
 $\LL_G^{ss,(u)}$ is fully-faithful.
\end{claim}	
Assuming the expected equivalences (3) and (5) above, this claim is equivalent to a statement about affine Hecke categories:
\begin{claim}
	\label{fullyfaithful}
The natural functor
$$\xymatrix{\colim_{J \in \mathscr{F}_C}  \Tr(Sh_{T \times T}(N_J\backslash G_J/N_J)) \ar[r] &
 \Tr(Sh_{T \times T}(I_0 \backslash LG/I_0))} $$
is fully-faithful.
\end{claim}	

We shall prove these claims in a future paper. An analogous statement to Claim~\ref{fullyfaithful} for Weyl group has been proved in \cite{Li19}.

\begin{rmk}
	The word global Fourier/Radon transform refers to being global on the moduli stack $\Bun^{(ss)}_G(\Sigma)$, rather than on the Rieman surface $\Sigma$. Nevertheless, we expect that localizing on moduli stack is related to taking nearby cycles along degenerations of Riemann surface. This fits into the general framework of \cite[Conjecture 4.15]{BZN16} where they proposed an approach to Conjecture~\ref{bettigl} via gluing (i.e taking (co)limits) from degenerations. 
	
\end{rmk}

\subsubsection{Topological nature of $Sh_\N(G_E)$}

We can also define the uniformizations universally over the moduli of elliptic curve $\M_{1,1}$, this gives a natural notion of parallel transport:

\begin{prop}[Corollary~\ref{locallyconstant}]
The $\oo$-category $Sh_\N(G_E)$ of complexes of sheaves with nilpotent singular support is locally constant over the moduli space of elliptic curves $\mathcal{M}_{1,1} $.
\end{prop}
\begin{rmk}
		
 As in Remark~\ref{remarkmainthm} (3), the right hand side in Theorem~\ref{intro:lim} (2) is irrelevant to the elliptic curve $E$. However, the equivalence there  depends on a choice of basis in $H^1(E,\mathbb{Z})$ (and a point in $E$). Hence Theorem~\ref{intro:lim} (2) does NOT imply $Sh_\N(G_E)$ is constant over $\mathcal{M}_{1,1} $.  It is only constant after making the choice of basis (i.e. after a base change to  the upper half plane $\mathcal{H}$). And in fact the resulting sheaf of categories on $\mathcal{M}_{1,1}$ has interesting monodromy. For $G=SL_n$, this sheaf contains the monodromy of the $SL(2,\Z)$ action on $E[n]/S_n$, by considering the cuspidal objects, where $E[n]$ is the set of $n$-torsion points of $E$.
\end{rmk}

With modest further effort, and similar  applications of the above results, one can extend the corollary 
to  the $\oo$-category $Sh_\cN(\Bun_G(E))$ of complexes of sheaves with nilpotent singular support on the entire moduli
of all $G$-bundles on $E$.
This category contains the Hecke eigensheaves of the geometric Langlands program, and we expect it to offer also a theory of affine character sheaves. Furthermore, under Langlands duality/mirror symmetry, it is expected to correspond to a derived category of coherent sheaves on the commuting stack. (Note that the commuting stack, and hence its coherent sheaves as well, is evidently a topological invariant, only depending on the fundamental group of the elliptic curve.) This is in turn the subject of beautiful recent developments
(Schiffmann-Vasserot \cite{SV1,SV2,SV3} on Macdonald polynomials and double affine Hecke algebras;
Ginzburg \cite{Gi4} on Cherednik algebras and the Harish Chandra system)
and in particular its role as affine character sheaves was established in~\cite{BZNP}.

\subsubsection{Dependence of restriction functor on parabolic subgroups}

During the proof of our main theorem, we also obtain the following result:

\begin{cor}
Let $P_1,P_2 \subset G$ be two parabolic subgroups of a $G$ with the same Levi $L \subset G$. Then there is a (non-canonical) natural isomorphism between the parabolic restrictions $$\xymatrix{\Res_{P_1} \simeq \Res_{P_2}:Sh_\N(G/G) \ar[r] & Sh_\N(L/L)  } $$
\end{cor}

Such statements has been proved for orbital sheaves on Lie algebras in \cite{Mi}, and for perverse character sheaves on Lie groups in \cite{Gi2}. During the proof of our main theorem, we define a restriction functor $R_U: Sh_\N(G/G) \to Sh_\N(L/L)$ depending on a choice of retractable subset $U \subset L$. The idea is that both parabolic restriction functors are isomorphic to the unique extension of the pullback along $U/L \to G/G$. Hence each choice of $U$ gives such a natural isomorphism. In fact, the space of choices of such $U$ is connected but not contractible. This is explained in detail in Section~\ref{sectiondependence}.

\subsection{Outline of the argument in an example} 
\label{firsteg}
To illustrate the ideas, we give a first example in its most plain form.

Let $G=SL_2$, $\g=\mathfrak{sl}_2$, $T,\t$ the diagonal matrices in $G,\g$. Let $U:=\{ X \in \g: |Re(\lambda(X))|<1/2    \}$, for $\lambda(X)$ an eigenvalue of $X$. Let $V$ be another copy of $U$. We have 
$U \to G$ by $X \mapsto \exp(2\pi i X)$ and $V \to G$ by $Y \mapsto 
\begin{pmatrix} -1 & 0 \\
                           0  &  -1
                       \end{pmatrix} 
\exp(2 \pi i Y)$. Let $D=\{ H \in \t: 0< \lambda_1(H) < 1/2\}$, where $\lambda_1(H)$ is the first eigenvalue of $H$. We have $D \x G/T \to U$ by $(H,g) \mapsto gHg^{-1}$ and $D \x G/T \to V$ by 
$(H,g) \mapsto g(H- 
\begin{pmatrix} 1/2 & 0 \\
                           0  &  -1/2
                       \end{pmatrix} ) g^{-1}$. 
Notice that $[H,\begin{pmatrix} 1/2 & 0 \\
                           0  &  -1/2
                       \end{pmatrix}]=0.$
The commutative diagram
\[\xymatrix{  &  D \x G/T \ar[dl]\ar[dr] \ar@{}|{\square}[dd] &\\
U  \ar[dr]  &            &    V \ar[dl]\\
&     G   &
}
\]
 is cartesian. All arrows are open embeddings and $U \coprod V \to G$ is surjective.
The diagram is $G$-equivariant, and passing to the quotient, we have 
\begin{equation}
\label{example}
\xymatrix{  &  D/T \ar[dl]\ar[dr] \ar@{}|{\square}[dd] &\\
U/G  \ar[dr]  &            &    V/G \ar[dl]\\
&     G/G   &
}
\end{equation}
with all actions being adjoint actions. Passing to sheaves, the pullback functors preserve nilpotent singular supports. Hence we have a

$$Sh_\N(G/G) \simeq
\lim{ 
\left(	
\xymatrixcolsep{-0.05in}
\vcenter{\xymatrix{
 &  Sh_{\mathcal{N}}(D/T) &  \\
Sh_\mathcal{N}(U/G) \ar[ru] & & Sh_\mathcal{N}(V/G) \ar[lu]
}}
\right) } $$
The singular support condition allows us to deduce that $Sh_\N(\g/G) \xrightarrow{j^*} Sh_\N(U/G)$ is an equivalence, for $j:U \to \frg$ the open embedding. Similarly we have $Sh_\N(\frt/T) \simeq Sh_\N(D/T)$ and $Sh_\N(\g/G) \simeq Sh_\N(V/G)$. Hence we get

$$Sh_\N(G/G) \simeq
\lim{ 
	\left(	
	\xymatrixcolsep{-0.05in}
	\vcenter{\xymatrix{
			&  Sh_{\mathcal{N}}(\frt/T) &  \\
			Sh_\mathcal{N}(\g/G) \ar[ru] & & Sh_\mathcal{N}(\g/G) \ar[lu]
	}}
	\right) } $$

Moreover the functors on the right hand side can be identified with parabolic restriction (see Section~\ref{propagation_untwisting} for a detailed discussion). Hence this gives a description of the category of character sheaves on a Lie group in terms of categories of character sheaves on Lie algebras. 
It turns out that the charts $U/G,V/G$ and $D/T$ above appear naturally in side the infinite dimensional gauge uniformizations.

\subsection{Acknowledgements} We are indebted to David Ben-Zvi, Roman Bezrukavnikov, Dragos Fratila,  Sam Gunningham, Quoc Ho, Jacob Lurie, Tao Su and Zhiwei Yun for  helpful discussions.  We would like to thank the organizers of the workshop ``Geometric Langlands and derived algebraic geometry" at CIRM Luminy, where a preliminary version of the results was presented. 
We thank an anonymous referee for many helpful comments on this paper.
PL is grateful for the support of the Advanced Grant “Arithmetic and Physics of Higgs moduli spaces” No. 320593 of the European Research Council.
DN is grateful for the support of NSF grant DMS-1502178.

\section{Preliminaries on category theory}

\subsection{The Grothendieck construction and Kan extensions}

References for the Grothendieck construction (also referred to as unstraightening functor) are \cite[Sect 3.2]{Lur}, \cite[I.1.1.4]{GR}. We shall only define it in the context we need.
By a $(2,1)$-category, we mean a strict $2$-category, such that all $2$-morphisms are invertible. And a $2$-groupoid is a strict $2$-category, such that all $1$ and $2$-morphisms are invertible. We can view a $(2,1)$-category as an $\oo$-category and a $2$-groupoid as an $\oo$-groupoid. For any $\oo$-category $\mathscr{C}$, denote by $\mathscr{C}^{\triangleright}$ the $\oo$-category obtained by add one final object $*$ to $\mathscr{C}$. For a functor $F:\mathscr{C}^{\triangleright} \to \mathscr{T}$ between $\oo$-categories, we say $F$ is a colimit diagram if the natural map $\colim_\mathscr{C} F \to  F(*)$ is an isomorphism in $\mathscr{T}$. 

\begin{defn} Let $\mathscr{C}$ be an ordinary category, and $\bF: \mathscr{C} \to \mathscr{S}_{ \leq 2} $ a strict functor to the category $\mathscr{S}_{\leq 2}$ of small $2$-groupoids.
The \textit{Grothendieck construction} $\int^{\mathscr{C}} \bF$ is a $(2,1)$-category with 
\begin{itemize}
	\item objects: $(x,E),$ for $x \in \mathscr{C}$ and $E \in \bF(x)$;
	\item 1-morphisms: $(f,a): (x,E) \to (y,F)$, for $f: x \to y$ a $1$-morphism in $\mathscr{C}$, and $a: \bF(f)(E) \to F$ a morphism in $\bF(y)$;
	\item 2-morphisms: $\eta_\alpha: (f,a) \Rightarrow (f,b)$ for $\alpha: a \Rightarrow b $ a $2$-morphism in $\bF(y)$.
\end{itemize}	
 There is a natural functor $p: \int^{\mathscr{C}} \bF \to \mathscr{C}$, via $(x,E) \mapsto x, (f,a) \mapsto f$, and $\eta_\alpha \mapsto id$.
\end{defn}

\begin{defn}
Given a functor $\rho: \mathscr{A} \to \mathscr{B}$ a functor between $\oo$-categories and an $\oo$-category $\mathscr{T}$, denote $\rho^*:[\mathscr{B},\mathscr{T}] \to [\mathscr{A},\mathscr{T}]$ the induced functor. The \textit{left Kan extension} $\rho_!$ is the left adjoint to $\rho^*$.

\end{defn}

We collect some basic properties of Kan extensions:

\begin{prop} Let $K:\mathscr{A} \to  \mathscr{T}$ be a functor,
\begin{enumerate}
\item Let $\pi: \mathscr{A} \to pt$ (the one point category), 
then $\pi_! \mathscr{A}(pt) \simeq \colim_\mathscr{A} K$,
 provided either side of the equation exist.
\item Let $\rho: \mathscr{A} \to \mathscr{B}, \varphi: \mathscr{B} \to \mathscr{C}$, and assume that $\rho_{!} (K) , \varphi_{!}(\rho_{!}(K))$ exist, then $(\varphi \circ \rho)_{!}(K)$ exist and $(\varphi \circ \rho)_{!}(K)\simeq \varphi_{!}(\rho_{!}(K))$.
\item Let $\rho: \mathscr{A} \to \mathscr{B}$, then $\colim_{\mathscr{A}} K \simeq \colim_{\mathscr{B}} \rho_!(K)$.
\end{enumerate}
\end{prop}
\begin{proof}
(1) follows from the definition of colimit. (2) follows from the fact that $\rho^* \circ \varphi^* \simeq (\varphi \circ \rho)^*$ and that adjoints are canonical. (3) follows from (1) and (2), by taking $\mathscr{C}=pt$ in (2).
\end{proof}

The following statement can be found in \cite[I.1.2.2.4]{GR}.
\begin{prop}
\label{computeKan}
Let $p : \int^\mathscr{C} \bF  \to \mathscr{C}$ the natural map, and $K \in [\int^\mathscr{C} \bF, \mathscr{T}]$.

Then $p_!(K)(x)= \textup{colim}_{\bF(x)} K$, provided that the colimits exist.

\end{prop}

\subsection{The index category $\int^{\Delta^{op}} {S^\bullet_{\Gamma}}$}
\label{descentcategory}

Denote by $\Delta$ the simplex category, it is the category of finite non-empty linearly ordered sets. Denote by $[n]$ the linear ordered set $\{0 \to 1 \to ... \to n \} \in \Delta.$

\begin{defn}
Let $S$ be a set, $\Gamma$ a group acting on $S$. 
\begin{enumerate}
\item $S // \Gamma $ denotes the set of orbits.	
\item  $S / \Gamma$ denotes the quotient groupoid.	
\item The $(2,1)$-category $S_{\Gamma}$ has
\begin{itemize}
	\item object $s$ for every $s \in S$;
	\item morphism $\gamma: s \to t$ for every $\gamma \in \Gamma, s,t \in S$, such that $\gamma(s)=t$;
	\item 2-morphism $\gamma'\gamma^{-1}: \gamma \Rightarrow \gamma'$, for every $\gamma,\gamma': s \to t$. 
\end{itemize}
The identity and composition are given by the obvious ones. We see that $(-)_{\Gamma}$ defines a functor $\Gamma$-$\Set \to \mathscr{S}_{\leq 2}.$ 

\item Let $S^{\bullet}: \Delta^{op,\triangleright} \to  \Gamma$-$\Set$ be the functor $[n] \mapsto S^{n+1}$ with diagonal $\Gamma$-action, where $S^0:=pt$ the one point set. And define $S^{\bullet}_{\Gamma}:=(-)_{\Gamma} \circ S^\bullet:   \Delta^{op,\triangleright} \to    \mathscr{S}_{\leq 2}$.

\end{enumerate}
\end{defn}

\begin{prop}
The natural functor $S_{\Gamma} \to S//\Gamma$ is an equivalence of $(2,1)$-categories, where we view a set as a $(2,1)$-category with trivial $1$ and $2$-morphisms.
\end{prop}
\begin{proof}
	This functor is clearly essentially surjective. Let $s,t \in S$, we need to show the map $\Hom_{S_\Gamma}(s,t) \to \Hom_{S//\Gamma}(s,t)$ is an equivalence of ordinary categories. If there is no  $\gamma \in \Gamma$, such that $\gamma(s)=t$, then both categories are empty. Otherwise, $\Hom_{S//\Gamma}(s,t)$ is the singleton category $\{*\}$ and $\Hom_{S_\Gamma}(s,t)$ consist of object $\gamma \in \Gamma$, with $\gamma(s)=t$, and any two objects $\gamma, \gamma'$ are isomorphic by the arrow $\gamma'\gamma^{-1}$. Hence $\Hom_{S_\Gamma}(s,t)$ is also isomorphic to $\{*\}$.
\end{proof}

\begin{rmk} 
\begin{enumerate}
\item The indexing $(2,1)$-category $S_{\Gamma}$ is natural since the stacks we use are $1$-truncated (e.g. $\g/G, G/G$ and $\Bun_G(C)$), and such stacks form a $(2,1)$-category inside $\Stk$. 
 
\item The category $\int^{\Delta^{op, \triangleright}} S^{\bullet}//\Gamma $ has a final object, denote by $pt$. And $\int^{\Delta^{op, \triangleright}} S^{\bullet}//\Gamma \simeq (\int^{\Delta^{op}} S^{\bullet}//\Gamma)^\triangleright$.
\end{enumerate}

\end{rmk}

\begin{cons} 
\label{egxreg}
Let $X$ be a set with an action of a discrete group $\Gamma$. 
Choose a $\Gamma$-set $S$, and subset $V_s \subset X,$ for each $s \in S$, such that the collection of subsets $\{ V_s: s \in S \}$ are $\Gamma$-invariant, i.e, $\gamma(V_s)=V_{\gamma s}$, for any $\gamma \in \Gamma, s \in S$. For any $\bs=(s_1,s_2,...,s_k) \in S^k$, put $|\bs|:=\{s_1,s_2,...,s_k\} \subset S$, denote by $V_\bs:= \bigcap_{s \in |\bs|} V_s$, $\Gamma_\bs$ the stabilizer of $\Gamma$ at $\bs$ (for the diagonal action).
Then we define a functor 
$$\xymatrix{\bV:\int^{\Delta^{op,\triangleright}} S^{\bullet}_{\Gamma}  \ar[r] & \mathscr{S}} $$

\begin{enumerate}
\item $\bV(\bs):=V_\bs/\Gamma_\bs$, where $\bV(pt):=X/\Gamma$;
\item $\bV(\gamma):=Ac_\gamma: V_\bs/\Gamma_\bs \to V_{\gamma{\bs}}/\Gamma_{\gamma \bs}$ the action by $\gamma$; 
\item for $2$-morphism $\gamma' \gamma^{-1}: \gamma \Rightarrow \gamma': \bs \to \mathbf{t}$. Define $\bV(\gamma' \gamma^{-1}):= \eta_{\gamma'\gamma^{-1}} \circ Ac_\gamma  : Ac_\gamma \Rightarrow Ac_{\gamma'}$, where $\eta_{\gamma' \gamma^{-1}} : Id_{V_\mathbf{t}} \Rightarrow Ac_{\gamma'\gamma^{-1}}: V_\mathbf{t}/\Gamma_\mathbf{t} \to V_\mathbf{t}/\Gamma_\mathbf{t}$  is the canonical trivialization of the action of $\gamma'\gamma^{-1} \in \Gamma_{\mathbf{t}}$ as inner automorphism.
\item for $\delta:\bs \to \bs'$, then $V_\bs \subset V'_{\bs'}$ and $\Gamma_\bs \subset \Gamma_{\bs'}$, this gives $\bU(\delta):=V_\bs/\Gamma_\bs \to V_{\bs'}/\Gamma_{\bs'}$
\end{enumerate}
Let$ \xymatrix{\bV_0:\int^{\Delta^{op,\triangleright}} S^{\bullet}_{\Gamma}  \ar[r] & \mathscr{S}}$ be the functor by same formula above except we replace / by //. Hence $\bV_0$ is isomorphic to the composition $\xymatrix{\int^{\Delta^{op,\triangleright}} S^{\bullet}_{\Gamma} \ar[r]^-{\bV} & \mathscr{S} \ar[r]^-{\pi_0} & \Set \subset \mathscr{S} } $
 \end{cons}
 
 \begin{prop}
 \label{xcolimit}
Assume that $\bigcup_{s \in S} V_s= X$, then the natural morphisms in $\mathscr{S}$ is an equivalence:
$$\xymatrix{\colim_{\int^{\Delta^{op}} S^{\bullet}_{\Gamma} } \bV
 \ar[r]^-{\sim} & \bV(*)= X/\Gamma}.$$
Assume further that for any $x \in V_s$, the stabilizers satisfies $\Gamma_x \subset \Gamma_s$, then the natural morphism is an equivalence:
$$\xymatrix{\colim_{\int^{\Delta^{op}} S^{\bullet}_{\Gamma} } \bV_0
	\ar[r]^-{\sim} & \bV_0(*)= X// \Gamma}.$$

 \end{prop} 
\begin{proof}
Denote $p: \int^{\Delta^{op}} S^\bullet_{\Gamma} \to \Delta^{op} $ the natural map. By Proposition~\ref{computeKan}, we have $p_!(\bV)([n]) = \colim_{S^n_{\Gamma}} V_\bs/\Gamma_\bs \simeq \coprod_{[\bs] \in S^n//\Gamma, \text{ fix } \bs \text{ a lift of } [\bs]} V_\bs/\Gamma_\bs \simeq (\coprod_{\bs \in S^n} V_\bs)/\Gamma$. 

Hence $p_!\bV$ is isomorphic to the functor $[n] \mapsto (\coprod_{s \in S} V_s)^n_X /\Gamma$, where for any map $Y \to X$, denote by $Y^n_X:= Y \times_X  ... \times_X Y$ the $n$-fold fiber product.
Therefore $\colim_{\int^{\Delta^{op}} S^\bullet_{\Gamma} } \bV \simeq  \colim_{\Delta^{op}} p_!\bV \simeq  X/\Gamma$.
For the second statement, note that under the additional assumption, the functor $\bV$ takes $1$-morphisms to fully-faithful morphisms in $\mathscr{S}$, hence the second equivalence follows from the first one.

\end{proof}

\begin{rmk}
 The upshot of this construction is that the charts $V_\bs/\Gamma_\bs \to X/\Gamma$ are usually non-Galois. The category $\int^{\Delta^{op}} S^{\bullet}_{\Gamma}$ gives a way to organizes these non-Galois charts. The main example we have in mind is when $X=\t$ and $\Gamma = W_{\aff}$. 
 
\end{rmk}

\subsection{The $\oo$-category of correspondence}
The reference is \cite[Chapter 7]{GR}. We shall use the $\oo$-category of correspondences to organize the higher morphisms between various restriction functors in our main theorem. Morally speaking, the advance of using correspondences is that the higher coherence follows from cartesian property of certain squares.

Let $\mathscr{C}$ be an $\infty$-category. 
\begin{defn}
	Define $\Corr(\mathscr{C})$ the $\infty$-category via $\Mor_{\Cat_{\oo}}([n], \Corr({\mathscr{C})}):= \textup{Grid}^{\geq \textup{dgnl}}_n(\mathscr{C}), $ where  $\textup{Grid}^{\geq \textup{dgnl}}_n(\mathscr{C}) \subset  \textup{Mor}_{\Cat_{\oo}}(([n] \times [n]^{op})^{\geq \textup{dgnl}},\mathscr{C})$ is the full $\oo$-subgroupoid consist of commutative diagram $\underline{c}$:
	\[\xymatrix{ c_{0,n}  \ar[r] \ar[d] &  c_{0,n-1}  \ar[r] \ar[d]  & ...  \ar[r] \ar[d] &  c_{0,1} \ar[r] \ar[d] &  c_{0,0}  \\
		c_{1,n}   \ar[r]  \ar[d]  & c_{1,n-1}  \ar[r] \ar[d]  & ...  \ar[r] \ar[d] &  c_{1,1}     \\
		... \ar[r]   \ar[d]      &     ... \ar[r] \ar[d]   &  ...   \\
		c_{n-1,n}   \ar[r] \ar[d]  &   c_{n-1,n-1}  \\
		c_{n,n}
	}
	\]
	such that all squares are cartesian.
\end{defn}

\begin{rmk}
	$\Corr(\mathscr{C})$ can be defined as an $(\infty,2)$-category. However, for our purpose, we only viewed $\Corr(\mathscr{C})$ as an $\infty$-category as defined above. 
\end{rmk}

There is an canonical equivalence $can:\textup{Corr}(\mathscr{C})^{op} \simeq \textup{Corr}(\mathscr{C})$, by switching the source and target for a correspondence (i.e flipping along the main diagonal of the above diagram). For any functor $F:\mathscr{I} \to \textup{Corr}(\mathscr{C})$, denote by 
\begin{equation}
\label{conjugatefunctor}
\overline{F}:= can \circ F^{op}: \mathscr{I}^{op} \to \textup{Corr}(\mathscr{C}).
\end{equation}

\subsubsection{Functors into the category of correspondences}

Assume that $\mathscr{C}$ is an $\oo$-category with a final object $*$. Then we can view $\Mor_{\Cat_{\oo}}([n],\Corr(\mathscr{C}))$ as a full subgroupoid of $\Mor_{\Cat_{\oo}}([n] \times [n]^{op},\mathscr{C})$, by sending all entries below the off-diagonal to $*$. For any $\oo$-category $\mathscr{I}$, write $\mathscr{I}= \colim_{[i]/\mathscr{I}} [i]$, then we can view $\Mor_{\Cat_{\oo}}(\mathscr{I},\Corr(\mathscr{C}))$ as a full subgroupoid of $\Mor_{\Cat_{\oo}}(\mathscr{I} \times \mathscr{I}^{op},\mathscr{C})$. For $\mathscr{I}=[1] \times [n]$, we have the following:

\begin{prop}
	\label{naturaltransformation}
	$\textup{Mor}_{\Cat_{\oo}}([1] \times [n],\Corr(\mathscr{C}))  \subset 
	\textup{Mor}_{\Cat_{\oo}}(([1] \times [1]^{op}  \times [n] \times [n]^{op}),\mathscr{C})$ is the full $\oo$-subgroupoid consists of $\underline{c}^{0,0}  \leftarrow \underline{c}^{0,1} \to \underline{c}^{1,1}$, such that $\underline{c}^{l,k} \in \textup{Grid}^{\geq \textup{dgnl}}_n(\mathscr{C})$ and the following labeled squares are cartesian, for all $i,j$:
	\begin{equation}
	\label{twocartesian}
	\xymatrix{  c^{1,1}_{i+1,j+1}   &  c^{1,1}_{i,j+1} \ar[l] \ar[r] \ar@{}[rd]|-{\square} &  c^{1,1}_{i,j} \\
		c^{0,1}_{i+1,j+1} \ar[u] \ar[d] \ar@{}[rd]|-{\square} &    c^{0,1}_{i,j+1} \ar[l] \ar[r] \ar[u] \ar[d]  &  c^{0,1}_{i,j} \ar[u] \ar[d] \\
		c^{0,0}_{i+1,j+1}  &         c^{0,0}_{i,j+1}     \ar[r] \ar[l]      &   c^{0,0}_{i,j}
	} 
	\end{equation}
\end{prop}	

\begin{proof} Let $\textbf{c}=c^{l,k}_{i,j} \in  \textup{Mor}_{\Cat_{\oo}}(([1] \times [1]^{op}  \times [n] \times [n]^{op}),\mathscr{C})$. Then 
	
	$\textbf{c} \in  \textup{Mor}_{\Cat_{\oo}}([1] \times [n],\Corr(\mathscr{C}))$ 
	
	$\Longleftrightarrow$ for any $f: [m] \to [1] \times [n]$, $f^*\textbf{c} \in \Mor_{\Cat_{\oo}}([m],\Corr(\mathscr{C}))$, 
	
	$\Longleftrightarrow$ for $f_i:[n+1] \to [1] \times [n]$, $f_i^*\textbf{c} \in \Mor_{\Cat_{\oo}}([n+1],\Corr({\mathscr{C}}))$, where $f_i$ is the $(n+1)$-simplex $(0,0) \to (0,1) \to ... \to (0,i) \to (1,i) \to ...\to (1,n)$, for $i=0,1,...,n$.
	
	$\Longleftrightarrow$ for $i=0,1,...,n$, the following squares are cartesian: 
	\[\xymatrix{ c_{0,n}^{0,1}  \ar[r] \ar[d] &  c_{0,n-1}^{0,1}  \ar[r] \ar[d]  & ... \ar[d] \ar[r] & c^{0,1}_{0,i} \ar[r] \ar[d] &  c_{0,i}^{0,0} \ar[r] \ar[d]  &  ... \ar[d] \ar[r] & c^{0,0}_{0,1} \ar[d] \ar[r]  & c_{0,0}^{0,0}   \\
		c_{1,n}^{0,1}   \ar[r]  \ar[d]  & c_{1,n-1}^{0,1}  \ar[r] \ar[d]  & ... \ar[d] \ar[r]  &  c^{0,1}_{1,i} \ar[r] \ar[d] & c^{0,0}_{1,i}  \ar[d] \ar[r] & ... \ar[d] \ar[r] & c^{0,0}_{1,1}     \\
		... \ar[r]   \ar[d]      &     ... \ar[r] \ar[d]   &  ... \ar[d] \ar[r]  &  ... \ar[d] \ar[r] & ... \ar[d] \ar[r] &  ...  \\
		c_{i,n}^{0,1}   \ar[r] \ar[d]  & c^{0,1}_{i,n-1}  \ar[d] \ar[r] & ... \ar[d] \ar[r]  & c^{0,1}_{i,i} \ar[d] \ar[r]  & c^{0,0}_{i,i}  \\ 
		c_{i,n}^{1,1}   \ar[d] \ar[r] & c_{i,n-1}^{1,1}  \ar[d] \ar[r] &  ... \ar[d] \ar[r] & c^{1,1}_{i,i} \\ 
		...   \ar[d] \ar[r]  & ... \ar[d] \ar[r]   &  ...   & 	\\
		c_{n-1,n}^{1,1}    \ar[d] \ar[r]  & c_{n-1,n-1}^{1,1}    \\
		c_{n,n}^{1,1}
	}
	\]
	The squares can be divided into five groups: (1) two rows consisting of $c^{0,1}_{i,*}$ and $c^{1,1}_{i,*}$; (2) two columns consisting of $c^{0,1}_{*,i}$ and $c^{0,0}_{*,i}$; (3) the left upper corner consisting of $c^{0,1}_{*,*}$; (4) triangle on the right consisting of $c^{0,0}_{*,*}$; (5) triangle on the bottom consisting of $c^{1,1}_{*,*}.$ Then $(1),(2)$ corresponds to the two cartesian squares in (\ref{twocartesian}), and (3),(4),(5) corresponds to the condition that $\underline{c}^{l,k} \in \textup{Grid}^{\geq \textup{dgnl}}_n(\mathscr{C})$. 
\end{proof}

\begin{cor}
	\label{natural trans in corr}
	Let $\underline{c} \to \underline{d} \in \textup{Mor}_{\Cat_{\oo}}([1] \times [n] \times [n]^{op},\mathscr{C})$, such that $\underline{c} , \underline{d} \in \textup{Grid}_n^{\geq \textup{dgnl}}(\mathscr{C})$. 
	\begin{enumerate}
		\item Assume that for any $i,j$ the square 
		$$\xymatrix{ c_{i,j+1} \ar[r] \ar[d] &   c_{i,j} \ar[d]  \\
			d_{i,j+1}   \ar[r]        &    d_{i,j}
		}$$
		is cartesian. Then $\underline{c} \leftarrow \underline{c} \to \underline{d} \in \textup{Mor}_{\Cat_{\oo}}([1] \times [n],\Corr(\mathscr{C})). $ 
		\item Similarly, assume that for any $i,j$ the square 
		$$\xymatrix{ c_{i,j} \ar[r] \ar[d] &   c_{i+1,j} \ar[d]  \\
			d_{i,j}   \ar[r]        &    d_{i+1,j}
		}$$
		is cartesian. Then $\underline{c} \leftarrow \underline{c} \to \underline{d} \in \textup{Mor}_{\Cat_{\oo}}([1] \times [n],\Corr(\mathscr{C})^{op}). $ 
	\end{enumerate}	
\end{cor}	

\subsubsection{Correspondences and sheaves}
See Appendix~\ref{analytic stacks} for our convention on analytic stacks. Let $\Stk^\circ \subset \Stk$ be the subcategory consist of analytic stacks and morphisms which are representable up to unipotent gerbes.

For $f:X\to Y$ in $\Stk^\circ$, the functors $f_!,f_*,f^!,f^*$ are defined. For $F: X \to Y$ a morphism in $\Corr(\Stk^\circ)$ given by $ X \xleftarrow{f} Z \xrightarrow{g} Y$, we define $Sh(F):=g_*f^!: Sh(X) \to Sh(Y)$ a morphism in $\Cat_\infty$. Following the method in \cite[Chapter 7, Theorem 3.2.2, Theorem 5.2.4]{GR}, we extend $Sh(F)$ to a functor 
$$ Sh: \Corr(\Stk^\circ) \to \Cat_{\infty}.$$


\begin{eg}
	\label{1simplex}
	Let $X_i \leftarrow Y_i \to Z_i$, $i=1,2$ be two $1$-simplices in $\Corr({\Stk^\circ})$. Then by Proposition~\ref{naturaltransformation}, the following commutative diagram 
	$$ \xymatrix{  X_1   &  Y_1 \ar[l] \ar[r] \ar@{}[rd]|-{\square} &  Z_1 \\
		X \ar[u] \ar[d] \ar@{}[rd]|-{\square} &    Y \ar[l] \ar[r] \ar[u] \ar[d]  &  Z \ar[u] \ar[d] \\
		X_2  &        Y_2     \ar[r] \ar[l]      &   Z_2
	} $$
	gives a map between these two $1$-simplices (i.e a functor from $\Delta^1 \times \Delta^1 \to \Corr(\Stk^\circ)$). Composing with $Sh$ we get the following morphism between 1-simplices in $\Cat_\oo$:
	$$\xymatrix{ Sh(X_1)  \ar[r] \ar[d]  &  Sh(Z_1) \ar[d]  \\
		Sh(X_2)	 \ar[r]	 & 	Sh(Z_2)		}$$
	The commutativity is induced by base change isomorphisms of the two cartesian squares.

\end{eg}

\section{Preliminaries on Lie theory}

\subsection{Groups generated by reflections}
A reference for this section is \cite[V]{Bour}.

We will denote by $A$  a real affine space of finite dimension, and by $L$ the vector space of translations of $A$. Assume that $L$ is provided with an inner product. Let $\mathfrak{H}$ be a set of hyperplanes of $A$, and $W=W_\frH$ be the subgroup of automorphism of $A$ generated by orthogonal reflections $r_H$ with respect to the hyperplanes $H \in \mathfrak{H}$. We assume the following conditions are satisfied:

\begin{enumerate}
\item For any $w \in W$ and $H \in \mathfrak{H}$, the hyperplane $w(H)$ belongs to $\mathfrak{H};$
\item The group $W$, provided with the discrete topology, acts properly on $A$.
\end{enumerate}

Given two points $x$ and $y$ of $E$, denote by $R\{x,y\}$ the equivalence relation: \\
for any hyperplane $H \in \mathfrak{H}$, either $x \in H$ and $y \in H$ or $x$ and $y$ are strictly on the same side of $H$.

\begin{defn}  \hfill
\label{facet}
\begin{enumerate} 
\item A \textit{facet} of $A$ is an equivalence class of the equivalence relation defined above.
\item A \textit{chamber} of $A$ is a facet that is not contained in any hyperplane $H \in \mathfrak{H}$.
\item A \textit{vertex} of $A$ is a facet that consists of a single point.
\item For $S \subset A$ subset, the \textit{star} of $S$ is $St_S:=\bigcup_{J \text{ facet}, S \cap \overline{J} \neq \emptyset} J$; and  $W_S:=\{w \in W: w|_S =id\}$ denotes the group of elements fixing $S$. Note that for $S,J$ facets, then $S \cap \overline{J} \neq \emptyset$ if and only if $S \subset \overline{J}$.
\item Let $J$ be a facet, A \textit{face} of $J$ is a facet $I$, such that $I \subset \overline{J}$.
\item For a facet $J$, denote by $\mathscr{F}_J$ the poset of all faces of $J$, and for  $I,I'\in \mathscr{F}_J$, we say $I \leq I'$ if $ I \subset \overline{I}'$. We also view $\mathscr{F}_J$ as a category with morphism given by the partial order.
\end{enumerate}
\end{defn}

We collect some facts:

\begin{thm} \hfill
\label{reflectiongroup}
\begin{enumerate}
\item For $J \subset A$ a facet, the group $W_J$ is generated by $\{r_H : J \subset H\}$.
\item For any chamber $C$, the closure $\overline{C}$ of $C$ is a fundamental domain for the action of $W$ on $A$, i.e., every orbit of $W$ in $A$ meets $\overline{C}$ in exactly one point. 
\end{enumerate}
\end{thm}

Fix a chamber $C$, for faces $J,J'$ of $C$ (which are automatically facets of $E$), such that $J \subset \overline{J'}$, we have $St_{J'} \subset St_J$ and $W_{J'} \subset W_J$. The maps $St_{J'}/W_{J'} \rightarrow St_J/W_J$ and $St_{J'}//W_{J'} \rightarrow St_J//W_J$ give two functors $\mathscr F_C \rightarrow \mathscr{S}$.

\begin{prop}
\label{opencolimit}
 The morphism in $\mathscr{S}$:
  $$ \xymatrix{ 
 	\colim_{\mathscr F^{op}_C} St_J/W_J \ar[r]^-{\sim} & A/W 
 }$$ 
 $$ \xymatrix{ 
\colim_{\mathscr F^{op}_C} St_J//W_J \ar[r]^-{\sim} & A//W 
  }$$ 
are equivalences, where $St_J$ and $A$ are equipped with discrete topology.
\end{prop}
\begin{proof} The two statement are equivalent since all arrows (including the augmentations) in the diagram are fully-faithful maps.
For the second statement, we have $St_J//W_J = St_J \cap \overline{C}$, and $A//W=\overline{C}$ by last Theorem. Hence $\colim_{\mathscr{F}^{op}_C} St_J//W_J \simeq \coprod_{x \in \overline{C}} |\textup{N}(x/\mathscr{F}^{op}_C)| \times \{x\} \simeq \coprod_{x \in \overline{C}} \{x\}  \simeq \overline{C}$. Where $x/\mathscr{F}^{op}_C \subset \mathscr{F}^{op}_C$ is the full subcategory consists of faces $I$, such that $x \in St_I$, and $|\textup{N}(-)|$ denotes the geometric realization of the nerve of $-$. We see that $x/\mathscr{F}^{op}_C = \mathscr{F}^{op}_{I}$, for $I$ the face containing $x$, hence $|\textup{N}(x/\mathscr{F}^{op}_C)| \simeq \overline{I}$ is contractible.
\end{proof}

\subsection{Lie theoretic reminder}
\label{Lietheoretic}

\begin{no}
\label{lienotation}
Let $G$ be a reductive algebraic group, $T \subset G$ a maximal torus. Denote by $\Phi=\Phi(G,T)$ the set of roots, by $X_*(T):=\Hom(\C^*,T)$ the coweight lattice, and by $\t_\R:=X_*(T) \otimes \R$. The Weyl group $W:=N_G(T)/T$ acts naturally on $T, X_*(T)$ and $\t_\R$. Let $W_\aff:= W \ltimes X_*(T)$ be the affine Weyl group and $ \Phi_{\aff}:=\{ \alpha_0 -n : \alpha_0 \in \Phi, n \in \Z \} \subset Map(\t_\R,\R)$ be the set of affine roots.  
Denote by $\g$ the Lie algebra of $G$, by $L\g$ and $LG$ the polynomial loop algebra and loop group. For any $\alpha_0 \in \Phi$, denote by $\g_{\alpha_0} \subset \g$ the root space of $\alpha_0$, and for $\alpha=\alpha_0 - n \in \Phi_\aff$, denote by $\g_\alpha:=\g_{\alpha_0} z^n \subset L\g$ the root space of $\alpha$. Fix a lift of set $W \to N_G(T) \subset G$. It gives a lift $W_\aff \to LG$. For $w \in W_\aff$, denote its lift by $\dot{w}$.
\end{no}
Assume further that $G$ is semisimple and simply-connected.  Then $\t_\R$ carries an inner product induced by the Killing form. Denote by $\frH:=\{ \{\alpha(x)=0: x \in \t_\R\}_{\alpha \in \Phi} \}$ and $\frH_{\aff}:=\{ \{\alpha(x)=0: x \in \t_\R\}_{\alpha \in \Phi_\aff}\} $ two collections of hyperplanes in $\t_\R$,  let $W_\frH,W_{\frH_\aff}$ be the corresponding groups generated by reflections. 
The inclusion $\Z \subset \R$ induces $ X_*(T) \subset \t_\R$.

\begin{thm} Viewing $X_*(T)$ as translations of $\t_\R$,  we have the following equality as subgroup of affine linear transformation of $\t_\R$:
\begin{enumerate}

\item $W_\frH=W$
\item  $W_{\frH_\aff} = W_\aff $.
\end{enumerate}
\end{thm}

\subsubsection{Levi and Parabolic subgroups associated to facet geometry} 
\begin{defn}
	\label{parabolicfacet}
Let $J$ be a facet of $\t_\R$ equipped with $\frH_\aff$. 
\begin{enumerate}
\item $\Phi_J:=\{ \alpha \in  \Phi_\aff: \alpha(J)=0   \}$.

\item Denote by $\g_J \subset L\g$  the Lie-subalgebra:
$$ \g_J:= \t \oplus \bigoplus_{\alpha \in \Phi_J } \g_\alpha, $$
where $\g_\alpha \subset L\g $ is the root space of $\alpha$. Denote by $G_J$ the corresponding subgroup of $LG$.

\item Let $J,J'$ two facet with $J \subset \overline{J'}$, denote by $\frp^{J}_{J'} \subset \g_J$ the subalgebra:
$$\frp^{J}_{J'}:= \t \oplus \bigoplus_{\alpha \in \Phi_J, \alpha(J')>0} \g_\alpha.$$ 
Denote by $P^J_{J'}$ the corresponding subgroup of $G_J$.

\item For $a \in \t_\R$, put $G_a = G_J$ for $J$ the affine facet $J$ containing $a$.
\end{enumerate}
\end{defn}

Theorem~\ref{reflectiongroup} (1) implies:
\begin{prop}
\label{connected}
$W_J \subset G_J.$ 
\end{prop}

\subsubsection{Transitivity of parabolic subgroups}
\label{transpara}
Let $R  \subset G$ be a parabolic subgroup with Levi $K$, and $Q \subset K$ be a parabolic subgroup with Levi $L$. Denote by $P:= Q \circ R:= Q \times_K R.$ We have the following diagram:
\[
\xymatrixcolsep{0.4in}
\xymatrix{   & &  P \ar[dr]^{\tilde{p}_1} \ar[ld]_{\tilq_2} & & \\
	& Q  \ar[dr]^{p_1} \ar[dl]_{q_1} \ar@{}[rr]|{\square}&  &   R \ar[dr]^{p_2} \ar[dl]_{q_2} \\
	L  & &  K  &  & G
}  \]

Denote by $\frl,\frp,\frq,\frk,\frr,\frg$ the corresponding Lie algebras.

\begin{prop}
\label{transad}
 There are commutative diagrams of stacks with the middle squares being cartesian, where all actions are adjoint actions:
\[
\xymatrixcolsep{0.4in}
\xymatrix{   & &  \frp/P \ar[dr]^{\tilde{p}_1} \ar[ld]_{\tilq_2} & & \\
& \frq/Q  \ar[dr]^{p_1} \ar[dl]_{q_1} \ar@{}[rr]|{\square}&  &   \frr/R \ar[dr]^{p_2} \ar[dl]_{q_2} \\
\frl/L  & &  \frk/K  &  & \frg/G
}  \]
\[
\xymatrixcolsep{0.4in}
\xymatrix{   & &  P/P \ar[dr]^{\tilde{p}_1} \ar[ld]_{\tilq_2} & & \\
& Q/Q  \ar[dr]^{p_1} \ar[dl]_{q_1} \ar@{}[rr]|{\square}&  &   R/R \ar[dr]^{p_2} \ar[dl]_{q_2} \\
L/L  & &  K/K  &  & G/G
}  \]

\end{prop}

\begin{proof}
We have
\[
\xymatrixcolsep{0.4in}
\xymatrix{   & &  BP \ar[dr]^{\tilde{p}_1} \ar[ld]_{\tilq_2} & & \\
	& BQ  \ar[dr]^{p_1} \ar[dl]_{q_1} \ar@{}[rr]|{\square}&  &   BR \ar[dr]^{p_2} \ar[dl]_{q_2} \\
	BL  & &  BK  &  & BG
}  \]
	
Therefore, for any stack $X$, we have 
\[
\xymatrixcolsep{0in}
\xymatrix{   & & Map(X, BP) \ar[dr]^{\tilde{p}_1} \ar[ld]_{\tilq_2} & & \\
& Map(X,BQ)  \ar[dr]^{p_1} \ar[dl]_{q_1} \ar@{}[rr]|{\square}&  &   Map(X,BR) \ar[dr]^{p_2} \ar[dl]_{q_2} \\
Map(X,BL)  & & Map(X,BK)  &  &  Map(X,BG)
}  \]
Then we obtain the first diagram by taking $X=B\widehat{\GG}_a$, and second diagram by taking $X=S^1$. 
\end{proof}

The following proposition is easy to check:
\begin{prop} Recall $\g_J,\frp^{J}_{J'}$ as in Definition~\ref{parabolicfacet}, we have:
\begin{enumerate}
\item $\g_{J'} \subset \frp^J_{J'} \subset \g_J$, and $\frp^J_{J'}$ is a parabolic subalgebra of $\g_J$ with Levi factor $\g_{J'}$.
\item $\frp^{J'}_{J''} \circ \frp^J_{J'} = \frp^J_{J''}$.
\end{enumerate}
\end{prop}

\section{Preliminaries on singular support} A reference for this section is \cite{KS}. In the following sections, we assume for simplicity that the coefficient ring $k$ has characteristic $0$, see Remark~\ref{char0} about how to drop this constrain. Actually many of the statements hold for more general (possibly non-stable) coefficients. 

Let $X$ be a differentiable manifold, $F \in Sh(X)$ a sheaf on $X$. Then the \textit{singular support} of $F$ is the subset of $T^*X$ defined by the following equivalent conditions:
\begin{enumerate}
	\item  $(x,\xi) \notin SS(F) $ .
	\item  For any differentiable function $f$ on $X$, with $f(x)=0$ and $df_x=\xi,$ the map $ \colim_{U \ni x} F(U)  \xrightarrow{\sim} \colim_{U \ni x} F(U \cap \{f<0\})$ is an isomorphism.
\end{enumerate}
For $X$ be a smooth analytic stack, $SS(F) \subset  T^*X$ is defined via descent. The singular support relates the sheaf theory on $X$ to the symplectic geometry of $T^*X$:
\begin{thm}
	\begin{enumerate}
		\item  $SS(F)$ is a closed conical coisotropic substack of $T^*X$.
		\item  $SS(F)$ is a Lagrangian substack if and only if $F$ is weakly constructible, i.e, there is a stratification $\{S_i\}_{i \in I}$ of $X$, such that $F|_{S_i}$ is a locally constant sheaf for all $i \in I$
	\end{enumerate}
	
\end{thm}	

\cite{La,Gi3} showed that the global nilpotent cone in $T^*\Bun_G$ is a Lagrangian, hence all sheaves we considered are actually weakly-constructible.  
For  $\Lambda \subset T^*X$, put $Sh_{\Lambda}(X):= \{ F \in Sh(X): SS(F) \subset \Lambda \}$, $D^{b}_{\Lambda}(X):=\{ F \in Sh_{\Lambda}(X): F \text{ is constructible }
\}$ the $\oo$-category of constructible sheaves, and $\Perv_{\Lambda}(X):=\{F \in D^{b}_{\Lambda}(X): F \text{ is perverse}\}$ the category of perverse sheaves.
Our results in the remaining of the paper will be stated for $Sh_\Lambda(X)$, but they also apply to $D^{b}_{\Lambda}(X)$ and $\Perv_{\Lambda}(X)$.

Let $f: Y \to X$ be a map between smooth analytic stacks. For any $\Lambda \subset T^*X$, denote $f^*\Lambda:=\;^t f'(f_\pi^{-1}(\Lambda))$ and $f_*\Lambda:= f_\pi(\;^t f'{}^{-1}(\Lambda))$:
$$\xymatrix{ T^*Y  & Y \times_X T^*X \ar[r]^-{ ^t f'}  \ar[l]_-{f_\pi} &  T^*X  }$$
Assume further that $f$ is smooth, then $SS(f^!(F))=SS(f^*(F)) = f^*SS(F)$. Hence we have:
\begin{prop}
	\label{descentss}
	In the context of Theorem~\ref{descent}, let $\Lambda \subset T^*X(*)$. For any $I \in \mathscr{I}^{\triangleright},$ put $f_I:= X(I \to *): X(I) \to X(*)$. Then the functor: 
	$Sh_{X,\Lambda}: \mathscr{I}^{\triangleright,op} \to \Cat_\oo,$ by $I \mapsto Sh_{f_I^*\Lambda}(X(I))$ is a limit diagram.
\end{prop}

\subsection{Fourier and Radon transform}

Let $V$ be a complex vector space, denote by $Sh_{\mathbb{R}^+}(V) \subset Sh(V)$ the fully subcategory of conical sheaves, i.e sheaves that are locally constant along $\mathbb{R}^{+}$-orbits.
The \textit{Fourier-Sato transform} is by definition  $\TT:= q_{2!}q_1^*[\dim_\C V]: Sh_{\mathbb{R}^+}(V) \to Sh_{\mathbb{R}^+}(V^*)$ 
$$\xymatrix{
	V   & Q \ar[l]_{q_1} \ar[r]^{q_2}  & V^* , }$$
where $Q=\{(x,y) \in V \times V^* | \textup{ Re} ( \langle x,y \rangle) \leq 0 \}.$ The functor $\TT$ is an equivalence between $\oo$-categories (with this shift, $\TT$ also preserves perverse sheaves).

Let $G$ be a reductive group, the \textit{Radon transform} and \textit{inverse Radon transform}  are by definition $R:=s_*r^!, \check{R}:=r_!s^*$: 
$$ \xymatrix{  \frac{N\backslash G/N}{T}  &    G/_{\Ad}B    \ar[r]^-r \ar[l]_-{s}  &      G/G              }$$

 Let $P$ be a parabolic subgroup of a reductive group $G$ with Levi factor $L$, define the \textit{parabolic restriction} with respect to $\frp$ to be the functor $\Res_\frp:=q_*p^!: Sh(\g/G) \to Sh(\frl/L)$:
\begin{equation}
\xymatrix{
	\frl/L   & \frp/P \ar[l]_q \ar[r]^p  & \g/G  }
\end{equation}
And similarly $\Res_P:=q_*p^!:$
$$\xymatrix{
	L/L   & P/P \ar[l]_q \ar[r]^p  & G/G  }$$

Assume further that $P$ contains $B$, put $B_L, N_L$ the image of $B,N$ under $P \to L$. Then $B_L$ is a Borel subgroup of $L$, and $N_L$ is the nilpotent radical of $B_L$. Define the corresponding parabolic restriction for Hecke categories $\Res_P:=q_* p^! : $
$$ \xymatrix{  \frac{N_L\backslash L/N_L}{T}  & \frac{N\backslash P/N}{T}   \ar[r]^-p  \ar[l]_-q & \frac{N\backslash G/N}{T}   } $$

Statement (2) below was pointed out to us by Sam Gunningham:
\begin{prop} 
	\label{restrictioncommute}
	 \begin{enumerate} 
		\item Identify $\g^* \simeq \g$ , $ \frl^* \simeq  \frl $ via an invariant bilinear form $\kappa$,  then $\Res_\frp$ naturally commute with Fourier transformation $\TT$:
		$$\xymatrix{   Sh_{\R^+}(\g/G)   \ar[r]^{\Res_\frp}  \ar[d]_{\TT}^{\sim}  &      Sh_{\R^+}(\frl/L) \ar[d]_{\TT}^{\sim}  \\
			Sh_{\R^+}(\g/G)   \ar[r]^{\Res_\frp}   &   Sh_{\R^+}(\frl/L) } $$
		\item $\Res_P$ naturally commute with Radon transformation $R$ 
			$$\xymatrix{  Sh(G/G)   \ar[r]^{\Res_P}  \ar[d]_{R}  &     Sh(L/L)   \ar[d]_{R} \\
			Sh(\frac{N\backslash G/N}{T})    \ar[r]^{\Res_P}   & Sh(\frac{N_L\backslash L/N_L}{T})  } $$
		\end{enumerate}
\end{prop}
\begin{proof}
	(1) is \cite[Lemma 4.2]{Mi}. We include a proof for reader's convenience. Denote by $\check{p}:\g^* \to \frp^*$ and $\check{q}: \frl^* \to \frp^*$ the dual map of $p$ and $q$. Put $\frn_\frp$ the nilpotent radical of $\frp$, and $d=\dim \frn_\frp$. We have the commutative diagram:
	$$\xymatrix{ Sh_{\R^+}(\g/G) \ar[r] \ar[d]^{\TT} & Sh_{\R^+}(\g/P)  \ar[r]^{p^!} \ar[d]^{\TT} & Sh_{\R^+}(\frp/P) \ar[r]^{q_*} \ar[d]^{\TT} & Sh_{\R^+}(\frl/P) \ar[r]^{\sim} \ar[d]^{\TT} & Sh_{\R^+}(\frl/L) \ar[d]^{\TT} \\
		Sh_{\R^+}(\g^*/G) \ar[r] \ar[d]^{\kappa} & Sh_{\R^+}(\g^*/P)  \ar[r]^{\check{p}_*[-d]} \ar[d]^{\kappa} & Sh_{\R^+}(\frp^*/P) \ar[r]^{\check{q}^![d]} \ar[d]^{\kappa}& Sh_{\R^+}(\frl^*/P) \ar[r]^{\sim} \ar[d]^{\kappa} & Sh_{\R^+}(\frl^*/L) \ar[d]^{\kappa} \\
		Sh_{\R^+}(\g/G) \ar[r] \ar@{=}[d] & Sh_{\R^+}(\g/P)  \ar[r]^{\check{p}_*[-d]} \ar@{=}[d] & Sh_{\R^+}((\frg/\frn_\frp)/P) \ar[r]^{\check{q}^![d]}  & Sh_{\R^+}(\frl/P) \ar[r]^{\sim} \ar@{=}[d] & Sh_{\R^+}(\frl/L) \ar@{=}[d] \\
		Sh_{\R^+}(\g/G) \ar[r] & Sh_{\R^+}(\g/P)  \ar[r]^{p^!} & Sh_{\R^+}(\frp/P) \ar[r]^{q_*} & Sh_{\R^+}(\frl^*/P) \ar[r]^{\sim} & Sh_{\R^+}(\frl^*/L)  
	}$$
	where the top middle two squares are given by \cite[Proposition 3.7.14]{KS}, and the bottom midterm square is the base change isomorphism of the cartesian square:
	$$\xymatrix{    \frp \ar[r] \ar[d] &  \frp/\frn_\frp = \frl \ar[d] \\
		\g \ar[r]   &    \frg/\frn_\frp} $$
	(2)  The natural isomorphism is given by the diagram:
	$$ \xymatrix{ 
	\frac{N_L\backslash L/N_L}{T}  	  &   L/_{\Ad} B_L \ar[r] \ar[l]  \ar@{}[rd]|-{\square} &    L/L   \\
	\frac{N\backslash P/N}{T} \ar[u] \ar[d]	\ar@{}[rd]|-{\square}  &	P/_{\Ad} B  \ar[r] \ar[l] \ar[u] \ar[d]  	&  P/P \ar[u] \ar[d]		\\
	\frac{N\backslash G/N}{T}	  & G/_{\Ad}B \ar[r] \ar[l]  &	G/G		 } $$
\end{proof}

\subsection{Retractable substacks}
\label{rescaling}

\begin{defn}
	Let $X$ be an analytic stack, $\Lambda \subset T^*X$. A open substack $j: U \hookrightarrow X$ is \textit{retractable} w.r.t. $\Lambda$ if the restriction functor:
	$$ \xymatrix{Sh_{\Lambda}(X) \ar[r]_{\sim}^{j^*} & Sh_{\Lambda}(U)}$$
	is an equivalence.
\end{defn}	

Let $A$ be a Lie group acting on a smooth manifold $X$, $\mathcal{L} \subset T^*X$ be a closed $A$-invariant conical isotropic subset. Let $X/A$ be the quotient stack, $\pi: X \rightarrow X/A$ the natural projection and $\mu: T^*X \rightarrow \mathfrak{a}^*$ the moment map, then $\pi^*(T^*(X/A))=\mu^{-1}(0)= \coprod_{x \in X} T^*_{Ax}X$, and

\begin{prop}
The subset	$\mathcal{L}$ is contained in $\mu^{-1}(0)$. 
\end{prop} 
\begin{proof}
	Suffices to check on the smooth locus of $\mathcal{L}$, where it follows from definition of isotropic submanifold.
\end{proof}  

\begin{prop}
	\label{orbit}
	Let $U \subset X$ open subset, $F \in Sh_{\mathcal{L}}(U):=Sh_{\mathcal{L}|_U}(U)$, then $F|_{Ax \cap U}$ is locally constant for all $x \in U$.
\end{prop} 
\begin{proof} $SS(F) \subset \mathcal{L} \subset \pi|_U^*(T^*(X/A))$, hence by \cite[Prop. 6.6.2]{KS}, in a neighborhood of $x \in U$, $F=\pi|_U^*(F')$, for some $F' \in Sh(X/A)$. 
\end{proof}

For $X$ a manifold with $\mathbb{R}^+$ action, recall that a biconical subset of $T^*X$ is a conical subset invariant under the induced $\RR^+$ action on $T^*X$. The following proposition is the most important technical result of the paper.

\begin{prop}
	\label{contraction}
	Let $X$ be a smooth manifold with $\mathbb{R}^{+}$ action, and $j: U \hookrightarrow X$ an open embedding, such that $U \cap \mathbb{R}^{+}x$ is contractible, for any $x \in X$. Then $U$ is retractile w.r.t any closed biconical isotropic subset $\mathcal{L}$ of $T^*X$. 
\end{prop}

\begin{proof} By Proposition~\ref{orbit}, any $F \in Sh_{\mathcal{L}}(X)$ is conic, i.e. satisfies $F|_{\mathbb{R}^{+}x}$ is locally constant for all $x \in X$. Hence for $F_1,F_2 \in Sh_{\mathcal{L}}(X)$ by \cite[Prop. 3.7.4(iii), Cor. 3.7.3]{KS}, $Hom_{Sh_{\mathcal{L}}(X)}(F_1,F_2) \xrightarrow{\simeq} Hom_{Sh_{\mathcal{L}}(U)}(j^*F_1,j^*F_2)$ is an isomorphism, hence $j^*$ is fully faithful. Let $F' \in Sh_{\mathcal{L}}(U)$, then by Proposition~\ref{orbit}, for any $x \in U$, $F'|_{\mathbb{R}^{+}x \cap U}$ is locally constant. We have natural maps
	\[
	\xymatrix{
		U \times \mathbb{R}^{+} \ar[d]^{\tilde{j}} \ar[dr]^{a'} & U \ar[l]_{i'} \ar[d]^j \\
		X \times \mathbb{R}^{+} \ar@<-1ex>[r]_p  \ar@<1ex>[r]^a & X \ar[l]|i
	}
	\]
	where $\tilde{j}=j \times Id$, $a$ is the action map, $p$ is the projection, $i(resp. \; i')$ is inclusion to $X(resp. \; U) \times \{1\}$. $a':=a \circ \tilde{j} : U \times \mathbb{R}^{+} \rightarrow X$, then $a'$ is $\mathbb{R}^{+}$-equivariant, has contractible fibers and $F' \boxtimes k_{\mathbb{R}^{+}}$ is constructible along the fibers of $a'$. Define $F:=a'_*(F' \boxtimes k_{\mathbb{R}^{+}}) \in Sh(X)$, then $F$ is conic by \cite[Prop. 3.7.4(ii)]{KS} . There is a chain of isomorphisms 
	\begin{equation}
	\label{chain}
	F' \simeq i'^*(F' \boxtimes k_{\mathbb{R}^{+}})  \simeq i'^*a'^*(F) \simeq i'^* \tilde{j}^*a^*(F) \simeq i'^*\tilde{j}^*p^*(F)  \simeq j^*(F)
	\end{equation}
	where the second isomorphism is by \cite[Prop. 2.7.8]{KS}, the fourth isomorphism is by \cite[Prop. 3.7.2]{KS}. Now $a'$ is smooth and surjective, and $SS(a'^*F)=SS(F' \boxtimes {k}_{\mathbb{R}^{+}}) \subset \mathcal{L} \times T^*_{\mathbb{R}^{+}}{\mathbb{R}^{+}}= a'^*(\mathcal{L})$ since $\mathcal{L}$ is biconical. Hence $SS(F) \subset \mathcal{L}$ by descent  \cite[Prop. 5.4.5]{KS}. Combining with (\ref{chain}), the functor $j^*$ is essentially surjective. 
\end{proof}

For any smooth manifold $U$, denote by $\mathbf{0}_U \subset T^*U$ the zero section. Fix $u \in U$, let $X$ any analytic stack, and $\Lambda \subset T^*X$. We have $\pi:U \times X \to X$, via $(u,x) \mapsto x$, and $i: X \to U \times X$ via $x \mapsto (u,x)$.
\begin{prop}
	\label{contractiblefibers}
	Assume that $U$ is contractible, then $\pi_!,\pi_*,i^!,i^*$ maps $Sh_{\mathbf{0}_U \times \Lambda}(U \times X)$ isomorphically onto $Sh_{\Lambda}(X)$, and $\pi^!,\pi^*$ maps $Sh_{\Lambda}(X)$ isomorphically onto $Sh_{\mathbf{0}_U \times \Lambda}(U \times X)$.
\end{prop}
\begin{proof}
By \cite[Prop. 5.4.5(ii)]{KS} we see that $Sh_{ \mathbf{0}_U \times T*X } (_U \times X)$ exactly consist of the sheaves on $U \times X$ which are contractible along fibers of $\pi$. Then by \cite[Prop.2.7.8]{KS}, the functors 
$$\pi^* :Sh(X)  \longleftrightarrow Sh_{\mathbf{0}_U \times T^*X}(U \times X)  :\pi_*$$ are inverse to each other.
Then by descent $SS(\pi^*(F)) \subset   \mathbf{0}_U \times \Lambda$ if and only if $SS(F) \subset \Lambda$, hence we have inverse functors:
$$\pi^* :Sh_{\Lambda}(X)  \longleftrightarrow Sh_{\mathbf{0}_U \times \Lambda}(U \times X)  :\pi_*.$$
The statement for $(\pi_!,\pi^!)$ follows since $\pi^!=\pi^*$ up to a shift. And the statement for $i^*,i^!$ also follows since $\pi \circ i =id$.
\end{proof}

\subsection{Nilpotent cones} 
We have the identifications: 
$$H^0(T^*(\g/G)) \simeq \{(Y,X) \in \g \times \g \;|\; [Y,X]=0 \}/G,$$
$$H^0(T^*(G/G)) \simeq \{(g,X) \in G \times \g \;|\; \Ad_g(X)=X \}/G,$$
$$H^0(T^*\Bun_G(\Sigma)) \simeq \{(\mathcal{P},\phi) \;|\; \mathcal{P} \in \Bun_G(\Sigma), \phi \in H^0(\Sigma, \g_\mathcal{P} \otimes \Omega_\Sigma) \}. $$ 
Where $\Sigma$ above is a compact Riemann surface. Denote the  corresponding nilpotent cones by: 
$$\N_{T^*(\g/G)}:=\{(Y,X) \in H^0(T^*(\g/G)) \;|\; X \text{ is nilpotent.}\}$$
$$\N_{T^*(G/G)}:=\{(g,X) \in H^0(T^*(G/G)) \;|\; X \text{ is nilpotent.}\}$$
$$\N_{T^*\Bun_G(\Sigma)}:=\{(\mathcal{P},\phi) \in H^0(T^*\Bun_G(\Sigma)) \;|\; \phi \text{ is nilpotent.}\}$$
We will use nilpotent cones lying in different spaces in the paper. When the context is understood, we shall drop the indices and write all nilpotent cones as $\N$. The nilpotent cones above are known to be Lagrangian \cite{La,Gi3}, hence the sheaves with singular support in $\N$ are actually weakly constructible.

\begin{prop}  
	\label{parabolicnilpotent}
	The parabolic restrictions preserve nilpotent singular support:
	\begin{enumerate}
		\item $\Res_\frp$ takes $Sh_\N(\g/G)$ into $Sh_\N(\frl/L).$
		\item $\Res_P$ takes $Sh_\N(G/G)$ into $Sh_\N(L/L)$.
\end{enumerate}	
\end{prop}	
	\begin{proof}
			(1) Since $\N_{T^*(\g/G)}$ is a biconical Lagrangian, we have $Sh_\N(\g/G) \subset Sh_{\R^+}(\g/G)$, and $Sh(\N/G) \subset Sh_{\R^+}(\g/G)$. Then by Proposition~\ref{restrictioncommute} the following diagram commute:
		$$\xymatrix{   Sh_{\N}(\g/G)   \ar[r]^{\Res_\frp}  \ar[d]_{\TT}  &      Sh_{\R^+}(\frl/L) \ar[d]_{\TT}  \\
			Sh(\N_\g/G)   \ar[r]^{\Res_\frp}   &   Sh_{\R^+}(\frl/L), } $$
	 where the left arrow is by  \cite[Theorem 5.5.5]{KS}.
	 Now the bottom arrow takes		$Sh(\N_\g/G) $ into $ Sh(\N_\frl/L)  $ because $q(\N_\g \cap \frp)= \N_\frl$.
		Hence the top arrow lands in $Sh_\N(\frl/L)$ by \cite[Theorem 5.5.5]{KS}. \\
		(2) Recall by assumption, the coefficient $k \supset \QQ$.  Denote by $Sh_T(\frac{N\backslash G/N}{T}) \subset Sh(\frac{N\backslash G/N}{T})$ the subcategory of sheaves constructible w.r.t to the left (or equivalently right) $T$ orbits. By \cite[Theorem 4.4]{MV}, for a sheaf $F$ on $G/G$, we have $SS(F) \in \N$ if and only if $R(F) \in Sh_T(\frac{N\backslash G/N}{T})$.  Then by Proposition~\ref{restrictioncommute}, the following diagram commute:
			$$\xymatrix{  Sh_\N(G/G)   \ar[r]^{\Res_P}  \ar[d]_{R}  &     Sh(L/L)   \ar[d]_{R} \\
			Sh_T(\frac{N\backslash G/N}{T})    \ar[r]^{\Res_P}   & Sh(\frac{N_L\backslash L/N_L}{T})  } $$ 
	\end{proof}	
So $R \circ \Res_P (F) \simeq  \Res_P \circ R(F)$ lies in $Sh_T(\frac{N_L\backslash L/N_L}{T})$. By \cite[Theorem 4.4]{MV} again, we conclude that $\Res_P(F)$ lies in $Sh_\N(L/L)$.

\begin{rmk}
	\label{char0}
	Proposition~\ref{parabolicnilpotent} (2) is the only place in the paper we use the assumption $char(k)=0$. One can alternatively prove this statement by using first part of  Theorem~\ref{groupcharactersheaf}, which does not rely on Proposition~\ref{parabolicnilpotent} (2), and therefore drop the characteristic $0$ assumption.
\end{rmk}


\nc\hol{\mathit{hol}}
\nc\poly{\mathit{poly}}

\section{Twisted conjugacy classes in the loop group}
\label{twistedconj}

Fix $q \in \mathbb{C}^*$ with $|q|<1$,
and let $E=\mathbb{C}^*/q^{\mathbb{Z}}$ be the corresponding elliptic curve. 
In this section, we focus on the connected component $G_E$ of the trivial bundle in the moduli stack of semistable $G$-bundles on $E$.
We describe the geometry of $G_E$ in terms of the Lie theory of  $q$-twisted conjugacy classes in the holomorphic loop group.
 We work in the context of complex analytic stacks (see Appendix~\ref{analytic stacks} for the facts used).


\subsection{Automorphism groups}
\label{automorphismgroups}

The aim of this subsection is to calculate the automorphism groups of semisimple semistable bundles. The main result is Corollary~\ref{automorphism}, stating that the automorphism group can be calculated in terms of affine root systems. This was previously obtained in~\cite[Theorem 5.6]{BEG}, though the approaches to the component groups differ somewhat. Our approach makes the role of the affine Weyl group transparent. 

We adopt Notation~\ref{lienotation}.

Let $\Bun_G(E)$ be the moduli stack of $G$-bundles on $E$. Let $L_{\hol}G$ be the holomorphic loop group of holomorphic maps $g(z):\CC^*\to G$. 
It  acts on itself by $q$-twisted conjugation 
$$
Ad'_{k(z)} g(z) := k(qz)g(z)k(z)^{-1}
$$

For any  $g(z) \in L_{hol}G$, we can define a $G$-bundle 
$$
\xymatrix{
\mathcal {P}_{g(z)}:= \mathbb{C}^* \times_{q^\mathbb{Z}} G \ar[r] & E = \mathbb{C}^*/q^\ZZ
}
$$ 
where the $q^\mathbb{Z}$-action of  is given by 
$$
q\cdot (z,x) = (qz,g(z)x)
$$ 
Note if $g(z), h(z)$ are $q$-twisted conjugate, then their associated bundles
$\mathcal {P}_{g(z)}, \mathcal {P}_{h(z)}$ are isomorphic.

The automorphism group of $\mathcal {P}_{g(z)}$ admits the description
$$Aut(\mathcal {P}_{g(z)}) \simeq \{k(z)|k(qz) =g(z)k(z)g(z)^{-1}\}=C_{L_{\hol}G}(g(z))
$$ 
 as a $q$-twisted centralizer, since the automorphisms of  $\mathcal {P}_{g(z)}$ are isomorphic to the automorphisms of the corresponding $q^\mathbb{Z}$-equivariant  $G$-bundle over $\mathbb{C}^*$.
 
 Since any $G$-bundle on $\CC^*$ is trivializable,
  we have an isomorphism of groupoids 
 $$
 L_{\hol}G/'L_{\hol}G\simeq \Bun_G(E)(\mathbb{C})
 $$
where $/'$ donotes the quotient with respect to $Ad'$.
For $s \in T$, set $$G_s:=C_{L_{hol}G}(s).$$
  Note that $G_s$ is preserved by $q$-twisted conjugation on itself: for $f(z),g(z) \in G_s, f(qz) g(z) f(z)^{-1} \in G_s$ by direct calculation. 

For the moment, $G_s$ is simply an abstract group. In the rest of this subsection, we will calculate  $G_s$ explicitly and equip it with the structure of  algebraic group, which is compatible with the one coming from the automorphism group of a $G$-bundle.

\subsubsection{Calculation of $G_s^0$}
Fix $s \in T$. We start by equipping $G_s$ with a structure of an algebraic group and calculate its neutral component $G_s^0$, the result is given in Corollary~\ref{Gs0}.\\

Recall $LG \subset L_{\hol} G$ denote the subgroup of polynomial loops $g(z) :\CC^*\to G$.
We will regard it as an ind-scheme, more specifically, as the increasing union of its closed subschemes of prescribed zeros and poles.
Let $L\fg$ denote its Lie algebra.

We will begin with $G=GL_N$. Let $T_N \subset GL_N$ be the invertible diagonal matrices. 

\begin{lem}  \label{twist: qde}
Let $f: \mathbb{C}^* \rightarrow \mathbb{C}$ be a holomorphic function. Assume that $f(qz) =af(z)$, for some $a \in \mathbb{C}$. If $a \in q^{\mathbb{Z}}$, then $f(z) =cz^n$, for some $c \in \mathbb{C}$ and $n=log_q(a)$; otherwise, $f(z) \equiv 0$.
\end{lem}
\begin{proof}
Follows from an elementary comparison of the coefficients of  the Laurent expansion of $f$. 
\end{proof}

\begin{prop}
Let $s=diag(\lambda_1, \lambda_2, ... ,\lambda_N) \in T_N$. 

Let $I_s$ consist of those 
$(i,j)$ such that $\lambda_i/\lambda_j \in q^\mathbb{Z}$, where $n_{ij}=\log_q(\lambda_i/\lambda_j)$.

Set $\mathfrak{gl}_{N,s}=\bigoplus_{(i,j)\in I_s} \mathbb{C}z^{n_{ij}}E_{ij} \subset L_\poly \mathfrak{gl}_N$,
where $E_{ij} \in \mathfrak{g}_N$ is the elementary matrix with non-zero  $(i,j)$-entry.

Then under the standard embedding $GL_N \rightarrow \mathfrak{gl}_N$ as invertible matrices, $GL_{N,s}$ consists of the invertible matrices in $\mathfrak{gl}_{N,s}$.
\end{prop}
\begin{proof} 
For $g(z) = \oplus_{(i, j)} g_{ij}(z) E_{ij} \in GL_{N,s}$, observe that
  $g(qz) =s\cdot g(z) \cdot s^{-1}$ is equivalent to $g_{ij}(qz) = (\lambda_i/\lambda_j) g_{ij}(z)$ for all $(i,j)$. By the previous lemma, $g_{ij}(z)=cz^{n_{ij}}$ if $\lambda_i/\lambda_j = q^{n_{ij}}$ and $g_{ij}(z)=0$ otherwise. In particular, $g_{ii}(z)$ is constant. 
\end{proof}

\begin{cor}
$GL_{N,s}\subset L_\hol GL_N$ lies in $LGL_N \subset L_{\hol}GL_N$ and is Zariski-closed therein. 
With its reduced subscheme structure, $GL_{N,s}$ is a reductive algebraic group, its $q$-twisted conjugation is an algebraic action, and the evaluation map 
\[
\xymatrix{
ev_1: GL_{N,s} \ar[r] &  GL_N & g(z) \ar@{|->}[r] &  g(1)
}
\]
is an injective homomorphism of algebraic groups. Furthermore, the Lie algebra of $GL_{N,s}$ is precisely  $\mathfrak{gl}_{N,s}$.
\end{cor}

For a general reductive algebraic group $G$ with maximal torus $T\subset G$, choose an embedding of pairs
$i: (G,T) \rightarrow (GL_N, T_N)$. This induces embeddings $LG \subset L_\poly GL_N$, $G_s \subset GL_{N,s}$, with $G_s =LG \cap GL_{N,s}$. Hence $G_s$ is Zariski-closed in both $GL_{N,s}$ and $LG$.
Thus we have the following generalization of the previous corollary.

\begin{prop}
\label{gs}
The subgroup $G_{s}\subset L_\hol G$ lies in $LG \subset L_{\hol}G$ and is Zariski-closed therein. 
With its reduced subscheme structure, $G_{s}$ is a reductive algebraic group, its $q$-twisted conjugation is an algebraic action, and the evaluation map 
\[
\xymatrix{
ev_1: G_{s} \ar[r] &  G & g(z) \ar@{|->}[r] &  g(1)
}
\]
is an injective homomorphism of algebraic groups. 
Moreover, the natural map 
$G_s \rightarrow Aut(\mathcal{P}_s)$ is an isomorphism of algebraic groups.

\end{prop}
\begin{proof}
Only the last statement needs proof. We have a commutative diagram of abstract groups:
$$\xymatrix{
G_s \ar[rr]^{\sim} \ar@{^{(}->}[rd]_{ev_1} & & Aut(\calP) \ar[ld]^{res_0}  \\
  &  G = Aut(\calP_0)
} $$
where $\calP_0$ is the fiber of $\calP$ over $0 \in E$, and the ``$=$'' means canonical isomorphism. The two vertical maps are injective morphisms of algebraic groups, so the top arrow is also a morphism of algebraic groups.

\end{proof}

Since $G_s =LG \cap GL_{N,s}$, its Lie algebra satisfies $\mathfrak{g}_s$= $L\mathfrak{g} \cap \mathfrak{gl}_{N,s}$. More explicitly, it admits the following description. Regard the set of affine roots $\Phi_\aff$ as a subset of $Map(T,\C^*)$ via $\alpha=\alpha_0 + n \mapsto \{ s \mapsto \alpha_0(s) q^n  \}$, and put $\Phi_s:=\{ \alpha \in \Phi_\aff \, | \, \alpha(s)=1 \}$. For any $\alpha \in \Phi_\aff$, recall the definition of $\g_{\alpha}$ as in Notation~\ref{lienotation}.

\begin{prop}
	\label{calculate gs}
The Lie algebra of $G_{s}$ is precisely $\mathfrak{g}_s=\mathfrak{t} \oplus \bigoplus_{\alpha\in \Phi_s} \mathfrak{g}_\alpha \subset L \fg$.
\end{prop}

\begin{proof} $\mathfrak{gl}_{N,s}$ is a finite dimensional subalgebra of $L\mathfrak{gl}_N$ satisfying $\{X(z) \in L\mathfrak{gl}_N | X(qz)=Ad(s)X(z)\}$. So $\mathfrak{g}_s$= $L_{poly}\mathfrak{g} \cap \mathfrak{gl}_{N,s}$=    $\{X(z) \in L\mathfrak{g} | X(qz)=Ad(s)X(z)\}$.
 Write $X(z)= h(z)+ \sum\limits_{\alpha_0 \in \Phi}f_{\alpha_0}(z)$, with respect to the root decomposition of $\mathfrak{t}$,
 i.e $h(z):\mathbb{C}^* \rightarrow \mathfrak{t}, f_{\alpha_0}(z):\mathbb{C}^* \rightarrow \mathfrak{g}_{\alpha_0}$.
Now the condition $X(qz)=Ad(s)X(z)$ is equivalent to $h(qz)=h(z)$, and $f_{\alpha_0}(qz)= \alpha_0(s) f_{\alpha_0}(z)$.
By Lemma~\ref{twist: qde}
 $h(z)$ is constant function, and the only nonvanishing $f_{\alpha_0}$ are those with $\alpha_0(s)=q^{n_{\alpha_0}}$ for some $n_{\alpha_0} \in \mathbb{Z}$ and in this case $f_{\alpha_0}=z^{n_{\alpha_0}}X_{\alpha_0} \in \g_\alpha$. So  $\g_s=\t \oplus \bigoplus_{\alpha_0 \in \Phi, \alpha_0(s) \in q^\Z} \g_{\alpha_0}z^{n_{\alpha_0}} $. By compare this expression with the definition of $\g_\alpha$ for affine root $\alpha$, the proposition follows.
\end{proof}

\begin{cor}
\label{Gs0}
The subgroup
$G_s^0 \subset LG$ is generated by $T\cup \{ \exp\g_\alpha  \, |\, \alpha\in \Phi_s\}$.

\end{cor}

\begin{eg}
$G=SL_2$, and $T$ diagonal matrices, with roots $\{\alpha, -\alpha\}$
take $s=\begin{pmatrix} \sqrt{q} & 0 \\
                           0    & \sqrt{q}^{-1}
                           \end{pmatrix}
\in T$, hence $\alpha(s)=\sqrt{q}/(\sqrt{q}^{-1})=q$, so $n_{\alpha}=1$, similarly  $n_{-\alpha}=-1$, \\
we have
$X_{\alpha}= \begin{pmatrix} 0 & 1 \\
                                0 & 0
\end{pmatrix},
X_{-\alpha}= \begin{pmatrix} 0 & 0 \\
                                1 & 0
\end{pmatrix} $, then by Corollary~\ref{Gs0},
the group $G_s^0$ is generated by $T,
\exp \begin{pmatrix}  0 & bz \\
                      0 & 0
\end{pmatrix}, $ and $
\exp \begin{pmatrix} 0 & 0 \\
                  cz^{-1} & 0
\end{pmatrix}        $
and so it is equal to the subgroup
$\Big\{ \begin{pmatrix}  a & bz \\
                 cz^{-1} & d
\end{pmatrix} \Big \}$  of $LG$.
In fact, in this case we have $G_s^0=G_s=C_{LG}(s)$. 
\end{eg}

\subsubsection{Calculation of $G_s$} We proceed to calculate $G_s$. The result is given in Corollary~\ref{automorphism}, which states that the component group is controlled by the affine Weyl group. \\

For $M_1, M_2$ two smooth/complex manifolds, write $\mathit{Map}(M_1,M_2)$ for the set of  differentiable/holomorphic maps
$M_1 \to M_2$.

\begin{lem}
\label{normalizer} 
Let $M$ be a connected complex manifold.

Regard $\mathit{Map}(M,G)$ as a group, and $T \subset N_T(G) \subset G\subset \mathit{Map}(M,G)$ as subgroups of constant maps.

Then $N_{\mathit{Map}(M,G)}(T)=\mathit{Map}(M,T) \cdot N_G(T)$ as subgroups of $\mathit{Map}(M,G)$.
\end{lem}

\begin{proof}
Let $f(x) \in N_{\mathit{Map}(M,G)}(T)$. Then $f(x)Tf(x)^{-1}=T$, for any $x \in M$, hence $f(x) \in \mathit{Map}(M, N_G(T))$.  Hence $N_{\mathit{Map}(M,G)}(T) \subset \mathit{Map}(M, N_G(T)).$ Since $M$ is connected, $\mathit{Map}(M, N_G(T))=\mathit{Map}(M,T)\cdot N_G(T)$. Now $\mathit{Map}(M,T) \subset N_{\mathit{Map}(M,G)}(T),$ and $N_G(T) \subset N_{\mathit{Map}(M,G)}(T)$. Hence $\mathit{Map}(M,T) \cdot N_G(T) \subset N_{\mathit{Map}(M,G)}(T).$ 
\end{proof}

\begin{lem} 
\label{twist: normalizer gs}
$(L_{hol}T \cdot N_G(T)) \cap G_s=T \cdot C_{W_\aff}(s)$ as subgroups of $L_{hol}G$.
\end{lem}

\begin{proof} The right hand side does not depend on the lifting of $W$ and equals $(X_*(T) \cdot N_G(T)) \cap G_s $ which naturally sits inside the left hand side. We need to show that $(L_{hol}T \cdot N_G(T)) \cap G_s  \subset X_*(T) \cdot N_G(T)$. Suppose $fw \in (L_{hol}T \cdot N_G(T)) \cap G_s$, 
for $f \in L_{hol}T, w \in N_G(T)$. Then we have $f(qz)w s(f(z)w)^{-1}=s$, i.e $f(qz)= w(s)^{-1}f(z)$.
However,  $w(s)^{-1},s \in T$, and this implies $f \in X_*(T)$ by Lemma~\ref{twist: qde}. 
\end{proof}

Let $W_s:=C_{W_\aff}(s)$ be the stabilizer of $s$ in $W_\aff$. By Lemma~\ref{normalizer} and \ref{twist: normalizer gs}, we have $N_{G_s}(T)=N_{L_{hol}G}(T) \cap G_s=
(L_{hol}T \cdot N_G(T) )\cap G_s = T \cdot W_s$ and the Weyl group $W(G_s,T) :=N_{G_s}(T)/T $ is isomorphic to  $W_s$, therefore we have:

\begin{cor}
\label{automorphism}
$G_s= G_s^0 \cdot N_{G_s}(T) =<G_s^0,\dot{w} \,|\, w \in W_s> = <T,\exp \g_\alpha, \dot{w} \,|\, \alpha \in \Phi_s, w \in W_s >.   $ 
\end{cor}

\begin{eg}
$G=PGL_2$, and $T$ diagonal matrices, with roots $\{\alpha, -\alpha\}$.
Take $s=\begin{pmatrix} \sqrt{q} & 0 \\
                           0    &  1
                           \end{pmatrix}$.
We have $G_s^0=T$, 
and $G_s=<T, w> \simeq T \rtimes \mathbb{Z}/2$, where  
$w=\begin{pmatrix} 0 & z \\
                               z^{-1} & 0
                           \end{pmatrix}$.
\end{eg}

\begin{rmk}
The same method can be used to calculate the automorphism group of any semisimple (not necessarily semistable) bundle.  Theta functions naturally show up in the calculation for non-semistable bundles, so in general $C_{L_{hol}G} (g(z))$ is not contained in $LG$. 
\end{rmk}

\subsubsection{Untwist twisted conjugation} The twisted conjugation of $L_{hol}G$ is very different from the usual conjugation. However, when restricted to the action of $G_s$ on itself, the twisted conjugation is isomorphic to usual conjugation:

\begin{prop}
\label{twistedconjugation}
The left multiplication by $s^{-1} :G_s \to G_s$ is a $G_s$-equivariant isomorphism of algebraic varieties, where the first action is $q$-twisted conjugation, and second action is usual conjugation. In other words, we have an isomorphism of stacks 
$s^{-1}: G_s/'G_s\to G_s/G_s$. 
\end{prop}
\begin{proof}
For any $k(z) \in G_s$, we have $k(qz) \,  s \, k(z)^{-1}=s$, hence $s^{-1}  Ad'_{k(z)} g(z)= s^{-1} k(qz) g(z) k(z)^{-1}= k(z) s^{-1}g(z) k(z)^{-1}=Ad_{k(z)} (s^{-1} g(z))$.
\end{proof}

 The space $G_s^0$ is stable under the twisted conjugation by $G_s$. Hence $G_s^0/'G_s$ has the usual properties of an adjoint quotient (on its neutral component.)

\begin{cor}
 $\mathbb{C}[G_s^0]^{G_s}=\mathbb{C}[T]^{W_s}  $, where the invariants are taken w.r.t the twisted conjugation.
\end{cor}

So there is a map $ \chi'_s: G_s^0 \rightarrow T//'W_s:=\Spec \:\mathbb{C}[T]^{W_s}  $.  Let $U \subset T$ be a $W_s$-invariant open subset, let $V:={\chi'_s}^{-1}(U//'W_s)$ 

\begin{defn}
Let $S$ be a topological space, and $A \subset S$  a subset. We say $A$ is \textit{abundant} (in $S$)  if the only open subset of  $S$ containing $A$ is $S$.
\end{defn}

Note that $A \subset S$ is abundant if and only if $A$ contains all the closed points of $S$. Examples of abundant subsets that we will use are given in the next corollary. 

\begin{cor}
\label{abundant}
The image of $U$ in $|V/'G_s|$ is abundant. 
\end{cor}
\begin{proof}
	All closed conjugation orbits are semisimple. By Proposition~\ref{twistedconjugation} all closed twisted conjugation orbits orbits are semisimple. 
\end{proof}

Let $T^{s-reg} \subset T$ be the locus where the action of $W_s$ is free.

\begin{cor}
Assume further that $U \subset T^{s-reg}$. Then  
$$V/'{G_s}  \xleftarrow{\simeq} U/'N_{G_s}(T)  \xrightarrow{\simeq} (U \times BT)/'W_s.   $$ 
\end{cor}

\subsection{Etale charts}
\label{etalecharts}
In this section, we will define some \'etale charts of $G_E$. 
The main result in this section is Theorem~\ref{etalesurjective}. Facts about semistable bundles on elliptic curves are collected in Appendix~\ref{ssbundles}.
\subsubsection{Definition and representability}
There are three, mutually commuting actions of $q^\mathbb{Z}, G_s, G$ on 
$\mathbb{C}^* \times G_s^0 \times G $:
$$\xymatrixrowsep{-0.in}\xymatrix{
q \cdot (z,h,g):= (qz, h, h(z)g), & q \in q^{\mathbb{Z}}\\
 k\cdot(z,h,g):= (z,Ad^q_k(h),k(z)g),  &  k \in G_s\\
 g'\cdot(z,h,g):= (z,h,gg'^{-1}),  &  g' \in G
}$$
Put  $\mathscr{P}_s:= (\mathbb{C}^* \times G_s^0 \times G) /q^{\mathbb{Z}}  $.
Then $\mathscr{P}_s$ maps naturally to $E \times G_s^0$ and it is a $G_s$-equivariant (with respect to the twisted conjugation on $G_s$) principal $G$-bundle. Hence $\mathscr{P}_s/G_s \to E \times G^0_s/'G_s$ is a principal $G$-bundle, and therefore it defines a map of stacks $p_s: G_s^0/'G_s \rightarrow \Bun_G(E)$.

Recall the notation $G_E:= \Bun_G^{0,ss}(E)$ for the stack of degree $0$ semistable $G$-bundles.

\begin{prop}
The image of $p_s$ lies in $G_E$.
\end{prop}
\begin{proof}
 Let $x \in |G_s^0/'G_s|$. Then by Corollary~\ref{abundant}, there is a $t \in T$, such that $t \in \overline{\{x\}}$. Hence $p_s(t) \in \overline{\{p_s(x)\}}$. Now $p_s(t) \in G_E$ because it is in the image of $T_E \to G_E$.  Hence $p_s(x) \in G_E$, since $G_E \subset \Bun_G(E)$ is open.
\end{proof}

\begin{prop}
The map $p_s$ is representable. 
\end{prop}
\begin{proof}
Let $G_{E,0}$ be the stack classifying pairs $(\mathcal{P}, \beta)$, where $\mathcal{P}$ is a semistable $G$ bundle of degree 0, and $\beta$ is a trivialization of $\mathcal{P}$ at $0 \in E$.  $G_{E,0}$ is representable by Proposition~\ref{representable}. The group $G$ acts on $G_{E,0}$ by changing the trivialization and $G_{E,0}/G=G_E$. 
There is a natural map $p_s': G_s^0 \rightarrow G_{E,0}$ defined by $\mathscr{P}_s$ with the natural trivialization that identifies the fiber over $0 \in E$ with the fiber over $1 \in \mathbb{C}^*$. The map $p_s'$ is $G_s$-equivariant, where $G_s$ acts on $G_{E,0}$ via $ev_1: G_s \rightarrow G$. So $p_s'$ induces $\overline{p_s'}: 	G_s^0/'G_s \rightarrow G_{E,0}/G_s$. Hence $\overline{p_s'}$ is representable.
 We have the following commutative diagram of stacks:
\[
\xymatrix{
 G_s^0/'G_s      \ar[r]^{\overline{p_s'}}  \ar[rrd]_{p_s}       &       G_{E,0}/G_s  \ar[r]   &   G_{E,0}/G \ar[d]_{\simeq}        \\
                                                                     &                                         &     G_E  
}
\]
By Proposition~\ref{gs}, $ev_1: G_s \rightarrow G$ is injective, so the top arrows are representable and hence $p_s$ is representable.
\end{proof}

\subsubsection{1-shifted symplectic stacks}

In this section, we show that the morphism $p_s$ is a symplectomorphism.
\begin{defn}Let $X$ be a smooth analytic stack,  $TX$ its tangent complex. 
\begin{enumerate} 
\item A \textit{weak 1-shifted symplectic structure} is a 1-shifted non-degenerate $2$-form $\omega_X$, i.e. a non-degenerate $\calO_X-$bilinear antisymmetric pairing 
$$ \omega_X:\xymatrix{TX[-1] \times TX[-1] \ar[r] & \calO_X[-1] }.$$

\item A \textit{symplectomorphism} $f : (X, \omega_X) \rightarrow (Y, \omega_Y)$ between smooth stacks with weak 1-shifted symplectic structure is a morphism of stacks $f: X \rightarrow Y$ together with an isomorphism $f^*\omega_Y \simeq \omega_X$.
\end{enumerate}
\end{defn}

\begin{rmk}
To define the actual shifted symplectic structure, the notion of closed forms is needed and requires a more careful definition, see \cite{PTVV}. The weak version above is sufficient for our purpose. 
For smooth stacks with positive dimensional automorphism group, $n=1$ is the only possible value for a $n$-shifted symplectic structure to exist.
\end{rmk}

Shifted symplectic structures relate the stacky and infinitesimal behaviors:
\begin{prop}
\label{symplectomorphismetale}
Let $f: (X, \omega_X) \to (Y, \omega_Y)$ be a symplectomorphism, and $x \in X$. Assume that $f_x: Aut(x)^0 \to Aut(f(x))^0$ is an isomorphism, then $f $ is \'etale at $x$.
\end{prop}
\begin{proof}
We need to show that $df_x: T_x X \rightarrow T_{f(x)} Y$ is a quasi-isomorphism. The tangent complex is concentrated in degrees $-1,0$ since the stacks are smooth and $1$-truncated. For degree $-1$, we have $H^{-1}(df_x)=d(f_x)$, so it is an isomorphism. For degree 0, the map $H^0(df_x)$ is also an isomorphism since the weak $1$-shifted symplectic structure pairs $H^{-1}$ and $H^0$.
\end{proof}

\begin{eg}
Fix $\kappa $ an invariant non-degenerate bilinear form on $\g$.  For $\calP \in \Bun_G(E)$, we have a natural identification $T_\calP \Bun_G(E)[-1]  \simeq  R\Gamma(E, \g_\calP)$, and $\Bun_G(E)$ (hence $G_E$) has a natural weak 1-shifted symplectic structure given by the Serre duality pairing:
$$\xymatrix{
R\Gamma(E, \g_{\calP_{g(z)}}) \times R\Gamma(E, \g_{\calP_{g(z)}}) \ar[r]_-{\kappa} &
\tau^{\geq 1}R\Gamma(E, \calO_E)}$$
Similarly, $G/G \simeq \Loc_G(S^1)$ has a natural weak 1-shifted symplectic structure given by Poincar\'e duality. In general, \cite{PTVV} shows that $\Bun_G(X)$ has a $2-n$ shifted symplectic structure for $X$ a $n$-dimensional Calabi-Yau manifold and $\Loc_G(M)$ has a $2-n$ shifted symplectic structure for $M$ a $n$-dimensional oriented smooth manifold.
\end{eg}

The uniformization $p: L_{hol}G /' L_{hol} G \to \Bun_G(E)$ can be thought of as a non-linear \v{C}ech resolution associated to the cover $\C^* \to E$, in the sense that, after linearization:
$$dp_{g(z)}: \xymatrix{      T_{g(z)} L_{hol}G/'L_{hol}G[-1] \ar@{=}[d]  \ar[r]^-{\sim}  & T_{\calP_{g(z)}}\Bun_G(E) [-1]       \ar@{=}[d] \\
\{L_{hol}\g \xrightarrow[\phi_{g(z)}]{}  L_{hol} \g \} \ar[r]^-{\sim} & R\Gamma(E, \g_{\calP_{g(z)}}) }$$
the tangent map in the first row can be identified with the \v{C}ech resolution in the second row above, where $\phi_{g(z)}(X(z))=Ad_{ g(z)^{-1} }X(qz)  - X(z)$, and also complexes are (cohomologically) concentrated in degree $0,1$.

 There is a natural pairing:
\[
\kappa:\xymatrix{ \{L_{hol}\g \to  L_{hol} \g \} \times  \{L_{hol}\g \to  L_{hol} \g \} 
\ar[r]  & \tau^{\geq 1}\{L_{hol}\C \to  L_{hol} \C \} \simeq \C }\]
\[
\xymatrix{
(X^\bullet(z), Y^\bullet(z))  \ar@{|->}[r] & \oint \kappa(X^\bullet(z),Y^\bullet(z) ) \frac{dz}{z}
}\]
It follows from definition that this pairing resolve the Serre duality pairing, i.e:
\begin{prop}
\label{pairing}
The diagram naturally commute:
\[
\xymatrix{ \{L_{hol}\g \to  L_{hol} \g \} \times  \{L_{hol}\g \to  L_{hol} \g \}
\ar[r]_-{\kappa} \ar[d]^{\sim} & \{L_{hol}\C \to  L_{hol} \C \} \ar[d]^{\sim}  \\
R\Gamma(E, \g_{\calP_{g(z)}}) \times R\Gamma(E, \g_{\calP_{g(z)}}) \ar[r]_-{\kappa} &
R\Gamma(E, \calO_E)
}\]
\end{prop}

If we view the tangent complex $T_{g(z)}G_s^0/'G_s$ as a subcomplex of $T_{g(z)}L_{hol}G/'L_{hol}G$, then $G_s^0/'G_s$ has an induced 1-shifted 2-form $\omega$.

\begin{prop} 
\label{pssymplecto}
The $1$-shifted $2$-form $\omega$ on $G^0_s/'G_s$ is non-degenerate. And the map $p_s : G_s^0/'G_s \rightarrow G_E$ is a 1-shifted symplectomorphism.
\end{prop}
\begin{proof}
The second statement follows from Proposition~\ref{pairing}. For the first statement, we first prove that the pairing $\oint \kappa(-,-) \frac{dz}{z}:  \xymatrix{\g_s \times \g_s \ar[r]  & \C }$ is non-degenerate. This is because $\g_s=\t \oplus \bigoplus_{\alpha \in \Phi_s} \g_\alpha $, and the pairing pairs $\t$ with $\t$, pairs $\g_\alpha $ with $\g_{-\alpha} $. Now the non-degeneracy of $\omega$ follows from the following tautological Lemma:

\begin{lem}
Let $<-,->: \xymatrix{ V^0 \times V^1 \ar[r] & \C }$ a non-degenerate pairing between finite dimensional vector spaces, and let $\phi: V^0 \to V^1$, such that ${< \Ker(\phi), \Image(\phi)> = 0}$, then the induced pairing 
$$\xymatrix{\Ker(\phi) \times (V^1/\Image(\phi))  \ar[r] &  \C} $$ 
is also non-degenerate.
\end{lem}
\noindent
To complete the proof of Proposition, take $V^0=\g_s, V^1=\g_s$, and $\phi=\phi_{g(z)}$.
\end{proof}

\subsubsection{Etale charts} Let $T_s^{et}:=\{t \in T: G_t \subset G_s\}$, then $T_s^{et}$ is a $W_s$-invariant open subset of $T$, 
also note that $T_s^{et}$ can be computed in terms of root datum and the elliptic parameter $q$ thanks to Corollary~\ref{automorphism}.
Denote $G_s^{0,et}:={\chi'_s} ^{-1}(T_s^{et}//W_s)$.

\begin{thm}
\label{etalesurjective}
The map $p^{et}_s: G_s^{0,et}/'G_s \rightarrow G_E$ is \'etale. And $p^{et,*}_s(\N_{T^*G_E})=\N_{T^*(G_s^{0,et}/'G_s)}.$
\end{thm}                                 
\begin{proof}
We first prove that $p^{et}_s$ is \'etale for $t \in T_s^{et}$. By Proposition~\ref{symplectomorphismetale} and \ref{pssymplecto}, we need to show that $(p_s)_t: Aut(t) \rightarrow Aut(\calP_t)$ is an isomorphism (on the neutral component). This is true because $(p_s)_t$ is identified as $Aut(t)= C_{G_s}(t) = G_s \cap C_{L_{hol}G}(t)= G_s \cap G_t = G_t \xrightarrow{\sim} Aut(\calP_t)$. 
Now the first assertion follows since $T_s^{et}$ is abundant in $|G_s^{0,et}/'G_s|$ and the  \'etale locus is open. For the second statement, for $g(z) \in G_s$, the \v{C}ech complex for $\g^*_{\mathcal{P}_{g(z)}}$ and the first statement give quasi-isomorphisms of complexes in degree $0,1$: $T_{g(z)}^*(G_E) \simeq \{L_{hol}\g^* \to L_{hol} \g^*\} \simeq \{L_{hol}\g \to L_{hol} \g\}^* \simeq \{\g_s \to \g_s\}^* \simeq \{\g^*_s \to \g^*_s\} \simeq T_{g(z)}^*(G^0_s/'G_s)$. Under this identification, $X(z) \in H^0(\g_s^* \to \g_s^*) \simeq_{\kappa} H^0(\g_s \to \g_s)$ is in $\N_{T^*G_E}$ if $X(z) $ is nilpotent in $\g$ for all $z \in \C^*$, and $X(z) $ is in $\N_{T^*(G_s^{0,et}/'G_s)}$ if it is nilpotent as an element in the Lie algebra $\g_s$. Now we see these two notions are equivalent by the explicit formula of $\g_s$ in Proposition~\ref{calculate gs}.  
\end{proof}

View $X_*(T)$ as a subgroup of $LT$, it acts freely on the constant loops $T \subset LT$ via twisted conjugation. We have $T_E \simeq T/'X_*(T) \times BT$. The group $W_\aff=X_*(T) \rtimes W$ acts on $T$, and let $T^{q-reg}$ be the open dense locus where the action of $W_\aff$ is free. Let $G_s^{0,q-reg}:=\chi_s^{-1}(T^{q-reg}//'W_s)$.
Using the identification $T_E^{reg}/W \simeq (T^{q-reg}/'X_*(T) \times BT)/W \simeq  (T^{q-reg} \times BT) /'  W_\aff$, we have a commutative diagram:
\begin{equation}
\label{genericcommutativediagram}
\xymatrix{
(T^{q-reg} \times BT) /'  W_s     \ar[d] \ar[r]^-{\sim}            &     G_s^{0,q-reg}/'G_s  \ar[d]^{p_s}\\
(T^{q-reg} \times BT) /'  W_\aff  \ar[r]^-{\sim}                                     &      G_E^{reg}
}
\end{equation}
Recall the semi-simplification map $\chi_E: G_E \to \fre_E$ as in Appendix~\ref{ssbundles}.
\begin{prop}
\label{pscartesian}
The following commutative diagram is cartesian:
$$\xymatrix{ G_s^{0,et}/'G_s \ar[r]^-{\chi'_s} \ar[d]^{p_s} \ar@{}[rd]|-{\square} &  T^{et}_s //' W_s \ar[d]^{} \\
              G_E        \ar[r]^-{\chi_E}   &  \fre_E \simeq  T//' W_\aff 
}$$ 
\end{prop}

\begin{proof}
 Suffices to show for each small open $U \subset T_s^{et}//'W_s$, the diagram obtained by restricting to $U$ is cartesian :
$$\xymatrix{ {\chi'_s}^{-1}(U) \ar[r]^{\chi'_s} \ar[d]^{p_U} \ar@{}[rd]|-{\square} &  U \ar[d]^{q} \\
               G_E        \ar[r]^{\chi_E}   &   T//'W_\aff
}$$ Assume $U$ is small so that $q$ is an open embedding. Let $\widetilde{U}$ be the preimage of $U$ in $T^{et}_s$.  Now by (\ref{genericcommutativediagram}), $p_U|_{{\chi'_s}^{-1}(U \cap (T^{q-reg}/'W_s))}$ is identified with the composition 
$$((\widetilde{U} \cap T^{q-reg}) \times BT)/'W_s \to (T^{q-reg} \times BT)/'W_\aff  \to G_E$$ 
which is an open embedding by the choice of $U$.
Therefore $p_U$ is generically open embedding, and it is also \'etale, so by Lemma~\ref{openembedding}, the map $p_U$ is an open embedding. Now we need to check the image of $p_U$ equal $\chi_E^{-1}( q(U))$. This is because the image contains all the semi-simple bundles in $\chi_E^{-1}( q(U))$ by construction and hence consist all $\chi_E^{-1}( q(U))$ by Proposition~\ref{semisimpleabundant}.\end{proof}

\subsection{Gluing of charts}
\label{glueing}
In this section, we will glue the charts defined in Section~\ref{etalecharts}, i.e we will calculate the fiber products of the charts. The combinatorics of higher descent data is naturally organized in diagrams introduced in Section~\ref{descentcategory}. The main result of this section is Theorem~\ref{colimitdiagram}.

For $\textbf{s}= (s_1,s_2,...,s_k) \in T^k$, let $G_{\textbf{s}}:=\bigcap_{i=1}^k G_{s_i}$, $W_\textbf{s}:=N_{G_{\textbf{s}}}(T)/T$, define $\chi_\textbf{s}, T_\textbf{s}^{et}, G_\textbf{s}^{et}$ analogously. All the above statements for $G_s$ still holds for $G_\textbf{s}$, and moreover $W_\textbf{s}=\bigcap_{i=1}^k W_{s_i}, T_\textbf{s}^{et}=\bigcap_{i=1}^k T_{s_i}^{et}$
, $T^{q-reg} \subset T_\textbf{s}^{et}$ for all $\textbf{s}$ by the connectedness of $G$.

\begin{prop}
For any $w \in N_G(T)\cdot X_*(T) \subset LG$, the twisted conjugation $Ad'_{w}: G_\bs^0 \xrightarrow{\sim} G_{w(\bs)}^0$ intertwines the action $Ad_{{w}} :G_\bs \xrightarrow{\sim} G_{w(\bs)}$. Hence we have an isomorphism of stacks $Ad'_{{w}}: G_\bs^0/'G_\bs \xrightarrow{\sim} G_{w(\bs)}^0/'G_{w(\bs)}$.
\end{prop}

\begin{proof}
 $G_{{w}(s)}=C_{LG}({w}(s))=Ad_{{w}}C_{LG}(s) $, so we have a isomorphism of algebraic groups: $Ad_{{w}}: G_s \rightarrow G_{w(s)}$.

  Write $w=u\lambda,$ for $u \in N_G(T),$ and $\lambda \in X_*(T)$.   Using the fact that $\lambda(q) \in T \subset G_s$, we have  
  $$Ad_{{w}}'(G_s)={u}\lambda(qz)G_s\lambda(z)^{-1}{u}^{-1}={u}\lambda(z)\lambda(q)G_s \lambda(z)^{-1}{u}^{-1}=Ad_{{w}}(G_s)=G_{w(s)}.$$ 
Since $Ad_{{w}}'$ stablize $T$, we have $Ad_{{w}}': G_s^0 \rightarrow G_{w(s)}^0$ isomorphism of algebraic varieties. 
The pair of isomorphisms of algebraic varieties and algebraic groups $(Ad_{{w}}', Ad_{{w}}):(G_s^0, G_s) \rightarrow (G_{w(s)}^0, G_{w(s)})$ intertwine the twisted conjugation action on both sides. Hence we have an induced isomorphism of quotient stacks, still denoted by $Ad_{{w}}': G_s^0/G_s \rightarrow G_{w(s)}^0/G_{w(s)}$. It's also easy to see that the above map takes \'etale locus to \'etale locus, so we have $Ad_{{w}}':G_s^{0,et}/G_s \rightarrow G_{w(s)}^{0,et}/G_{w(s)}$.
\end{proof}

\begin{defn}
\label{descentdiagramstack}

Let $\dot{W}_\aff \subset N_G(T) \cdot X_*(T)$ be a subgroup such that the map $\dot{W}_\aff \to  (N_G(T) \cdot X_*(T)) /T =W_\aff$ is surjective.
Let $S \subset T$ be a $W_\aff$ (or equivalently $\dot{W}_\aff$) invariant subset, and let $\{V_s,s\in S\}$ be a collection of open subsets of $T$ satisfying $V_{w(s)}=w(V_s)$, for all $w \in W_\aff$. Let $V_\bs := \bigcap_{s \in \bs} V_s$, $U_{\bs}:={\chi'_{\bs}}^{-1}(V_\bs//W_\bs)$.

 Then we have a functor 
$$ \xymatrix{ \bU: \int^{\Delta^{op}} S^\bullet_{\dot{W}_\aff}  \ar[r] & \Stk_\oo  }$$
defined similarly to Construction~\ref{egxreg}  by:
\begin{enumerate}
\item $\bU(\bs):= U_\bs/'G_\bs$;
\item $\bU(w ):= Ad'_{{w}}: \xymatrix{U_\bs/'G_\bs \ar[r]^-{\sim} & V_{w(\bs)}/'G_{w(\bs)}};$
\item $\bU(w'w^{-1}):= \eta_{{w'} {w^{-1}}} \circ Ad'_{{w}} :
 \xymatrix{Ad'_{{w}}  \ar@{=>}[r] &  Ad'_{{w'}} }$;
 \item $\bU(\delta: \bs \to \bs' ):=  i: \xymatrix {U_\bs/'G_\bs \ar[r] &  U_{\bs'}/'G_{\bs'}    }$.
\end{enumerate}

The augmentation morphisms $p_s$ and 2-morphisms $\varphi_{{w}}$ defined below 
extends the functor $\bU$ to  \\
$\xymatrix{ \bU_{+}:\int^{\Delta^{op,\triangleright}} S^\bullet_{\dot{W}_\aff} \ar[r] & \Stk_{\oo}},$ by sending the final object to $G_E$:

There is commutative diagram:
\[
\xymatrix{
\mathbb{C} \times G^0_s \times G \ar[r]^{{\varphi}_{w}}_{\simeq}  \ar[d]
&  \mathbb{C} \times G^0_{w(s)} \times G \ar[d] \\
\mathbb{C} \times G^0_s    \ar[r]^{Id \times Ad'_{\dot{w}}}_{\simeq}  &  \mathbb{C} \times G^0_{w(s)}
}
\]
where  $\varphi_{{w}}(z,h,g):= (z,Ad'_{{w}}(h), {w}(z)g)$. The diagram is $q^\mathbb{Z}$-equivariant and hence induces:
\[
\xymatrix{
\mathscr{P}_s \ar[r]^{{\varphi}_{{w}}}_{\simeq}  \ar[d] & \mathscr{P}_{w(s)} \ar[d] \\
E \times G^0_s    \ar[r]^{Id \times Ad^q_{{w}}}_{\simeq}  &  E \times G^0_{w(s)}
}
\]

Hence and induces $\varphi_{{w}}: p_s \Rightarrow p_{w(s)} \circ Ad'_{{w}}: G^0_s/G_s \rightarrow G_E$  an isomorphism between the morphisms of stacks.
\end{defn}	

We have the main theorem of this section:

\begin{thm} 
\label{colimitdiagram}
Assume $V_s \subset T_s^{et}$ and $\bigcup_{s \in S} V_s =T$, then
\begin{enumerate}
	\item  the natural map in $\Stk$  is an isomorphism :
 $$\xymatrix{\colim_{\int^{\Delta^{op}} S^\bullet_{\dot{W}_\aff}} \bU \ar[r]^-{\sim}  & G_E}$$  
 \item the natural map in $\Cat_\oo$ is an equivalence:
  $$\xymatrix{\lim_{({\int^{\Delta^{op}} S^\bullet_{\dot{W}_\aff}})^{op}} Sh_{\N}(U_\bs/'G_\bs) \ar@{<-}[r]^-{\sim}  & Sh_\N(G_E)  }$$  
 \end{enumerate}
\end{thm}
\begin{proof} 
 By Construction~\ref{egxreg}, we have 
 $\xymatrix{ \bV_{+}:{\int^{\Delta^{op,\triangleright}} S^\bullet_{\dot{W}_\aff}} \ar[r] & \mathscr{S} }$, 
which is a colimit diagram by Proposition~\ref{xcolimit}. 
The character polynomial maps $\chi_\bs$ and $\chi_E$ give a natural transformation $\chi: \bU_{+}(\C) \Rightarrow \bV_{+}$, which is cartesian by argument similar to Proposition~\ref{pscartesian} (note that $T//'W_\aff = T//'\dot{W}_\aff$). Hence the functor $\bU_{+}(\C)$ is a colimit diagram since colimit in $\mathscr{S}$ is stable under base change. Hence we conclude by Theorem~\ref{etalesurjective} and Proposition~\ref{descentss}.
\end{proof}

\subsubsection{A Lie theoretic choice of charts for simply-connected groups}

In this section, we will simplify the previous general discussions to concrete Lie theoretic data involving alcove geometry. We assume $G$ is simply-connected throughout this section.  \\

Choose $\tau \in \mathcal{H},$ such that $q=\exp(2\pi i \tau)$. 
The identification $\Z \simeq \Z \tau$ gives 
$\t_\R= X_*(T) \otimes \R  \simeq X_*(T) \otimes \R \tau$. And hence gives a natural wall stratification on $X_*(T) \otimes \R \tau $. 
Under the identification $\C= \R \times  \R \tau$, we have:
$$(X_*(T) \otimes \R/\Z) \times (X_*(T) \otimes \R \tau)= X_*(T) \otimes \C/\Z \xrightarrow{\Exp:=\exp(2\pi i -)} X_*(T) \otimes \C^*=T. $$ 
Note that the restriction of  exponential map: $X_*(T) \otimes \R \tau  \rightarrow T$ is an embedding. The groups defined in Section~\ref{Lietheoretic} and in Section~\ref{automorphismgroups} coincide under this embedding:
\begin{prop}
\label{maximalgs}
\begin{enumerate}
\item For $a \in \t_\R$, we have $G_a = G_{\Exp(0, a \tau)}$.
\item For $a \in \t_\R$, and $\theta \in X_*(T) \otimes \R/\Z,$ we have
 $G_{\Exp(\theta,a \tau)} \subset G_{\Exp(0,a \tau)}$.
\end{enumerate}
\end{prop}
\begin{proof}
There is a commutative diagram:
$$
\xymatrixcolsep{5pc}
\xymatrix{
(X_*(T) \otimes \R/\Z) \times (X_*(T) \otimes \R \tau) \ar[r]^-{\Exp}_-{W_\aff \text{-equivariant}} \ar[d]^{\alpha} & X_*(T) \otimes \C^* \ar[d]^{\alpha} \\
\R/\Z  \times \R \tau  \ar[r]^-{\Exp}_{\sim} &  \C^*  \\
 \{0\} \times \Z\tau \ar@{^{(}->}[u]  \ar[r]^-{\Exp}_-{\sim}  &   q^\Z \ar@{^{(}->}[u]
 }
$$
Denote $\Phi_\theta:=\{ \alpha=\alpha_0 -n \in \Phi_\aff \,|\, \alpha_0(\theta)=0 \}$ and $W_\theta:=C_{W_\aff}(\theta)$. For $s:=\Exp(\theta, a\tau)$,  then $\Phi_s= \Phi_\theta \cap \Phi_a$ as subset of $\Phi_\aff$ and $W_s =W_\theta \cap W_a$ as group of $W_\aff$. Hence $(1),(2)$ follow since for $\theta=0$, $\Phi_{\theta}= \Phi_\aff$ and $W_\theta=W_\aff$, c.f Proposition~\ref{connected}.
\end{proof}

Now we assume that $G$ is simply-connected.

\begin{cor}
For $s=\Exp(0,a\tau)$, the group $G_s$ is connected. 
\end{cor}

\begin{rmk}
For general $s$,  the group $G_s$ may not be connected, a counter-example is given in \cite{BEG}. 
\end{rmk}

Denote by $T_J^{se}:= (X_*(T) \otimes \R/\Z) \times St_J \cdot \tau  \subset T$ and $G_J^{se}:= {\chi'_J}^{-1} (T_J^{se}//'W_J) \subset G_J$ be the set of elements with ``small eigenvalues''. 

\begin{prop}
\label{seinet}
The subset $T_J^{se}$ is contained in $T^{et}_J$. 
\end{prop}
\begin{proof}
Need to prove that for any $s \in T^{et}_J$, the group $G_s $ is contained in $G_J$. By Proposition~\ref{maximalgs}, we can assume $s= (0, a\tau)$, i.e we need to prove $G_a \subset G_J$ for $a \in St_J $, and this can be easily checked.
\end{proof}

\begin{prop} 
\label{colimtorus} 
There are isomorphisms in $\mathscr{S}$:
\begin{enumerate} 
	\item $ \xymatrix{\colim_{J \in \mathscr{F}_C^{op}} T^{se}_J//'W_J  \ar[r]^-{\sim}   & T//'W_{\aff}}$
	\item $ \xymatrix{\colim_{J \in \mathscr{F}_C^{op}} G^{se}_J/'G_J (\C) \ar[r]^-{\sim}   & G_E(\C)}$
\end{enumerate}
\end{prop}

\begin{proof}
(1) By Proposition~\ref{opencolimit}, we have $ \colim_{J \in \mathscr{F}_C} St_J\tau \times_{W_J} W_\aff \simeq \t_\R \tau ,$ as $W_\aff$-set. Multiply the $W_\aff$-set $X_*(T) \otimes \R/\Z$  on both sides yields  $\colim_{J \in \mathscr{F}_C} (X_*(T) \otimes \R/\Z) \times St_J \tau \times_{W_J} W_\aff \simeq (X_*(T) \otimes \R/\Z) \times \t_\R\tau $ compatible with the diagonal $W_\aff$ action. Now dividing $W_\aff$ on both sides gives $\colim_{J \in \mathscr{F}_C} T^{se}_J/W_J \simeq T/W_{\aff}$. 	For $s= (\theta, a \tau) \in T_J^{se}, w \in W_\aff$, such that $w(s)=s$, we have $w(a)=a$, and $a \in St_J$, hence $w \in W_J$. So the maps $T^{se}_J/W_J \to T/W_\aff$ are fully-faithful, then we get (1) by take $\pi_0$ of the previous equivalence. (2) Follows from (1) and Proposition~\ref{pscartesian}. 
\end{proof}	

\begin{thm}
\label{colimitdiagram2} \begin{enumerate}
	\item 
There is an isomorphism of stacks: 
$$
\xymatrix{
\colim_{J \in \mathscr{F}_C} G^{se}_J/'G_J \ar[r]^-\sim & G_E
}
$$
\item There is an equivalence of $\oo$-categories:
$$
\xymatrix{
	\lim_{J \in \mathscr{F}_C} Sh_{\N}(G^{se}_J/'G_J) & Sh_{\N}(G_E)  \ar[l]_-\sim
}
$$
\end{enumerate}
\end{thm}

\begin{proof}
	The functor $\mathscr{F}^{\triangleright}_C \to \Stk$ (via $J \mapsto G_J^{se}/'G_J$ and $* \mapsto G_E$) satisfies the assumption of Theorem~\ref{descent}: (i) $G^{se}_J/'G_J$ and $G_E$ are analytic stacks. (ii) The maps are \'etale by Proposition~\ref{seinet}. (iii) By Proposition~\ref{colimtorus} (2).
\end{proof}	

\begin{rmk} The locus of small eigenvalues $G_J^{se}$ depends on the choice of $\tau$. Nevertheless, as we will see later in Corollary~\ref{levigroup}, the category of  sheaves with nilpotent singular support $Sh_\N(G_J^{se}/'G_J)$ does not depends on $\tau$ and it is equivalent to $Sh_\N(G_J/G_J)$.
\end{rmk}

\subsubsection{A Lie theoretic choice of charts for general reductive groups}
\label{reductivecharts}
Now assume that $G$ is a connected reductive group. We write $G = (Z^0(G) \times \widetilde{G}_{\textup{der}})/F$, where $\widetilde{G}_{{\der}}$ is the simply-connected cover for the derived group of $G$, and $F$ is some central finite group. Denote by $\widetilde{T}_\der$ the corresponding torus in $\widetilde{G}_\der$, and $\widetilde{\t}_\der$ its Lie algebra. Let $S_0 \subset \widetilde{\t}_\der$ the set of vertices of the affine alcoves in $\widetilde{\t}_{\der,\R}$. Denote by $S:=(X_*(Z^0(G)) \times S_0)/F \subset \t_\R$, for any $s =[(c,s_0)] \in S$, the open subset $V_{\R,s}:= (\frz_\g \times St_{x_0})/F \subset \t_\R$ is independent of the choice of the representatives of $s$. Put $V_s = X_*(T) \otimes \R/\Z \times V_{\R,s}  \subset T$. Then the collection $\{V_s, s \in S \}$ satisfies the assumption of Definition~\ref{descentdiagramstack} and Theorem~\ref{colimitdiagram}.


\subsection{Complex gauge theory on $S^1$}

In this section, we study the stack $G/G \simeq \Loc_G(S^1)$ using 
the gauge uniformization on $S^1$. We will establish the results in previous sections in the present situation. It can be viewed as a nonabelian analog of the uniformization $\C \to \C^*=\C/\Z$. Many of the proofs are similar as before, we shall only highlight some differences in the present situation.  

Denote by $\usG$ the trivial $G$-bundle on $S^1$,  by $\mathcal{A}(\usG)$ the space of connections on $\usG$.  and by $\Conn_G(S^1)$ the moduli stack of smooth $G$-bundles on $S^1$ with connection. Since every $G$-bundle on $S^1$ is trivial, we have an isomorphism of groupoids $\Conn_G(S^1)(pt) = \mathcal{A}(\usG)/ \Aut(\usG)$. We have an identification $\Aut(\usG) \simeq C^{\infty}(S^1,G)=: L_{sm}G$. The trivial connection on $\usG$ gives $\mathcal{A}(\usG) \simeq \Omega^1(S^1, \mathfrak{g})$. Fix $z \in C^{\infty}(S^1,\C^*)$ a degree $1$ map, such that $dz$ is nowhere vanishing. (For example, take $S^1$ to be the unit circle with angle coordinate $\theta$ and $z=e^{i\theta}$). Then we have a identification $- \wedge d \log(z): L_{sm}\g:=C^{\infty}(S^1,\g) \xrightarrow{\sim} \Omega^1(S^1,\g)$.

We have 
\[\Conn_G({S^1})(pt)=L_{sm}\g /'  L_{sm}G  \]

And the action above of $L_{sm}G$ on $L_{sm}\g$ is identified with the gauge transformation (twisted adjoint action): $\ad'_g(a) := gag^{-1} - \frac{dg}{d\log(z)}  \cdot g^{-1}  $ for $g \in L_{sm}G, a \in L_{sm}\g$.

We have $X_*(T) \to L_{sm}T$ via $\lambda \mapsto \lambda\circ z$, then $\frt \subset L_{sm}\frt$ is stable under the gauge action of $X_*(T)$ and the action is identified as translation under $X_*(T) \hookrightarrow X_*(T) \otimes \C \simeq \frt$, where the last isomorphism is given by $(\lambda,c) \mapsto d\lambda(c)$. The group $\dot{W}_\aff \subset L_{sm}G$ acts on $\t$ via gauge transformation, this action factor through $W_\aff$. When restricted to $\t_\R$, this action of $W_\aff$ equal to the one in Section~\ref{Lietheoretic}.

For $a \in \frt \subset L_{sm}\g,$ put $\Phi_a:=\{ \alpha \in \Phi_\aff \;| \; \alpha(a)=0 \}$, $W_a:=C_{W_\aff}(A)$, $G_A:= C_{L_{sm}G}(a)$ and $ \g_a:=Lie(G_a)$. For $\bfa= (a_i) \in \frt^n$, let $G_\bfa := \bigcap_{i=1}^n G_{a_i}, \g_{{\bfa}}:= \bigcap_{i=1}^n \g_{a_i} , \Phi_\bfa := \bigcap_{i=1}^n \Phi_{a_i},  $ and $W_\bfa:=\bigcap_{i=1}^n W_{a_i}.$

\begin{thm}
	\label{covergroup}
	\begin{enumerate}
		\item $G_\bfa^0= <T, \exp  \g_\alpha \,|\, \alpha \in \Phi_\bfa>$.
		\item $G_\bfa=<G^0_\bfa, \dot{w} \,|\, w \in W_\bfa>, $ where $W_\bfa=C_{W_\aff}(\bfa)$. In particular, it agree with $G_a$ in Definition~\ref{parabolicfacet} for $\bfa= (a) \in \t_\R$.
		
		\item The space $\g_\bfa$ is stable under the gauge transformation of $G_\bfa$. The translation by $-a$ gives an isomorphism of stacks $-a: \g_a/'G_a \xrightarrow{\sim}  \g_a/G_a$, where the later action is adjoint action.
		
		\item Let $\chi'_a : \g_a \to \t_a//W_a$ the characteristic polynomial map with respect to the gauge action. Let $ \t_a^{et}:=\{ x \in \t \,|\, W_x \subset W_a, \Phi_x \subset \Phi_a \}  $, and $\g^{et}_a := {\chi'_a}^{-1}(\t_a^{et}//W_a)$. Then the natural map $p^{et}_a: \g^{et}_a/G_a \to \Loc_G(S^1)$ is (representable) \'etale. And $p^{et,*}_a(\N)=\N$.
		
		\item Let $S \subset \t$ be a $W_\aff$-invariant subset, for each $a \in S$, let $V_a \subset \t_a$ be $W_a$-invariant open subset, satisfying $V_{w(a)}=w(V_a)$ for all $w \in W_\aff$. Let $V_\bfa:= \cap_{a \in \bfa} V_a$, and $U_\bfa:= \chi_\bfa^{-1}(V_\bfa//W_\bfa)$, then we have a functor by sending $\bfa$ to $U_\bfa/'G_\bfa$ and $*$ to $\Loc_G(S^1)$:
		
		$$\xymatrix{ \bU: \int^{\Delta^{op,\triangleright}} S^\bullet_{\dot{W}_\aff}  \ar[r] & \Stk} $$
		
		Assume further more that $V_\bfa \subset \t_\bfa^{et}$, and $\bigcup_{a \in S} V_a=\frt$, then the induced map is an isomorphism:
		$$\xymatrix{ \colim_{\int^{\Delta^{op}} S^\bullet_{\dot{W}_\aff} } \bU \ar[r]^-{\sim} & \Loc_G(S^1) \simeq G/G  }$$

		\item  Assume that $G$ is simply-connected, let $\t_J^{se}:=St_J \times i \t_\R \subset \t$, and $\g_J^{se}:=\chi^{-1}(\t_J^{se}//W_J)$. Then there is an isomorphism:
		$$\xymatrix{\colim_{J \in \mathscr F_C^{op}} \g_J^{se}/'G_J \ar[r]^-{\sim} & \Loc_G(S^1) \simeq G/G }$$
		
	\end{enumerate}
\end{thm}
\begin{proof}
	(1) This is similar to Corollary~\ref{Gs0}. We have $\g_a = \{x \in C^{\infty}(S^1,\mathfrak{g}): dx+[a,x] \wedge d \log(z)=0 \}$, Let $x=h + \sum_{\alpha \in \Phi}f_{\alpha} x_{\alpha}$, where $h: S^1 \rightarrow \mathfrak{t}, f_{\alpha}: S^1 \rightarrow \mathbb{C}$. 
	Then the equation $dx+[a,x] \wedge d \log(z)=0$ is equivalent to $dh=0$ and $df_{\alpha}=  \alpha(a) f_{\alpha} \wedge d \log(z)$. The first equation has solution constant functions.  
	The second equation has a nontrivial solution only when $\alpha(a) \in \Z$, and in this case,  the solutions are $f_\alpha=  cz^{\alpha(a)}, c \in \CC$. \\
	(2) Similar to Corollary~\ref{automorphism}, where we need a version of Lemma~\ref{twist: normalizer gs}, with $L_{hol}$ replaced by $L_{sm}$, and $G_s$ by $G_a$. \\
	(3) Similar to Proposition~\ref{twistedconjugation}. As a remark, the map $-a: \g_a/'G_a \rightarrow \g_a/G_a$ can be thought of as untwisting the gauge transformation. Since the gauge transformation is an affine linear action, and the action of $G_a$ fix $a$, so re-center the affine space $\g_a$ at $a$ will make the action of $G_a$ a linear action (in fact adjoint action).\\
	(4) Similar to Theorem~\ref{etalesurjective}. The $1$-shifted symplectic structure on $\Loc_G(S^1)$ is used.\\
	(5) Similar to Theorem~\ref{colimitdiagram}.  \\
	(6) Similar to Theorem~\ref{colimitdiagram2}(1).
\end{proof}

From Theorem~\ref{covergroup}(4)(5)(6), and Proposition~\ref{descentss}, we have:
\begin{thm} 
	\label{sesheaves}
	There is an equivalence:
	$$(1) \xymatrix{ \lim_{(\int^{\Delta^{op}} S^\bullet_{\dot{W}_\aff})^{op} } Sh_\N(U_\bfa/'G_\bfa)  & Sh_\N(\Loc_G(S^1)) \simeq Sh_\N(G/G)  \ar[l]_-{\sim} \  }$$
	And for $G$ simply-connected:
	$$ (2) \xymatrix{\lim_{J \in \mathscr F_C} Sh_\N(\g_J^{se}/'G_J) & Sh_\N(\Loc_G(S^1))  \simeq  Sh_\N(G/G)     \ar[l]_-{\sim}  }$$
\end{thm}

\subsection{Holomorphic gauge theory on elliptic curves}
\label{elliptichol}
 Let $\omega= (\omega_1,\omega_2)$ be a pair of complex numbers not contained in the same real line. $E =E_\omega:= \C/(\Z \omega_1 \oplus \Z \omega_2)$ an elliptic curve. As we shall established below, similar to previous sections, the holomorphic gauge uniformization gives an nonabelian analogue of the uniformization $\C \to E$. Results in this section will not be used in our main theorem.

\subsubsection{Holomorphic gauge uniformization on $E$.}
Denote by $\usG$ the trivial smooth $G$-bundle on $E$,  by $\mathcal{A}^{0,1}(\usG)$ the space of (0,1)-connections on $\usG$. Any such connection $\nabla$ defines a holomorphic structure on $\usG$ by defining the holomorphic sections are those section $s$ satisfying $\nabla(s)=0$. Since every degree 0 holomorphic $G$-bundle on $E$ is trivial as smooth bundle, we have an isomorphism of groupoids $\Bun_G^0(E)(pt) = \mathcal{A}^{0,1}(\usG)/ \Aut(\usG)$. We have an identification $\Aut(\usG) \simeq C^{\infty}(E,G)$. The $\bar{\partial}$ operator and the (0,1)-form $d\bar{z}$ give identifications $\mathcal{A}^{0,1}(\usG) \simeq \Omega^{0,1}(E, \mathfrak{g}) \simeq C^{\infty}(E,\g)$. 
Hence we have 
\[\Bun_G^0(E)(pt)=C^{\infty}(E,\g) /'  C^{\infty} (E,G)  \]

And the action above is identified with the Gauge transformation: $\ad'_g(b) := gbg^{-1} - \bar{\partial} g  \cdot g^{-1}  $ for $g \in C^{\infty} (E,G),$ and $b \in C^{\infty}(E,\g)$. 

Let $ \S_1, \S_2$ be two copies of the unit circle. We have isomorphism of Lie groups $\S_1 \times \S_2 \xrightarrow{\simeq} E$, by $(\theta_1, \theta_2) \mapsto  \frac{\omega_1 \theta_1 + \omega_2 \theta_2}{2 \pi}$. This induces $X_*(T) \times X_*(T) \simeq \Hom_{Lie} (E, T)$. An easy calculation shows that under the identification 
$$\xymatrix { \t_\R \times \t_\R  \ar[r]^-{\sim} &   \t, &   (A_1,A_2) \ar@{|->}[r] & \frac{-2\pi i}{\omega_1 \bar{\omega}_2 - \bar{\omega}_1 \omega_2}  (\omega_2 A_1 - \omega_1 A_2)}   $$
The translation of $X_*(T) \times X_*(T)$ on $ \t_\R \times \t_\R $ is identified with the gauge transformation of $\Hom_{Lie} (E, T) \subset C^{\infty}(E,T)$ on $\t \subset C^{\infty}(E,\t)$ (as constant maps). Let $\Phi_{\ellp}:= \Z \times \Z \times \Phi$, and for any $\alpha= (n_1,n_2,\alpha_0) \in \Phi$, define $\g_\alpha:=e^{i(n_1\theta_1+ n_2\theta_2)}\g_{\alpha_0} \subset C^{\infty}(E,\g)$, and such $\alpha$ defines a map $\alpha: \t_\R \times \t_\R \to \R \times \R,$ by $(a_1,a_2) \mapsto (\alpha(a_1)+n_1, \alpha(a_2) + n_2 )$. And define $W_\ellp:=  (X_*(T) \times X_*(T)) \rtimes W$, and $\dot{W}_\ellp \subset (X_*(T) \times X_*(T)) \rtimes N_G(T)$ be a subgroup, such that $\dot{W}_\ellp \to W_\ellp$ is surjective. For $b= (b_1,b_2)$ under the identification, define $\Phi_b:=\{\alpha \in \Phi_\ellp \;|\; \alpha(a_1,a_2)=0 \}$, and $W_b:=C_{W_\ellp}(a_1,a_2)$.

	For $b= (a_1,a_2) \in \frt \subset C^{\infty}(E,\g) $, let $G_b:= G^{\omega}_b:= C_{C^{\infty}(E_\omega,G)}(b)$ the stabilizer under the gauge transformation, and $ \g_b:=Lie(G_b)$.  For $\bfb= (b_i) \in \frt^n$, let $G_\bfb := \bigcap_{i=1}^n G_{b_i}, \g_{{\bfb}}:= \bigcap_{i=1}^n \g_{b_i} , \Phi_\bfb := \bigcap_{i=1}^n \Phi_{b_i},  $ and $W_\bfb:=\bigcap_{i=1}^n W_{b_i}.$ The following theorem is analogous to Theorem~\ref{covergroup}, we shall omit the proof.

\begin{thm}
	\label{gaugecoverge}

	\begin{enumerate}
		\item $G_\bfb^0= <T, \exp  \g_\alpha \,|\, \alpha \in \Phi_\bfb>$.
		\item $G_\B=<G^0_\bfb, \dot{w} \,|\, w \in W_\bfb>.$

		\item Let $\chi'_\bfb : \g_\bfb \to \t_\bfb//'W_\B$ the characteristic polynomial map with respect to the gauge action. Let $ \t_\bfb^{et}:=\{ x \in \t \,|\, W_x \subset W_\bfb, \Phi_x \subset \Phi_\bfb \}  $, and $\g^{et}_\bfb := {\chi'_\bfb}^{-1}(\t_\bfb^{et}//'W_\bfb)$. Then the natural map $p_\bfb: \g^{et}_\bfb/'G_\bfb \to \Bun_G^0(E)$ is (representable) \'etale.
		
		\item Let $S \subset \t$ be a $W_\ellp$-invariant subset, for each $B \in S$, let $V_b \subset \t_b^{et}$ be $W_b$-invariant open subset, satisfying $V_{w(b)}=w(V_b)$ for all $w \in W_\ellp$. Let $V_\bfb:= \cap_{b \in \bfb} V_b$, and $U_\bfb:= {\chi'_\bfb}^{-1}(V_\bfb//'W_\bfb)$, then we have a functor by sending $\bfb$ to $U_\bfb/'G_\bfb$ and $pt$ to $G_E$:
		
		$$\xymatrix{ \bU: \int^{\Delta^{op,\triangleright}} S^\bullet_{\dot{W}_\ellp}  \ar[r] & \Stk} $$
		
		Assume further more that $V_\bfb \subset \t_\bfb^{et}$, and $\bigcup_{b \in S} V_b=\frt$, then the induced map is an isomorphism:
		$$\xymatrix{ \colim_{\int^{\Delta^{op}} S^\bullet_{\dot{W}_\ellp} } \bU \ar[r]^-{\sim} & G_E }$$
		
		\item There is induced equivalence	$$\xymatrix{ \lim_{(\int^{\Delta^{op}} S^\bullet_{\dot{W}_\ellp} )^{op} }Sh_\N(U_\bfb/'G_\bfb)  & Sh_\N(G_E)  \ar[l]_-{\sim} \  }$$

	\end{enumerate}
\end{thm}

Recall that the points in coarse moduli  $\fre_E$ can be identified with the set of isomorphism classes of degree $0$ semisimple $G$-bundles. 
\begin{cor}
	\label{smoothnessfre}
	Let $\mathcal{P}$ be a point in $\fre_E$, assume that $\Aut(\mathcal{P})$ is connected. Then $\fre_E$ is smooth at $\mathcal{P}$.
\end{cor}
\begin{proof}
	By Theorem~\ref{gaugecoverge}, near $\mathcal{P}$, the stack $G_E$ is locally isomorphic to the quotient stack $Lie(\Aut(\mathcal{P}))/\Aut(\mathcal{P})$ near $0$. When $\Aut(\mathcal{P})$ is connected reductive, the coarse moduli of the later stack is smooth (in fact, an affine space).
\end{proof}	

\begin{rmk}
	\begin{enumerate}
		\item By a theorem of Looijenga, $\fre_E$ is isomorphic to a weighted projective space (with explicit weights depending on the root datum), and it is not always smooth.	\
		\item It is possible to deduce Corollary~\ref{smoothnessfre}  from some general slicing theorem such as in \cite{AHR}. 
	\end{enumerate}
\end{rmk}	

\subsubsection{Relation with gauge uniformization on circle}

\begin{no}
	We denote by $G^i_a, \Phi^i_\aff, W_\aff^i$ the corresponding notation associated to $\S_i, i=1,2.$
\end{no}

The inclusion $\xymatrixcolsep{0.1 in}
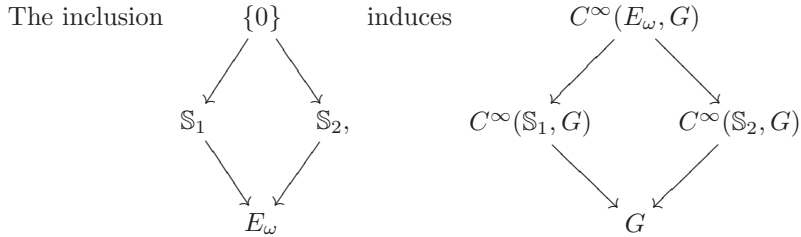
\xymatrix{  & {\{0\}} \ar[ld] \ar[rd] \\
	\S_1  \ar[rd]  &  &  \S_2, \ar[ld] \\
	&   E_\omega
} $
induces $\xymatrixcolsep{-0.2 in}
\xymatrix{  & C^{\infty}(E_\omega,G) \ar[ld] \ar[rd] \\
	C^{\infty}(\S_1,G)  \ar[rd]  &  &  C^{\infty}(\S_2,G) \ar[ld] \\
	&   G
}$
\begin{prop}
	Under the above map, let $\bfb= (\bfa_1,\bfa_2)$ we have  
	\[\xymatrixcolsep{0.1 in}
	
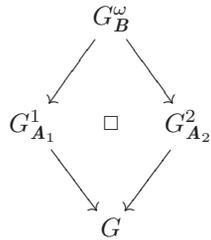
\xymatrix{  & G^{\omega}_\bfb \ar[ld] \ar[rd] \ar@{}|{\square}[dd]\\
		G^1_{\bfa_1} \ar[rd]  &  &  G^2_{\bfa_2} \ar[ld] \\
		&   G
	}\]
	and all the arrows are injective.
\end{prop}
\begin{proof}
	It is easy to check that 
	\[
	\xymatrixcolsep{0. in}
	\xymatrixrowsep{0.2 in}
	
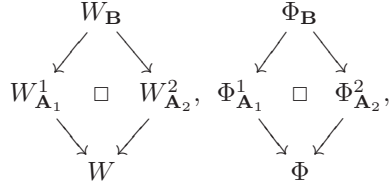
\xymatrix{  & W_{\bfb} \ar[ld] \ar[rd] \ar@{}|{\square}[dd]\\
		W^1_{\bfa_1}  \ar[rd]  &  &  W^2_{\bfa_2}, \ar[ld] \\
		&   W
	} 
	\xymatrix{  & \Phi_{\bfb} \ar[ld] \ar[rd] \ar@{}|{\square}[dd]\\
		\Phi^1_{\bfa_1}  \ar[rd]  &  &  \Phi^2_{\bfa_2}, \ar[ld] \\
		&   \Phi
	}
	\]
	The proposition follows since the groups involved are determined by the above data.
\end{proof}

\begin{rmk}
	Let $G_c$ be a maximal compact subgroup of $G$. Under Yang-Mills equation, this Proposition can be thought of as an analogue of the fact that for a $G_c$-local system $L$ on $E$, we have 
	\[\xymatrixcolsep{0.1 in}
	
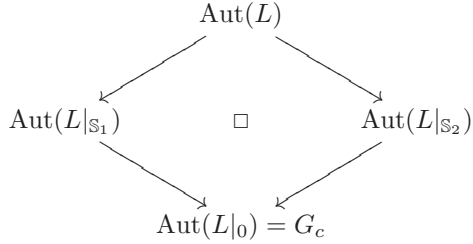
\xymatrix{  & \Aut(L) \ar[ld] \ar[rd] \ar@{}|{\square}[dd]\\
		\Aut(L|_{\S_1}) \ar[rd]  &  &  \Aut({L|_{\S_2}}) \ar[ld] \\
		&   \Aut(L|_0)=G_c
	}\]
	Note that both $\Aut(L|_{\S_i})$ and $G^i_{a_i}$ are of the form $C_G(s)$ for some $s \in T_c $ ( or $T$). In particular, they are connected if $G$ is simply-connected.
\end{rmk}

\section{Character sheaves}

\label{liealgebra}
In this section, we establish results on propagation (Proposition~\ref{probagation}) and untwisting (Proposition~\ref{untwisting}) for character sheaves, and then use them to prove our main theorem.

\begin{no}
Recall $\Phi_\aff$ is by definition of set of affine roots of $G$. A subset $R \subset \Phi_\aff$ is called \textit{admissible} if: 
	\begin{itemize}
		\item  $ \alpha + \beta \in R$, for any $\alpha,\beta \in R$, such that $\alpha+\beta \in \Phi_\aff;$
		\item  the map $R \subset  \Phi_\aff \to \Phi$ is injective, where $\Phi_\aff \to \Phi$ is the natural map via $\alpha_0+n \mapsto \alpha_0$.
	\end{itemize}	 
	
	Let $\mathcal{R}:=\{ R \subset \Phi_\aff \; | \; R \text{ admissible }  \}$. For a connected finite dimensional subgroup of $LG$, we shall use the notation $K \in \mathcal{R}$ if $K \supset T$ and the set of roots of $T$ acting on $Lie(K)$ is admissible. Similar for $\frk \in \mathcal{R}$. We see that $\g_s,\g_a,\g_b, G_a, G^0_s,G^0_b \in \mathcal{R}$.
	
Let $L,K \in \mathcal{R},$ with Lie algebra $\frl,\frk$, and assume that $L \subset K$, denote by:
 \begin{itemize} 
 	\item  $\frl/L$ quotient stack  with respect to the adjoint action $\textup{ad}$. 
 	\item $\frl/'L$ the quotient stack with respect to the gauge transformation $\ad '$.
	\item  $L/L$ quotient stack  with respect to the adjoint action $\Ad$. 
	\item  $L/'L$ quotient stack  with respect to the twisted conjugation action $\Ad '$. 
 	\item $\frc_{\frl}:=\{c \in \frl \;|\;  \ad_g(c)=c, \forall g \in L\}$ the center of $\frl$.
 	\item $\frc'_{\frl}:=\{c \in \frl \;|\;  \ad'_g(c)=c, \forall g \in L\}$ the twisted center of $\frl$.
 	\item $Z(L):=\{c \in L \;|\; \Ad_g(c)=c, \forall g \in L\}$ the center of $L$.
 	\item $Z'(L):=\{c \in L \;|\; \Ad '_g(c)=c, \forall g \in L\}$ the twisted center of $L$.
 	\item $W_L$ the Weyl group of $L$.
 	\item $\chi_\frl :  \frl \to \frt//W_L, \chi_L :  L \to \frt//W_L$, the characteristic polynomials.
 	\item $\chi'_\frl :  \frl \to \frt//'W_L, \chi'_L :  L \to \frt//'W_L$, the twisted characteristic polynomials.
    \item For $x \in \frl$, denote $C_L(x):=\{g \in L \;|\; \ad_g(x)=x\}$ and $C'_L(x):=\{g \in L \;|\; \ad'_g(x)=x\}.$
    \item For $x \in L$, denote $C_L(x):=\{g \in L \;|\; \Ad_g(x)=x\}$ and $C'_L(x):=\{g \in L \;|\; \Ad'_g(x)=x\}.$
    \item $\frl^{\frk\textup{-reg}}:=\{ x \in \frl \;|\;  C_K(x)=C_L(x) \}$, and $\frl^{\frk\textup{-reg}'}:=\{ x \in \frl \;|\;  C'_K(x)=C'_L(x) \}. $
    \item $L^{K\textup{-reg}}:=\{ x \in L \;|\;  C_K(x)=C_L(x) \}$, and $L^{K\textup{-reg}'}:=\{ x \in L \;|\;  C'_K(x)=C'_L(x) \}. $
 \end{itemize}
\end{no}	

The following proposition is easy to check.

	\begin{prop}  \label{untwistsingle}
		Let $L,K \in \mathcal{R}$ and $L \subset K$, then:
		\begin{enumerate}
			\item 
		$\frc'_{\frl} \neq \emptyset$. For any $ c\in \frc'_{\frl}$, translation by $-c$ induces identifications: $\frl/'L \simeq \frl/L, \frc'_\frl \simeq \frc_{\frl}, \chi'_\frl = \chi_\frl, C'_L(x)=C_L(x-c),\frl^{\frk\textup{-reg}'}= \frl^{\frk\textup{-reg}}$.
		\item $Z'(L) \neq \emptyset$. For $c \in Z'(L)$,  multiplication by $c^{-1}$ induces identifications: $L/'L \simeq L/L, Z'(L) \simeq Z(L), \chi'_L= \chi_L, C'_L(x)=C_L(c^{-1}x),L^{K\textup{-reg}'}= L^{K\textup{-reg}}$.
		\end{enumerate}
	\end{prop}
	
	Many results in the following sections are stated for both twisted and untwisted case. We shall only give proof for the untwisted one because the twisted statement follows from Proposition~\ref{untwistsingle}.
	
	We call $L \subset P \subset G$ a \textit{parabolic sequence} if $P$ is a parabolic subgroup of the reductive $G$ and  $L \xrightarrow{\sim} P/U_P$ is an isomorphism, where $U_P$ is the unipotent radical of $P$.

\begin{lem} 
	\label{baserestriction}
	\begin{enumerate}

		\item The maps $f:\frl^{\frk \textup{-reg}}/L \to \frk/K,$ and $f':\frl^{\frk \textup{-reg}'}/'L \to \frk/'K$  are \'etale.
	
		\item The maps $df^*$ and $df'{ }^{*}$ respect nilpotent cones.
	
	\end{enumerate}
\end{lem}
\begin{proof}
	(1) Follows from Proposition~\ref{symplectomorphismetale}.       
	(2) For any $Y \in \mfl^{\frk\textup{-reg}}$, and under the identification by shift symplectic form,  we have $H^0(df_Y^*)|_{\mfl}=Id$.
\end{proof}

\begin{lem}
	\label{levigenericiso}
		Let $L \subset Q \subset K$ be a parabolic sequence, denote by $i:\frl \subset \frq$ the inclusion and $q: \frq \to \frl$ the projection. Let $U \subset \frl^{\frk\textup{-reg}}$ invariant open subset, denote by $ ^Q U :=  q^{-1}(U) \subset \frq $, then  $i: U/L \leftrightarrows \;^Q U/Q : q$ are inverses to each other. 
\end{lem}

\begin{proof}
 The map $i$ is fully-faithful by definition of $ \frl^{\frk\textup{-reg}}$. To prove that $i$ is essentially surjective, we need to show that the unipotent radical $N_0$ of $Q$ acts transitively on $ x + \frn_0$, for any $x \in  \frl^{\frk\textup{-reg}}$. Now let $N_{i}:=[N_{i-1},N_0]$, and $\frn_i$ the corresponding Lie algebra, we have $[\frq,\frn_i] \subset \frn_i$.  For any $x \in \frq$, the adjoint action $N_i$ stabilize  $x + \frn_i$. To see this, let $u \in \frn_i$, then by Baker-Campbell-Hausdorff formula:
$$Ad_{e^u} x = x + [u,x] + \frac{1}{2!}[u,[u,x]] +... \in x + \frn_i. $$ 
Now assume $x \in  \frl^{\frk\textup{-reg}}$, then $Ad_x : \frn_i \xrightarrow{\sim} \frn_i$ is an isomorphism (if $[u,x]=0,$ then $e^u \in C_K(x)=C_L(x) \subset L$, hence $u=0$). We prove by induction (in reverse order) that the action of $N_i$ on $x + \frn_i$ is transitive. For $n>>0$, we have $N_n=1$, and the action is automatically transitive. Assume now that $N_{i+1}$ acts transitively. Suffices to show that every $N_i$ orbit intersects $x + \frn_{i+1}$. Now for any $v \in \frn_i$, take $u=Ad_x^{-1}v \in \frn_i$. Then $Ad_{e^u}(x+v) = x +v+ [u,x]+([u,v] +  \frac{1}{2!}[u,[u,x+v]] +...)  \in x + \frn_{i+1}.$ Hence the action of $N_i$ is also transitive.  
\end{proof}	

\subsection{Propagation and untwisting}
\label{propagation_untwisting}

Let $\mathscr{P}$ be the category consist of object $(K,U)$, where $K \in \mathcal{R}$ reductive, and $U$ a twisted-invariant open subset of $\frk$. with morphisms  
$$\mathscr{P}((G_1,U_1),(G_2,U_2))=
\begin{cases}
\{P | G_1 \subset P \subset G_2 \textup{ a parabolic sequence}\},  \;\;\; \textup{ if } G_1 \subset G_2, U_1 \subset U_2 \cap \g_1^{\g_2\textup{-reg}'};   \\
\emptyset, \;\;\; \textup{ otherwise.}
\end{cases}  
$$
The composition is given by 
$$\mathscr{P}((G_2,U_2),(G_3,U_3)) \times \mathscr{P}((G_1,U_1),(G_2,U_2)) \to \mathscr{P}((G_1,U_1),(G_3,U_3))$$
$$ (P_3, P_2) \mapsto P_3 \circ P_2 $$

The nerve of $\mathscr{P}$ 
is the simplicial set  $\textup{N}(\mathscr{P})$ with 
$\textup{N}(\mathscr{P})_n:$= the tuples $\{ (G_0,G_1,...,G_n);(P_1,P_2,...,P_n);(U_0,U_1,...,U_n) \}$, such that $G_i,P_i \in \mathcal{R}$, and $G_{i-1} \subset P_i \subset G_i$ is a parabolic sequence, and $U_i$ is twisted-invariant open subset in $ \g_{i}^{\g_{i+1}\textup{-reg}'}$. The face and degeneration maps are defined in the obvious way.

Put $P_{[i,j]}:=
\begin{cases}
P_{i+1} \circ ... \circ P_{j}, \text{ for } i < j. \\
  G_j, \text{ for } i=j.
  \end{cases}$

We define three functors $\alpha,\alpha',\beta': \textup{N}(\mathscr{P}) \to \textup{Corr}(\Stk^\circ)$, via 
\begin{enumerate}
\item	$\alpha(G_0,G_1,...,G_n;P_1,P_2,...,P_n;U_0,U_1,...,U_n)_{i,j}= \frp_{[i,j]}/P_{[i,j]}.$
\item $\alpha'(G_0,G_1,...,G_n;P_1,P_2,...,P_n;U_0,U_1,...,U_n)_{i,j}= \frp_{[i,j]}/'P_{[i,j]}.$
\item $\beta'(G_0,G_1,...,G_n;P_1,P_2,...,P_n;U_0,U_1,...,U_n)_{i,j}= U_i/'G_i.$
\end{enumerate}

Note that ${\alpha}$ and $\alpha'$ are well defined by Proposition~\ref{transad}. 

Let $\mathscr{P}_{ret} \subset \mathscr{P}$ be full subcategory consists of objects whose $U$-factors are retractable w.r.t $\N$.

\begin{prop}
	\label{probagation}
	There is a natural isomorphism  $Sh_\N(\overline{\alpha}')|_{\textup{N}(\mathscr{P}_{ret})} \simeq Sh_\N(\overline{\beta'})|_{\textup{N}(\mathscr{P}_{ret})}$.
\end{prop}
\begin{proof}
By Lemma~\ref{levigenericiso}, the natural maps $U_{i}/'G_{i} \to \frp_{[i,j]}/'P_{[i,j]}$
induce cartesian squares:
$$  
\xymatrix{
	U_{i}/'G_{i} \ar@{=}[r] \ar[d]  &  U_{i}/'G_{i}  \ar[d] \\
	\frp_{[i,j+1]}/'P_{[i,j+1]}  \ar[r] & \frp_{[i,j]}/'P_{[i,j]}
	.}
$$
And by Corollary~\ref{natural trans in corr}, 
this defines a natural transformation $\eta': \beta'\Rightarrow \alpha'$, or equivalently  $\overline{\eta}':\overline{\alpha}'  \Rightarrow \overline{\beta}'$ (notation c.f. \ref{conjugatefunctor}).
Hence we get an natural transformation $Sh(\overline{\eta}'): Sh(\overline{\alpha}') \Rightarrow  Sh(\overline{\beta}').$ The arrows in $Sh(\overline{\alpha}')$ and $Sh(\overline{\beta}')$ are given by parabolic restriction and restriction to open substacks, hence preserve nilpotent singular support by Proposition~\ref{parabolicnilpotent} and Lemma~\ref{baserestriction}. Let $Sh_\N(\alpha') \subset Sh(\alpha')$ and $ Sh_\N(\beta') \subset Sh(\beta')$ be the corresponding functor that takes an object $(G_0,U_0)$ to $Sh_\N(\g_0/G_0)$ and $Sh_\N(U_0/G_0)$ respectively.
The natural transformation $Sh(\overline{\eta}')$ induces $Sh_\N(\overline{\eta}'):Sh_\N(\overline{\alpha}') \Rightarrow Sh_\N(\overline{\beta}')$. Then $Sh_\N(\overline{\eta}')|_{\textup{N}(\mathscr{P}_{ret})}$ is a natural isomorphism by the definition of retractable substacks.

\end{proof}

Define $ \textup{un}: \frc'_{\g} \times \g/'G \to \g/G$, via $(c,x) \mapsto x-c$. 
Put $\N_{\textup{fat}}:= \mathbf{0}_{\frc'_{\g}} \times \N \subset T^*(\frc'_{\g} \times \g/'G )$.

\begin{lem}
	\label{untwistequivalence}
$\textup{un}^*, \textup{un}_* $ preserves nilpotent singular support and  there are inverse equivalence of functors:
$$\textup{un}^*:Sh_\N(\g/G) \longleftrightarrow Sh_{\N_{\textup{fat}}}(\frc'_{\g} \times \g/'G ): \textup{un}_* $$
\end{lem}	
\begin{proof}
We can write $\textup{un}$ as the composition 
$ \frc'_{\g} \times \g/'G  \xrightarrow{\textup{Un}} \frc'_{\g} \times \g/G \xrightarrow{\pi} \g/G  $, where $\textup{Un}(c,x)=(c,x-c)$, and $\pi$ is the projection onto the second factor. $\textup{Un}$ is an isomorphism and preserves $\N_\textup{fat}$, hence it induces equivalences 
$$\textup{Un}^*:Sh_{\N_{\textup{fat}}}(\frc'_\g \times \g/G) \longleftrightarrow Sh_{\N_{\textup{fat}}}(\frc'_{\g} \times \g/'G ): \textup{Un}_* $$
We then prove the Lemma by applying Proposition~\ref{contractiblefibers} to $\pi$.

\end{proof}

\begin{prop}
	\label{untwisting}
	There is a natural isomorphism $Sh_\N(\overline{\alpha}') \simeq Sh_\N(\overline{\alpha})$.
\end{prop}

\begin{proof}
	Define $\alpha'_{\textup{fat}}:\textup{N}(\mathscr{P}) \to \Corr(\Stk^\circ)$, via $$\alpha'_{\textup{fat}}(G_0,G_1,...,G_n;P_1,P_2,...,P_n;U_0,U_1,...,U_n)_{i,j}= \frc'_{\g_j} \times \frp_{[i,j]}/'P_{[i,j]}. $$

	By Corollary~\ref{natural trans in corr}, the maps $ \frc'_{\g_{j}} \times \frp_{[i,j]}/' P_{[i,j]} \to \frp_{[i,j]}/ P_{[i,j]},$    
	via $(c,x) \mapsto x-c$ defines a natural transformation $\epsilon: \overline{\alpha}'_{\textup{fat}} \Rightarrow \overline{\alpha}$. By Lemma~\ref{untwistequivalence} The induced map $Sh_\N(\epsilon): Sh_{\N_\textup{fat}}(\overline{\alpha}'_{\textup{fat}}) \Rightarrow Sh_\N(\overline{\alpha})$ is an equivalence. Similarly define $\overline{\alpha}'_{\textup{fat}} \Rightarrow \overline{\alpha}'$ by projecting to the second factor. It induces equivalence $Sh_{\N_\textup{fat}}(\overline{\alpha}'_{\textup{fat}}) \Rightarrow Sh_\N(\overline{\alpha}')$. Hence we get $Sh_\N(\overline{\alpha}') \simeq Sh_\N(\overline{\alpha})$. 
\end{proof}

We also state the group analogue of the above statements. Proofs are similar and we shall omit.

Let $\widetilde{\mathscr{P}}$ be the category consist of object $(K,U)$, where $K \in \mathcal{R}$ reductive, and $U$ a twisted-invariant open subset of $K$, with morphisms 
$$\widetilde{\mathscr{P}}((G_1,U_1),(G_2,U_2))=
\begin{cases}
\{P | G_1 \subset P \subset G_2 \textup{ a parabolic sequence}\},  \;\;\; \textup{ if } G_1 \subset G_2, U_1 \subset U_2 \cap G_1^{G_2\textup{-reg}'};   \\
\emptyset, \;\;\; \textup{ otherwise.}
\end{cases}
$$

Define similarly the functors $\alpha,\alpha',\beta': \textup{N}(\widetilde{\mathscr{P}}) \to \textup{Corr}(\Stk^\circ)$, via 
\begin{enumerate}
	\item	$\widetilde{\alpha}(G_0,G_1,...,G_n;P_1,P_2,...,P_n;U_0,U_1,...,U_n)_{i,j}= P_{[i,j]}/P_{[i,j]}.$
	\item $\widetilde{\alpha}'(G_0,G_1,...,G_n;P_1,P_2,...,P_n;U_0,U_1,...,U_n)_{i,j}= P_{[i,j]}/'P_{[i,j]}.$
	\item $\widetilde{\beta}'(G_0,G_1,...,G_n;P_1,P_2,...,P_n;U_0,U_1,...,U_n)_{i,j}= U_i/'G_i.$
\end{enumerate}

\begin{prop} 
	\label{group retraction}
	There are equivalence of functors:
\begin{enumerate}
	\item $Sh_\N(\overline{\widetilde{\alpha}}')|_{\textup{N}(\widetilde{\mathscr{P}}_{ret})} \simeq Sh_\N(\overline{\widetilde{\beta}}')|_{\textup{N}(\widetilde{\mathscr{P}}_{ret})}.$
	\item 
	$Sh_\N(\overline{\widetilde{\alpha}}) \simeq Sh_\N(\overline{\widetilde{\alpha}}').$
\end{enumerate}

\end{prop}

Next, we give some examples of retractable substacks.

\begin{prop} 
	\label{baseopensetexample} 
	\begin{enumerate}
		\item 	Let $K$ be a reductive group with Lie algebra $\frk$. And $V \subset \frt$ be an $W_K$-invariant open subset. Then $V$ is star-shaped centered at some $c \in \frc_\frk$ if and only if  $U:=\chi_\frk^{-1}(V//W) \subset \frk$ is so. In this situation, $U/K$ is a retractable open substack of $\frk/K$ w.r.t $\N$. 
		\item Let $K \in \mathcal{R}$ and $V \subset \frt$ be an $W_K$-twisted-invariant open subset. Then $V$ is star-shaped centered at some $c \in \frc'_\frk$ if and only if  $U:={\chi'_\frk}^{-1}(V//'W) \subset \frk$ is so. In this situation, $U/'K$ is a retractable open substack of $\frk/'K$ w.r.t $\N$. 
	\end{enumerate}
\end{prop}
\begin{proof}
	Denote by $\pi: \frt \to \frt//W$.	Let $x \in U$, and $t \in V$, such that $\chi(x)=  \pi(t)$, then $\chi(\overline{cx}) = \pi(\overline{ct})$. Hence $V$ is star-shaped at $c$ if and only if $U$ is so.
	Now $U/K$ is retractable by Proposition~\ref{contraction}, since $\mathcal{N}$ is biconical with respect to the $\R^+$ action (which commutes with the $K$ action)  centered at $c$.
\end{proof}	
We also have the Lie group version of last proposition.
\begin{prop} 
	\label{levigroup}
	\begin{enumerate}
		\item Let $K$ be a reductive group,  $c \in Z(K)$, and $V$ be a $W$-invariant open subset of $T$, such that $\tilV:=\Exp^{-1}(c^{-1}V) \subset X_*(T) \otimes \C$ is convex and containing $\t_\R$. Put $U := \chi_K^{-1}(V//W)$, then $U/K$ is a retractable substack of $K/K$ w.r.t $\N$.
		
		\item Let $K \in \mathcal{R}$, $c \in Z'(K)$, and $V$ be a $W$-twisted-invariant open subset of $T$, such that $\tilV:=\Exp^{-1}(c^{-1}V) \subset X_*(T) \otimes \C$ is convex and containing $\t_\R$. Put $U := {\chi'_K}^{-1}(V//'W)$, then $U/'K$ is a retractable substack of $K/'K$ w.r.t $\N$.
	\end{enumerate}
	
\end{prop}
\begin{proof}
	Since $\N \subset T^*(K/K)$ is invariant under translation by central elements, it is suffices to assume $c=1$.
	In Proposition~\ref{sesheaves} (1), we could choose $U_a = {\chi'_a}^{-1}(V_a//'W_a)$ for $V_a$ convex in $\t$, and $V_a \cap \t_\R \neq \emptyset$. Then the restriction map $Sh_\N(K/K) \to Sh_\N( U/K)$ is induced by taking the limit over 
	$\int_{\Delta^{op}} S^\bullet_{\dot{W}_{\aff}}$ of the restrictions $Sh_\N({\chi'_\bfa}^{-1}(V_\bfa//'W_\bfa)/K_\bfa) \to Sh_\N(({\chi'_\bfa}^{-1}(V_\bfa \cap \tilde{V} //'W_\bfa)/K_\bfa)$, which is an isomorphism, because when $V_\bfa$ is non-empty, both ${\chi'_\bfa}^{-1}(V_\bfa//'W_\bfa)$ and ${\chi'_\bfa}^{-1}(V_\bfa \cap \tilde{V}//'W_\bfa) $ are retractable open subset of $\frk_\bfa$ by Proposition~\ref{baseopensetexample}. 
\end{proof}

\begin{cor}
	\label{se_is_retractible}
	\begin{enumerate}
		\item 	The substack $\g^{se}_J/'\g_J \subset \g_J/'\g_J$ is retractible w.r.t $\N$.
		\item 	The substack $G^{se}_J/'G_J \subset G_J/'G_J$ is retractible w.r.t $\N$.
	\end{enumerate}
\end{cor}
\begin{proof}
(1)  $V=St_J$ is star-shaped centered at any $c \in J$. Hence $\g^{se}_J= \chi_{\g_J}'^{-1}(St_J//'W_J)$ satisfies the assumption of Proposition~\ref{baseopensetexample}. \\
(2) Take $V=T^{se}_J$, we see that the subset $\Exp^{-1}(c^{-1} T^{se}_J)=\frt_\R \times i(St_J-c) \subset \frt_\R \times i \frt_\R = X_*(T) \otimes \C$ is convex and containing $\t_\R$, hence $G^{se}_J= \chi_{G_J}'^{-1}(T^{se}_J//'W_J)$ satisfies the assumption of Proposition~\ref{levigroup}.

\end{proof}

\subsection{The main theorem}
\begin{thm}
	\label{groupcharactersheaf}
	There are equivalences of $\oo$-categories:
	$$ \xymatrix{ Sh_\N(G/G) \ar[r]^-{\sim}  &   \lim_{J \in \mathscr F_C} Sh_\N(\g_J/G_J) } $$
		$$ \xymatrix{   Sh_\N(G_E)  \ar[r]^-{\sim}  &   \lim_{J \in \mathscr F_C} Sh_\N(G_J/G_J) } $$
		where in the limit, the arrows are parabolic restriction functors.
\end{thm}
\begin{proof}
	Define $\mathscr{F_C} \to \mathscr{P}^{op}$, via $\{J_0 \to J_1 \} \mapsto \{P^{J_0}_{J_1}: (G_{J_1},\g^{se}_{J_1}) \to (G_{J_0}, \g^{se}_{J_0}) \}$. Then 
	\begin{align*} 
	Sh_\N(G/G) & \simeq \lim_{J \in \mathscr{F}_C} Sh_\N(\frg_J^{se}/'G_J) \qquad \text{by Propsition~\ref{sesheaves} (2)}  \\ 
	& \simeq   \lim_{J \in \mathscr{F}_C} Sh_\N(\frg_J/'G_J)    \qquad \text{by Proposition~\ref{probagation} and Corollary~\ref{se_is_retractible} (1)}  \\
	& \simeq   \lim_{J \in \mathscr{F}_C} Sh_\N(\frg_J/G_J)    \qquad \text{by Propsition~\ref{untwisting}.} 
	\end{align*}
Define	$\mathscr{F}_C \to \widetilde{\mathscr{P}}^{op}$, via $\{J_0 \to J_1 \} \mapsto \{P^{J_0}_{J_1}: (G_{J_1},G^{se}_{J_1}) \to (G_{J_0}, G^{se}_{J_0}) \}$. Then
\begin{align*} 
Sh_\N(G_E) & \simeq \lim_{J \in \mathscr{F}_C} Sh_\N(G_J^{se}/'G_J) \qquad \text{by Propsition~\ref{colimitdiagram2} (2)} \\ 
& \simeq   \lim_{J \in \mathscr{F}_C} Sh_\N(G_J/'G_J)    \qquad \text{by Proposition~\ref{group retraction} (1) and Corollary~\ref{se_is_retractible} (2)} \\
& \simeq   \lim_{J \in \mathscr{F}_C} Sh_\N(G_J/G_J)    \qquad \text{by Propsition~\ref{group retraction} (2).} 
\end{align*}
\end{proof}

\begin{eg} For $G=SL_2$, identifying $X_*(T) \subset \t_\R$ as $\Z \subset \R$, and take the alcove $C= (0,1/2) \subset \t_\R$. We have  \\
$Sh_{\mathcal{N}}(G/G)= 
\lim{ 
\xymatrixcolsep{-0.05in}
\xymatrix{
 &  Sh_{\mathcal{N}}(\mathfrak{t}/T) &  \\
Sh_\mathcal{N}(\mathfrak{g}_0/G_0) \ar[ru]^{\Res_{\mfp^{0}_{(0,1/2)}}} & & Sh_\mathcal{N}(\mathfrak{g}_{1/2}/G_{1/2}) \ar[lu]_{\Res_{\mfp^{1/2}_{(0,1/2)}}}
} } $ \\
If the coefficient $k=\C$, the above diagram can be explicitly calculated as:\\
$Sh_\N(G/G, \C)=\lim 
\xymatrixcolsep{-0.05in}
\xymatrix{
 &  \Vect &  \\
\Vect \oplus {\C}[\mathbb{Z}/2] \modu \ar[ru]^{0 \oplus U} & & \Vect \oplus  {\C}[\mathbb{Z}/2] \modu \ar[lu]_{0 \oplus U}
}   \\
=\Vect \oplus \Vect \oplus {\C}[(\Z/2 * \Z/2)] \modu=\Vect \oplus \Vect \oplus {\C}[W_{\aff}] \modu$\\
Where $U:\C[\Z/2] \modu \rightarrow \Vect$ is the forgetful (restriction) functor.  The two $\Vect$ are generated by two cuspidal sheaves, and ${\C}[W_{\aff}] \modu$ corresponds (see \cite{BZN2}) to sheaves coming from Grothendieck-Springer correspondence: 
$$ \xymatrix{  T/T  & B/B  \ar[r]  \ar[l]   &   G/G }.$$
\end{eg}

\section{Uniformization over $\mathcal{M}_{1,1}$ and parallel transport}

Denote $\mathcal{M}_g$ the moduli stack of genus $g$ Riemann surface. Denote $\Bun_{G,g}$ be the stack classifying pairs $(C,P)$, for $C$ a Riemann surface of genus $g$, and $P$ a $G$-bundles on $C$. We have a natural map $\Bun_{G,g} \to \mathcal{M}_g$. 

Denote by $\pi: \Bun_{B,g} \to \Bun_{G,g}$, and $\N_g:= \pi_*(\mathbf{0}_{\Bun_{B,g}})$. We expect $\N_g$ defines parallel transport of the automorphic category $Sh_\N(\Bun_G(C))$ over $\M_g$. More precisely, for any $D \to \M_g$ \'etale map, put $\Bun_{G,D}:= \Bun_{G,g} \times_{\M_g} D$, we expect:

\begin{conj}
Let $i: \{C\} \hookrightarrow D$ be a point. Then functor
 $ i^*: Sh_{\N_g}(\Bun_{G,D})  \to Sh(\Bun_G(C))$ takes value in $Sh_\N(\Bun_G(C))$. Moreover, if $D$ is contractible, then the  induced functor
 $$
 \xymatrix{
  Sh_{\N_g}(\Bun_{G,D})  \ar[r]^-{\sim} & Sh_\N(\Bun_G(C)).
}
  $$ 
   is an equivalence.
\end{conj}

We shall prove a similar statement for elliptic curves. Denote $\M_{1,1}$ the moduli stack of elliptic curves. Denote by $\Bun_{G}:=\Bun_{G,1,1}$ the stack classifying $(E,P)$, for $E$ an elliptic curve and $P$ a $G$-bundle on $E$. And $\Bun^{0,ss}_{G,1,1} \subset \Bun_{G,1,1}$ the subset of degree $0$ semistable bundles. 
Denote $\N^{ss}_{1,1}:=\pi_*(\mathbf{0}_{\Bun^{0,ss}_{B,1,1}})$, 
	for $\pi:  \Bun^{0,ss}_{B,1,1} \to  \Bun^{0,ss}_{G,1,1}$. Let $\mathcal{H} \to \M_{1,1}$ be the universal cover by upper half plane. Put $\Bun^{0,ss}_{G,\calH}:= \Bun^{0,ss}_{G,1,1} \times_{\M_{1,1}} \calH$, and $\N^{ss}_\calH:=\N^{ss}_{1,1} \times_{\M_{1,1}} \calH$.

Choose  $S \subset \t$  a $W_\aff$-invariant subset, and for each $a \in S$, choose $V_{\R,a} \subset \t^{et}_{\R,a}$ open subset, so that $V_{\R,w(a)}=w(V_{\R,a})$ for all $w \in W_\aff$. For any $\tau \in \calH$, put $V_{\tau,a}:=X_*(T) \otimes \R/\Z \times V_{\R,a}\cdot \tau \subset T$. Put $V_{\calH,a}:=\cup_{\tau \in \calH} V_{\tau,a} \subset T \times \calH =:T_{\calH}$. Denote by $U_{\calH,a} \subset G_a \times \calH=:G_{\calH,a}$ consist of those elements with eigenvalues lying in $V_{\calH,a}//W_a$. For $\bfa=(a_1,...,a_n),$ put $U_\bfa = \cap_{a \in \bfa} U_a$, we have charts $p_\bfa:  U_{\calH,\bfa}/'G_{\calH,a} \to \Bun^{0,ss}_{G,\calH}.$ These charts give the uniformization over $\M_{1,1}$:

\begin{thm}
	\label{universaluniformization}
	There is an isomorphism of stacks:
	 $$  \colim_{\int^{\Delta^{op}} S^\bullet_{\dot{W}_\aff}} U_{\calH,\bfa}/'G_{\calH,a} \simeq \Bun^{0,ss}_{G,\calH}. $$
\end{thm}

\begin{prop}
	\label{universalnilpotent}
	$p_\bfa^*(\N^{ss}_\calH) =   \N  \times \mathbf{0}_\calH $ as substacks of $T^*(U_{\calH,\bfa}/'G_{\calH,\bfa}) \subset  T^*(G_{\bfa}/'G_{\bfa}) \times T^*\calH$
\end{prop}
\begin{proof}
	Pick $V_{\R,a}$ small, so that $V_{\R,a}//W_a \to \t_\R//{W_\aff}$ is an open embedding. Put $B_{\bfa}:= G_{\bfa} \cap LB, B_{\calH,\bfa}:= B_{\bfa} \times \calH$ and $U^B_{\calH,\bfa}= U_{\calH,\bfa} \cap B_{\calH,\bfa} $  We have a commutative diagram
	$$\xymatrix{   B_{\calH,\bfa}/'B_{\calH,\bfa} \ar[d] &  U^B_{\calH,\bfa}/'B_{\calH,\bfa}  \ar[l] \ar[r] \ar[d] &  \Bun^{0,ss}_{B,\calH} \ar[d] \\ 
		 G_{\calH,\bfa}/'G_{\calH,\bfa} \ar[d]  &  U_{\calH,\bfa}/'G_{\calH,\bfa}  \ar[l] \ar[r] \ar[d]  &  \Bun^{0,ss}_{G,\calH} \ar[d]  \\
		     T_{\calH}//W_\bfa         &     V_{\calH,\bfa} //W_\bfa    \ar[l]  \ar[r]      &    T_{\calH}//W_{\aff}
 		}$$
 The horizontal arrows are open embeddings, and all squares are cartesian. Hence the claim follows.
\end{proof}

\begin{prop}
	Let $D \to \M_{1,1}$ be an \'etale map, and $i:\{E\} \hookrightarrow D$ a point. Then the functor
	$ i^*: Sh_{\N^{ss}_{1,1}}(\Bun^{0,ss}_{G,D})  \to Sh(\Bun^{0,ss}_G(E))$ takes value in $Sh_\N(\Bun_G(E))$. Moreover, if $D$ is contractible, then the  induced functor
	$$
	\xymatrix{
	 Sh_{\N^{ss}_{1,1}}(\Bun^{0,ss}_{G,D})   \ar[r]^-{\sim} &  Sh_\N(\Bun^{0,ss}_G(E))
	}
	$$ 
	is an equivalence.
\end{prop}
\begin{proof}
 	Assume $D$ is small and hence can be lifted to $\mathcal{H}$. Then by Theorem~\ref{universaluniformization} and Proposition~\ref{universalnilpotent}. We see that $i^*$ is given by the colimit over $\int^{\Delta^{op}} S^\bullet_{\dot{W}_\aff}$ of the restriction functors
$$ i^*_\bfa: Sh_{ \N  \times \mathbf{0}_D } ( U_{D,\bfa}/'G_{D,\bfa} ) \to Sh(U_\bfa/'G_{\bfa}). $$

The notation involving $D$ is self explanatory. We can choose $V_{\R,a}$ so that $U_{D,\bfa}, U_\bfa$ are retractable w.r.t the nilpotent cones, e.g the ones as in section~\ref{reductivecharts}. We need to show the functor  $i^*_\bfa$ maps isomorphically onto $Sh_\N(U_\bfa/'G_{\bfa})$. To see this, we have a commutative diagram
$$\xymatrix{   Sh_{ \N  \times\mathbf{0}_D } (G_{D,\bfa}/G_{D,\bfa})  \ar[r] \ar[d]^{\simeq}  &    Sh(G_\bfa/G_\bfa) \ar[d]^{\simeq} \\
		    Sh_{ \N  \times \mathbf{0}_D } (G_{D,\bfa}/'G_{D,\bfa})   \ar[r] \ar[d]^{\simeq}    &     Sh(G_\bfa/'G_\bfa)  \ar[d]  \\
		  Sh_{ \N  \times \mathbf{0}_D } (U_{D,\bfa}/'G_{D,\bfa})   \ar[r]^-{ i^*_\bfa}     &     Sh(U_\bfa/'G_\bfa)  
}$$
where the top two vertical arrows are untwisting by some element $a \in \bfa$. And bottom two vertical arrows are restriction to open substacks, where the left arrow is an equivalence since $U_{D,\bfa}$ is retractable w.r.t $\mathbf{0}_D \times \N$. The top horizontal arrow can be identified as  $Sh_{ \N  \times \mathbf{0}_D } (G_{\bfa}/G_{\bfa} \times D )  \to  Sh(G_\bfa/G_\bfa)$, which maps isomorphically onto $Sh_\N(G_\bfa/G_\bfa)$. Hence we see $i^*_\bfa$ also maps isomorphically onto $Sh_\N(U_\bfa/'G_\bfa)$.
\end{proof}

Define $\mathscr{A}^{ss}$ a sheaf of category over $\M_{1,1}$ via $ D \mapsto  Sh_{\N^{ss}_{1,1}}(\Bun^{0,ss}_{G,D}) $.

\begin{cor}
	\label{locallyconstant}
	$\mathscr{A}^{ss}$ is a locally constant sheaf  of $\oo$-categories over $\M_{1,1}$, whose stalk at $E \in \M_{1,1}$ is $Sh_\N(G_E)$.
\end{cor}

\section{Remarks}

In this section, we make some comments about several related topics. 

\subsection{Stratification of compact group} Lusztig stratification is commonly used in the study of geometry of $G/G$. We explain its relation with our charts when restricted to a maximal compact subgroup. \\

Let $G_c$ be a simple and simply-connected compact Lie group, and $T_c \subset G_c$ a maximal compact torus. Choose an alcove $C$ in $X_*(T_c) \otimes \R$, we define open cover $\mathscr C$ of $G_c$ by $\mathscr C=\{G_{c,J}^{se}: J  \in \text{vertices of }C \}$. Identifying the cover with its image, and the cover does not depend on the choice of $T_c$ and $C$, hence the cover is intrinsic associated to $G_c$. Denote by $\mathscr S$ the finest stratification of $G_c$ generated by $\mathscr C$ via taking complement and intersection. Then $\mathscr{C}$ can also be recovered from $\mathscr{S}$: a chart in $\mathscr C$ is the union of all strata whose closure containing a fixed closed stratum in $\mathscr S$. It is clear from the definition that $\mathscr C$ and $\mathscr S$ are conjugation invariant. The stratification $\mathscr S$ can also be described more explicitly:

\begin{prop}
$\mathscr S= \{ ^{G_c} (\Exp (J)): J \in \text{faces of C}  \}$.
\end{prop} 

Now let $G$ be the complexification of $G_c$, then $G$ has a Lusztig stratification $\mathscr L$  by conjugation invariant subvarieties \cite{Lu7}. Let $\mathscr L_c$ denote the induced stratification on $G_c$. Note that even each stratum in $\mathscr L$ is connected, its intersection with $G_c$ may not be connected.

\begin{prop}
Strata in $\mathscr S$ are precisely the connected components of strata in $\mathscr L_c$.
\end{prop}

\begin{eg}
For $G_c= SU(3), G=SL(3,\C)$.  $ \mathscr L=\{\text{(connected components of) } L_\lambda: \lambda \text{ a partition of } 3. \}$, where $L_\lambda = \{g \in G: $\textup{ the semisimple part }$ g^{ss} \text{ has eigenvalue of type } \lambda.  \} $. The stratum $L_{(2,1)}$ is connected. However $L_{(2,1)} \cap G_c = \coprod_{k=0,1,2} S_k$ has three connected component, where $S_k=\{g \in G_c: g \text{ has eigenvalues }  \{a, a, a^{-2} \},  a = e^{2 \pi i \theta/3}, \text{ and } \theta \in (k,k+1) \}$. 
And $\mathscr S=\{\{I\}, \{e^{2\pi i/3} I \}, \{e^{4\pi i/3} I\}, S_0,S_1,S_2, G_c^{reg}\}$
consists of $7$ strata.
\end{eg}

The closed strata in $\mathscr S$ (or $\mathscr L_c$) are precisely the isolated conjugacy classes in $G_c$, they are in bijection with the vertices of $C$. For a vertex $v$, the corresponding conjugacy class is isomorphic to $G/C_G(\Exp(v))$.   For type A, the isolated conjugacy classes are central elements, hence discrete. For other types, there exists isolated class corresponds to a non-special vertex (i.e those such that $C_G(\Exp(v)) \neq G )$, and is therefore positive dimensional.

\subsection{Nonabelian Weierstrass $\wp$-function}

We have understood the nonabelian analog of $E=\C/(\Z+\Z \tau)$ and $E=\C^*/q^\Z$, we also describe the nonabelian analog of view $E$ as of a cubic equation $y^2=4x^3-g_2x-g_3$ birationally. 

Let $E=\C/\Lambda,$ where $\Lambda=\Z + \Z\tau$, recall the $\wp$-function is defined as a $\Lambda$-invariant meromorphic function on $\C$:
$$ \wp(z)= \frac{1}{z^2}+\sum_{\omega \in \Lambda -\{0\}} (\frac{1}{(z + \omega)^2}-\frac{1}{\omega^2})  $$

\begin{defn} The nonabelian $\wp$-function and its derivative is defined as the following meromorphic functions $\mathfrak{gl}_n \to \mathfrak{gl}_n$:
$$ \wp(Z)= \frac{1}{Z^2}+\sum_{\omega \in \Lambda -\{0\}} (\frac{1}{(Z + \omega I)^2}-\frac{1}{(\omega I)^2 })  $$
$$\wp'(Z)= -2\sum_{\omega \in \Lambda} \frac{1}{(Z + \omega I)^3}$$
\end{defn}

We shall consider $G=GL_n$, recall that $GL_{n,E,0}$ is representable by an smooth algebraic variety, and there is a map $p: GL_n \to GL_{n,E,0}$. For $n=1$, $GL_{1,E,0} = \Pic^0(E) \simeq E$, and the map $p$ is identified as $\C^* \to \C^*/e^{2\pi i \tau \Z}=E$.

\begin{thm} 
\begin{enumerate}
\item Under the maps $\xymatrix{\gl_n \ar[r]^-{\Exp}  & GL_n  \ar[r]^-{p} & GL_{n,E,0} } $, the function $\wp(Z)$ and $\wp'(Z)$ descent to a rational function on $GL_{n,E,0}$.
\item The map $\xymatrix{(\wp, \wp'): GL_{n,E,0} \ar[r] & \mathfrak{gl}_n \times     \mathfrak{gl}_n}$, defines a birational isomorphism between $GL_{n,E,0}$ and the subvariety:
$$\{(X,Y) \in \mathfrak{gl}_n \times \mathfrak{gl}_n : [X,Y]=0, \textup{ and } Y^2=4X^3-g_2 X - g_3\},$$ where $g_2=60 \sum_{\omega \in \Lambda -\{0\}}\omega^{-4}$ and $g_3=140 \sum_{\omega \in \Lambda -\{0\}}\omega^{-6}$.
\end{enumerate}

\begin{rmk}
It is more complicated to (partially) compactify the image of the above rational map to give an actual isomorphism. So far we have only use a single chart, and the various other charts may be useful for this purpose.
\end{rmk}

\end{thm}

\subsection{Dependence of restriction functors on parabolic subgroups}
\label{sectiondependence}
The main results here is Proposition~\ref{dependenceparabolic}, which follows immediately from Proposition~\ref{group retraction}. However, we want to proceed to explain the problem in a more natural point of view that compatible with other previous approaches.

In the theory of finite groups,
let $f: A \hookrightarrow B$ be an inclusion of finite groups. A useful tool is the induction/restriction of characters:
\[
\xymatrix{
	\mathbb{C}[A/A] \ar@<.5ex>[r]^{f_*} & \mathbb{C}[B/B] \ar@<.5ex>[l]^{f^*} 
}
\]

Now let $G$ be a reductive Lie group, $L \subset G$ a Levi. It turns out that direct induction/restriction between $G$ and $L$ as in finite group case does not behave well. To correct it, the idea is to use an intermediate parabolic subgroup $P$. And define the parabolic induction/restriction in various context using the diagram 
$$\xymatrix{ L   &  P \ar[r]^p \ar[l]_q & G 
}$$

It's natural to ask  what is the dependence of the resulting restriction/induction on the choice of parabolic subgroups. One heuristic reason the restriction (pull back) along $f: L/L \to G/G$ does not behave well is that the map $f$ is not \'etale (nor smooth). Nevertheless, put $L^r=L^{G\textup{-reg}}$, the map $f|_{L^r} : L^r/L \to G/G$ is \'etale, so when restricted to $L^r$, the ``correct" restriction functor should agree with $f|^*_{L^r}$. And we are done if we could recover the restriction functor from its information on $L^r$. In the setting of perverse character sheaves, this is what happens, essentially as explained in \cite{Gi2}:
\begin{prop} The bottom horizontal arrow is fully faithful and the triangle is naturally commutative.
	$$\xymatrix{ \Perv_\N(G/G) \ar[d]^{\Res_{P}} \ar[dr]^{f^*|_{L^r/L}} \\
		\Perv_\N(L/L) \ar@{^{(}->}[r]  &  \Perv_\N(L^r/L) 
	}$$
	In particular, $\Res_P$ is the unique (up to canonical isomorphism) functor making the diagram commutative. 
\end{prop}

From our perspective, at the level of 1-categories, $L^r$ play the role of a retractable subset. This is possible because $H^0(L) \to H^0(L^r)$ is an isomorphism (both are connected). So two parabolic restrictions are canonical isomorphic since there is a canonical choice of retractable subset, namely the largest one $L^r$. 

However, at the level of $\infty$-categories, $L^r$ is not a retractable subset since the map $H^*(L) \to H^*(L^r)$ is not an isomorphism. (This is more obvious for Lie algebras, where $\mfl$ is contractible while $\mfl^r$ is not.) Nevertheless, we could still choose a retractable subset to get:

\begin{prop} 
	\label{dependenceparabolic}
	Let $P_1,P_2$ be two parabolic subgroup of a reductive group $G$ with the same Levi factor $L$, then there is an isomorphism of functors (depending on a choice of retractable $U$ in $L^{G\textup{-reg}}$ with respect to $\N$):
	$$\xymatrix{ \Res_{P_1} \simeq \Res_{P_2} :  Sh_\N(G/G) \ar[r] & Sh_\N(L/L)  }$$	
\end{prop}	
\begin{proof}
	Let $U \subset L^{G\textup{-reg}}$ be a retractable subset w.r.t $\N$ (we left it to the readers to show the existence of such $U$.)
	Define $i_1,i_2: \Delta^1 \to \widetilde{\mathscr{P}}$ mapping to $P_1: (L,U) \to (G,G)$ and $P_2: (L,U) \to (G,G)$ respectively. By Proposition~\ref{group retraction}, we have $Sh_\N(\alpha(i_1)) \simeq Sh_\N(\beta(i_1)) \simeq Sh_\N(\beta(i_2)) \simeq Sh_\N(\alpha(i_2))$
\end{proof}	
\begin{rmk}
The retractable subset does not always exist for general $H \subset G$ of maximal rank (other than Levi subgroups). A counterexample is that for $H=SL_2 \times SL_2 \subset Sp_4 =G$.
\end{rmk}	

Since the choice of retractable subset is not canonical, it is natural to understand the 
space of choices. This is more clear in the situation of Lie algebras. Let $\mfp_1,\mfp_2$ be two parabolic subalgebra of $\g$ with Levi $\mfl$. The space of choice of a retractable subset of $\mfl$ w.r.t $\g$ is $\mfc_{\mfl}^r:=\frc_\frl \cap \frl^r$. Indeed, for any $x \in \mfc_\mfl^r$, we could choose a small contractible and retractable $U_x$ near $x$, and for $x,y$ close enough and $U_x,U_y$ small enough, we could choose a $U_{x,y}$ containing $U_x, U_y$. Hence we have constructed
\begin{prop}
	Regard $\mfc^r_{\mfl}$ as an $\infty$-groupoid. There is a morphism of $\oo$-groupoids:
	$$ \xymatrix{\mfc^r_{\mfl} \ar[r] & [\Res_{\mfp_1}, \Res_{\mfp_2}], }$$
	where the right hand side is the $\oo$-groupoid of natural isomorphisms between $\Res_{\mfp_1}$ and $\Res_{\mfp_2}$.
\end{prop}

\begin{rmk}
	\begin{enumerate}
		\item Pick any $x \in \mfc_\mfl^r$, it gives Lie algebra version of Proposition~\ref{dependenceparabolic}.
		\item Under Fourier transform, for orbital sheaves, such morphism is constructed in \cite{Mi} using nearby cycle functor of the family given by characteristic polynomial map. Note that the same choice $\mfc_\mfl^r$ is implicit in the proof.
	\end{enumerate}
\end{rmk}

The same approach does not apply to the elliptic situation, we don't know yet if the parabolic restriction for nilpotent sheaves on $G_E$ are isomorphic for different choice of parabolics. The problem is that $E$ is compact, restricting to any (proper) open subset will miss the top cell, hence there is no retractable open subset in this case. This is contrary to the case for $\C^*$ or $\C$ where one could restrict to a smaller open subset where the relavant maps behave well while still retains the topology. We will understand this question for $E$ in a future paper via a different method.


\begin{appendices}

\section{Analytic stacks} \label{analytic stacks}

Denote by $\mathscr{S}$ the $\oo$-category of topological spaces, $\mathscr{C}\textup{plx}$  the site of complex spaces with classical topology, a \textit{prestack} is a presheaf $\mathscr{Y}:\mathscr{C}\textup{plx}^{op} \to \mathscr{S}$, and a \textit{stack} is a sheaf on $\mathscr{C}\textup{plx}$. Denote by 
   $\Stk$ the $\oo$-topoi of stacks.  We view $\Cplx \subset \Stk$ via the Yoneda embedding. A morphism $\mathscr{Y}' \to \mathscr{Y}$ in $\Stk$ is \textit{representable} if for any $X \to \mathscr{Y}, X \in \Cplx$, the fiber product $X \times_{\mathscr{Y}} \mathscr{Y}'$ is representable (by a complex spaces). Let $P$ be a property of morphism in  $\Cplx$, which is stable under base change, we say a representable morphism $\mathscr{Y}' \to \mathscr{Y}$ in $\Stk$ satisfying $P$ if for any $X \to \mathscr{Y}, X \in \Cplx$, the base change map $X \times_{\mathscr{Y}} \mathscr{Y}' \to X$ satisfies property $P$. Such properties include being surjective, \'etale (:=locally biholomorphic), smooth, closed embedding, open embedding, open dense embedding, isomorphism.
\begin{defn}
	A stack $\mathscr{Y}$ is an \textit{analytic stack} if 
	\begin{enumerate}
		\item $\mathscr{Y}(S)$ is a (1-)groupoid, for all $S$.
		\item The diagonal map $\mathscr{Y} \to \mathscr{Y} \times \mathscr{Y}$ is representable.
		\item There exist $Z \in \Cplx$, and $f:Z \to \mathscr{Y}$ (which is automatically representable by (2)) smooth and surjective.
	\end{enumerate}
	We shall call the pair $(Z,f)$ an atlas of $F$.
\end{defn}

\begin{defn}
$f: \mathscr{X} \rightarrow \mathscr{Y}$ is \textit{generically open} if there is $\mathscr{U} \subset \mathscr{X}$ open dense embedding, 
such that $f|_\mathscr{U}: \mathscr{U} \rightarrow \mathscr{Y}$ is an open embedding. 
\end{defn}

\begin{lem}
\label{openembedding}
An \'etale and generically open morphism of analytic stacks is open embedding. 
\end{lem}
\begin{proof}
Suppose $f: X \to Y$ is generically open and \'etale, but not an open embedding. Then there are $x,x' \in X$ mapping to the same point $y$. By \'etaleness, there is a small ball $U$ near $y$, so that $f^{-1}(U) \to U$ contains a double cover. Therefore $f$ can't be generically open. Contradiction.
\end{proof}

\begin{no}
\label{isoclassofgroupoid}
For a groupoid $\mathscr{C}$, let $|\mathscr{C}|$ be the set of isomorphism classes in $\mathscr{C}$. For a stack $\mathscr{X}$, Denote $|\mathscr{X}|:=|\mathscr{X}(\mathbb{C})|$. It has a natural topology coming from (representable) \'etale morphisms.
\end{no}

Let $x \in Ob({\mathscr{X}(\mathbb{C}))}$, we say $f$ is $\textit{\'etale at x}$ if for any base change $Y \rightarrow \mathscr{Y}$, $f': X:=\mathscr{X} \times_\mathscr{Y} Y \rightarrow Y$ is \'etale at any point $x' \in X$ over $x$. By definition, \'etaleness only depends on isomorphism class of $x$, so we can also speak about $f$ being \'etale at $\underline{\smash{x}} \in |\mathscr{X}|$. The locus in $|\mathscr{X}|$ where $f$ is \'etale is open. And $f$ is \'etale if and only if it is \'etale at every $|\mathscr{X}|$.

\subsection{Tangent groupoid and tangent complex}
$T_x \mathscr{X}$ the tangent groupoid at $x$ is defined to be the fiber category of $\sigma: \mathscr{X}(\mathbb{C}[\epsilon]) \rightarrow \mathscr{X}(\mathbb{C})$  over $x$, 
i.e. $Ob(T_x \mathscr{X}):=\{(v,\phi): v \in \mathscr{X}(\mathbb{C} [\epsilon]),  \phi: \sigma(v) \xrightarrow {\sim} x \}$, and morphism are those induces identity on $x$.  There is an action of $Aut(x)$ on $T_x \mathscr{X}$ via $(v,\phi) \mapsto (v, g \circ \phi)$
We have natural map of groupoid $df_x: T_x \mathscr{X} \rightarrow T_{f(x)} \mathscr{Y}$ by post-composition with $f$. The map intertwine the action of $Aut(x)$ and $Aut(f(x)). $ By base change, we have 

\begin{prop}
Assume $\mathscr{X,Y}$ smooth analytic stacks, then $f$ is \'etale at $x$ if and only if $df_x$ is an equivalence. 
\end{prop}

Where an analytic stack is \textit{smooth} if it has an atlas $(Z,f)$, such that $Z$ is a complex manifold.
For smooth analytic stack $\mathscr{X}$, the tangent groupoid $T_x \mathscr{X}$ has a natural structure of category in vector spaces such that the commutativity constraint is trivial. Such datum is equivalent to complex of vector spaces in degree $-1,0$. The assignment is by associate such a category $\mathscr{C}$ to $H^{-1} \rightarrow H^{0}$ where the differential is trivial and $H^0(\mathscr{C})$:=isomorphism classes of objects in $\mathscr{C}$, and $H^{-1}(\mathscr{C})$ := the automorphisms of identity object. Note that both of them have vector space structures. Under this assignment, $Aut(x)$ acts linearly on $H^i(T_x \mathscr{X})$, and $df_x$ induces an linear map between $H^i$ 's and it is an isomorphism if and only if the original functor between groupoids is an equivalence. We have $H^{-1}(T_x \mathscr{X})=Lie(Aut(x))$, and the action of $Aut(x)$ is conjugation.

\begin{eg}
Let $\mathscr{X}=X/G$ be the quotient stack,
$T\mathscr{X}$ is represented by the complex $\mathfrak{g} \otimes_{\mathbb{C}} \mathscr{O}_X \rightarrow TX$,
for $x \in X$, let $\bar{x}$ be the image of $x$ in $\mathscr{X}$, then $T_{\bar{x}}\mathscr{X}$ is quasi-isomorphic to the complex $ \mathfrak{g} \rightarrow T_xX$,
where the arrow is $H \mapsto \frac{d}{dt}|_{t=0}\exp(tH)x$. The action of $Aut(\bar{x})= C_G(x)$  on  $T_{\bar{x}}\mathscr{X} = \{ \mathfrak{g} \xrightarrow{\delta} T_xX \}$ can be identified as:
\begin{itemize}
	\item conjugation on $\mathfrak{g}$;
	\item $g \in Aut(\bar{x}), v=\gamma'(0) \in T_xX,$ for some curve $\gamma(t)$ through $x$. We have $g \cdot v := \frac{d}{dt}|_{t=0}g\gamma(t)$.
\end{itemize}
 Then $\delta$ is an $Aut(\bar{x})$ module map.
Indeed, $\delta(ad_g(H))=\frac{d}{dt}|_{t=0}\exp(t \cdot ad_g(H))x= \frac{d}{dt}|_{t=0}g\exp(tH)g^{-1}x= \frac{d}{dt}|_{t=0}g\exp(tH)x =g \cdot \delta(H),$ where $g^{-1}x=x$ because $g \in C_G(x).$ 
\end{eg}

\subsection{Sheaves on analytic stacks} 
\label{sheaves}
We shall work in the framework of \cite{Lur2,Lur}. Let $k$ be a commutative ring of finite global dimension. Denote by $Sh(X)=Sh(X,k)$ the $\oo$-category of sheaves of $k$-modules on $X$. And for $\mathscr{Y}$ a prestack, define $Sh(\mathscr{Y}):= \lim_{X \in \Cplx_{/\mathscr{Y}}} Sh(X)$. We have the following smooth descent:

\begin{thm}
	\label{cechcover}
	Let $F$ be an analytic stack, $(Z,f)$ an atlas of $\mathscr{Y}$. Denote $Z^\bullet_\mathscr{Y}: \Delta^{op} \to \Stk$ the \v{C}ech nerve of $f$. Then:
	\begin{enumerate}
		\item  $\colim_{\Delta^{op}} Z^\bullet_\mathscr{Y} \simeq \mathscr{Y}$ in $\Stk$.
		\item $\lim_{\Delta} Sh(Z^\bullet_\mathscr{Y}) \simeq Sh(\mathscr{Y}) $ in $\Cat_\oo$.
	\end{enumerate}	
\end{thm}

\begin{thm}
	\label{descent}
	Let $\mathscr{I}$ be an $\oo$-category, and $X: \mathscr{I}^\triangleright \to \mathscr{S}\textup{tk}$ be a functor, such that 
	\begin{enumerate}[label=(\roman*)]
		\item all objects go to analytic stacks;
		\item all arrows go to representable \'etale morphisms;
		\item the induced functor on $\C$-points $X(\C):\mathscr{I}^\triangleright \to \mathscr{S}$ is a colimit diagram.
	\end{enumerate}	
	Then
	\begin{enumerate}
		\item $X$ is a colimit diagram. 
		\item   the induced functor on sheaves $Sh_{X}: \mathscr{I}^{\triangleright,op} \to \Cat_{\infty} $, defined by $I \mapsto {Sh}(X(I))$, is a limit diagram.
	\end{enumerate}
\end{thm}	

\begin{proof}
	By Theorem~\ref{cechcover}, we can assume $X$ takes value in $\Cplx$.  \\
(2) For $I \in \mathscr{I}$, put $f_I: X(I) \to X(*)$. For any $F \in Sh(X(*))$ and $x \in X(*) $ , we have  $(f_{I!}f_i^* F)_x \simeq \coprod_{f_I^{-1}(x)} F_x$ by the \'etaleness of $f_I$.  For fully-faithfulness: by next Lemma, it is equivalent to show that the natural map $\colim_{i \in \mathscr{I}} f_{I!}f_I^*(F) \to F$ is an isomorphism for any $F \in Sh(X(*))$. This can be checked on stalks: for any $x \in X$,  we have  $(f_{I!}f_I^* F)_x \simeq \coprod_{f_I^{-1}(x)} F_x$ by the \'etaleness of $f_I$, hence $(\colim_{I \in \mathscr{I}} f_{I!}f_I^*(F))_x \simeq \colim_{I \in \mathscr{I}} (f_{I!}f_I^*(F))_x \simeq \colim_{I \in \mathscr{I}} \coprod_{f_I^{-1}(x)} F_x \simeq F_x$ because $ \colim_{I \in \mathscr{I}} \coprod_{f_I^{-1}(x)} \simeq \{x\}$ by assumption. For essential surjectivity: let $F_I \in Sh(X(I)),  I \in \mathscr{I}$ be a compatible system, put $F := \colim_{I \in \mathscr{I}} f_{I!} F_I$. Define the $\oo$-category $X_x:=\{(I,x_I)\;|\; x_I \in f_I^{-1}(x)  \}$,  then $F_x =\colim_{I \in \mathscr{I}} \coprod_{x_I \in f_I^{-1}(x)} {F_I}_{x_I} =  \colim_{X_x} {F_I}_{x_I} = \colim_{|X_x|}{F_I}_{x_I} $, the last equality is because for any $(I,x_I) \to (J,x_J)$, the natural map ${F_J}_{x_J} \to {F_I}_{x_I},$ is an isomorphism. Now $|X_x| \simeq *$ by assumption. This induces the isomorphism ${F_I}_{x_I}  \to F_x \simeq  f_I^* F_{x_I}$, hence $f_I^* F \simeq F_I$. \\
	(1) We see that the functor $X$ factors as 
$\xymatrix{\mathscr{I}^{\triangleright}\ar[r]^-{X'} &  Sh(X(*),\mathscr{S}) \ar[r]^-i & \Stk}$, where $X'(I) =f_{I!}f_I^*(X(*))$. The functor $i$ preserves colimit, so it is suffices to show $X'$ is a colimit diagram. This follows the argument in (2), since $X \simeq \colim \; f_{I!}f_I^*(X)$ in $Sh(X(*),\mathscr{S})$.
\end{proof}	

\begin{lem}
Let $C:\mathscr{I}^{\mathscr{\triangleright},op} \to \Cat_\oo$ be a functor, such that for any $a: I \to J$ in $\mathscr{I}$, the left adjoint of $C(a)$ exist. Put $g_I:=C(I \to *),$ and $f_I$ its left adjoint. Then the following are equivalent:
\begin{enumerate}
	\item The induced functor $C(*) \to \lim_{I \in \mathscr{I}} C(I)$ is fully faithful.
	\item The natural transformation $\epsilon: \colim_{i \in \mathscr{I}} f_I  g_I \rightarrow Id$ is an isomorphism.
\end{enumerate}	
\begin{proof} The limit $\oo$-category $\lim_{I \in \mathscr{I}} C(I)$ consists of objects: compatible system $\{c_I \in C(I)\}_{I \in \mathscr{I}}$, and morphisms: $\Hom({c_I},{d_I})=\lim \Hom(c_I,d_I)$. 
	For $x,y \in C(*),$ we have $\Hom(x,y) \to \lim_{i \in \mathscr{I}}  \Hom(g_i(x),g_i(y)) \simeq   \Hom ( \colim_{i \in \mathscr{I}}f_i g_i(x),y)$. Then both (1) and (2) are equivalent to the above arrow being an equivalence.
	\end{proof}
\end{lem}	

\begin{rmk}
	\begin{enumerate}
  \item (Descent for hypercovering)  \'Etale hypercoverings satisfy the assumption of Theorem~\ref{descent}. In fact, a functor $U^{\bullet,\triangleright}: \Delta^{op,\triangleright} \to \Stk $ is an \'etale hypercovering of an analytic stack if and only if it satisfies (i)-(iii) and the additional assumption:
   (iv) the simplicial set $U^{\bullet,\triangleright}(\C)|_{\Delta^{op}}$ is a Kan complex.
	\item Let $\{U_s | s \in S\}$ be a finite open cover of an analytic stack $X$. Let  $\mathscr{P}(S) (\resp \mathscr{P}^0(S))$ be the category of (\resp nonempty) subset of $S$.
		Then the functor $\bU: \mathscr{P}(S)^{op}=\mathscr{P}^0(S)^{op,\triangleright} \to \Stk,$ via $I \to \cap_{s \in I} U_s$ satisfies assumption.
	\end{enumerate}
\end{rmk}

\section{Semisimple and semistable bundles} \label{ssbundles}

In this section, we collected some basic facts about $G$-bundles on elliptic curves. Some references are \cite{BG,FM1,FM2}.
\begin{defn}
 Let $\mathcal{P}$ be a $G$ bundle on a compact Riemann surface $C$. \\
$\mathcal{P}$  is of \textit{degree 0} if it lies in the neutral component $\Bun^0_G(C)$ of $\Bun_G(C)$\\
$\mathcal{P}$  is \textit{semisimple} if $\mathcal{P}$ has reduction to a maximal torus $T$.\\
Let $C=E$ be an elliptic curve.\\
$\mathcal{P}$  is \textit{semistable} if the associated adjoint bundle $\mathfrak{g}_{\mathcal{P}}$ is semistable.
\end{defn}
It's easy to see that semisimple semistable $G$ bundles of degree 0 are exactly those in the image of the map $\Bun^0_T(C) \rightarrow \Bun^0_G(C)$.
Let $G_E:=\Bun^{0,ss}_G(E)$ be the stack of degree 0 semistable $G$ bundles on $E$. And $\fre_E$ be the coarse moduli space of degree 0 semistable $G$ bundles. There is a (non-representable) maps between stacks $\chi_E: G_E \rightarrow \fre_E$, by taking a bundle to its S-equivalence class.
We have $\fre_E \simeq  (\Pic^0(E) \otimes X_*(T))//W$. 
 There is a commutative diagram:
$$ \xymatrix{ T_E/W \ar[r]^i \ar[rd]^{\pi}& G_E \ar[d]^{\chi_E} \\
	&      \fre_E }$$
where $\pi: T_E/W  \to \fre_E$ by forgetting the automorphisms.
Let $\fre_E^{reg} =(\Pic^0(E) \otimes X_*(T))^{reg}//W \subset \fre_E,$ 
where $(\Pic^0(E) \otimes X_*(T))^{reg} \subset (\Pic^0(E) \otimes X_*(T))$ is the open dense locus where the $W$ action is free.
 Take $G^{reg}_E:=\chi^{-1}_E(\fre_E^{reg})$ and $T_E^{reg}/W:=\pi^{-1}(\fre_E^{reg})$, then we have an isomorphism of analytic stacks $T_E^{reg}/W \xrightarrow{\sim} G^{reg}_E.$ 

\begin{prop}
	\label{semisimpleabundant}
	For any $c \in \fre_E$, the subset of semisimple bundles is abundant in $|\chi_E^{-1}(c)|$.
\end{prop}
\begin{proof} Take any $\mathcal{P} \in |\chi_E^{-1}(c)|$, by \cite{BZN2}, there exist a $B$-reduction $\mathcal{P}_B$ of degree 0. Now by \cite[Proposition 3.5]{Ram96}, there exist a family of $G$ bundles on $E$, whose generic fiber is $\mathcal{P}$ and special fiber is $\mathcal{P}^{ss}:=\mathcal{P}_B \times_B T \times_T G$. Hence $\mathcal{P}^{ss} \in \overline{\{\mathcal{P}\}}$, note also that $\mathcal{P}^{ss} \in |\chi^{-1}_E(c)|$.
\end{proof}

Let $\Bun^{0,ss}_{G,x}(E)$ be the moduli stack of degree 0 semistable $G$-bundles on $E$ with trivialization at $x \in E$. We have:

\begin{prop} \label{representable}
	$\Bun^{0,ss}_{G,x}(E)$ is representable by a complex manifold.
\end{prop}
\begin{proof} Consider $\Bun^{0,ss}_{G,x}(E)^{\textup{alg}}$ the smooth algebraic stack of degree $0$ semistable $G$ bundles with trivialization at $x$. For any $G$-bundles $\mathcal{P} $, the natural map $Aut(\mathcal{P}) \rightarrow Aut(\mathcal{P}_x)$ is injective, by Atiyah's classification of vector bundles over elliptic curves. Hence $\Bun^{0,ss}_{G,x}(E)^{\textup{alg}}$ takes values in $\Set \subset \mathscr{S}$, therefore it is representable by a smooth algebraic space. And $\Bun^{0,ss}_{G,x}(E)$ is analytification of  $\Bun^{0,ss}_{G,x}(E)^{\textup{alg}}$, hence it is representable by a complex manifold. 
\end{proof}

\end{appendices}

\newpage
\bibliographystyle{alpha}
\bibliography{paper}

\end{document}